\documentclass[12pt]{article}

\usepackage{amsmath,amssymb,amsthm,graphics,amsfonts,graphicx,float}
\usepackage{calc, eurosym}
\usepackage{lineno}
\usepackage{tikz}
\usepackage{ifthen}
\usepackage{ifpdf}
\usepackage[breaklinks, colorlinks=true, citecolor=blue]{hyperref}
\usepackage{color}
\usepackage{enumerate}
\usepackage{caption}
\usepackage{subcaption}
\usetikzlibrary{arrows, decorations.pathmorphing, backgrounds, positioning, fit, petri, shapes}
\usepackage{multicol}
\usepackage{breakurl}
\hypersetup{breaklinks=true}
\usepackage{minitoc}
\usepackage{epstopdf}

\usepackage{minitoc}

\newtheorem{theorem}{Theorem}[section]
\newtheorem{lemma}{Lemma}[section]

\newtheorem{remark}{Remark}[section]
\def\today{\number\day \space\ifcase\month\or
  January\or February\or March\or April\or May\or June\or
  July\or August\or September\or October\or November\or December\fi
  \space\number\year}

\begin{document}
\begin{center}  
{\Large \bf Backward Bifurcation and Control in Transmission Dynamics of Arboviral Diseases}
\end{center}

\smallskip
\begin{center}
{\small \textsc{Hamadjam Abboubakar}\footnote{Corresponding author. Present Adress:  UIT--Department of Computer Science, P.O. Box 455,  Ngaoundere, Cameroon, email: abboubakarhamadjam@yahoo.fr, Tel. (+237) 694 52 31 11}$^{,2}$, 
\textsc{Jean Claude Kamgang}$^{2}$, \textsc{Daniel Tieudjo}$^{2}$}
\end{center}
\begin{center} {\small \sl $^{1}$ The University of Ngaoundere, UIT, Laboratoire d'Analyse, Simulation et Essai, P.O. Box 455 Ngaoundere, Cameroon}\\
{\small \sl $^{2}$ The University of Ngaoundere, ENSAI, Laboratoire de Math\'ematiques Exp\'erimentales, P. O. Box 455 Ngaoundere, Cameroon
}
\end{center}

\bigskip
{\small \centerline{\bf Abstract} 
In this paper, we derive and analyze a compartmental model for the control of arboviral diseases which takes into account an imperfect vaccine combined with individual protection and some vector control strategies already studied in the literature. After the formulation of the model, a qualitative study based on stability analysis and bifurcation theory reveals that the phenomenon of backward bifurcation may occur. The stable disease-free equilibrium of the model coexists with a stable endemic equilibrium when the reproduction number, $R_0$, is less than unity. Using Lyapunov function theory, we prove that the trivial equilibrium is globally asymptotically stable; When the disease--induced death is not considered, or/and, when the standard incidence is replaced by the mass action incidence, the backward bifurcation does not occur. Under a certain condition, we establish the global asymptotic stability of the disease--free equilibrium of the full model. Through sensitivity analysis, we determine the relative importance of model parameters for disease transmission. Numerical simulations show that the combination of several control mechanisms would significantly reduce  the spread of the disease, if we maintain the level of each control high,  and this, over a long period.
}
\medskip

\medskip



\section{Introduction}
\label{sec:IntAr3opc}
Arboviral diseases are affections transmitted by hematophagous arthropods. There are currently 534 viruses registered in the International Catalog of Arboviruses and 25\% of them have caused documented illness in human populations \cite{ch,ka,gu}. Examples of those kinds of diseases are dengue, yellow fever, Saint Louis fever, encephalitis, West Nile fever and chikungunya. A wide range of arbovirus diseases are transmitted by mosquito bites and constitute a public health emergency of international concern. According to WHO, dengue, caused by any of four closely-related virus serotypes (DEN-1-4) of the genus Flavivirus, causes 50--100 million infections worldwide every year, and the majority of patients worldwide are children aged 9 to 16 years \cite{Sanofi2013,WHO2006,WHO2009}. The dynamics of arboviral diseases like dengue or chikungunya are influenced by many factors such as human and mosquito behavior, the virus itself, as well as the environment which  directly or indirectly affects all the present mechanisms of control. 

For all  mentioned diseases, only yellow fever has a licensed vaccine. Nevertheless, considerable efforts are made to obtain the vaccines for other diseases. In the case of Dengue for example, the scientists of french laboratory SANOFI have conducted different tries in Latin America and Asia. Thus, a tetravalent vaccine could be quickly set up in the coming months. But in any case, it is clear that this vaccine will be imperfect. However, the tries in Latin America have shown that vaccine efficacy was 64.7\%. Serotype-specific vaccine efficacy was 50.3\% for serotype 1, 42.3\% for serotype 2, 74.0\% for serotype 3, and 77.7\% for serotype 4~\cite{Villar2015}. The tries in Asia have shown that efficacy was 30.2\%, and differed by serotype~\cite{Arunee2012}.

Host-vector models for arboviral diseases transmission were proposed in \cite{Aldila2013,Antonio2001,caga,coutinho2006,cretal,Derouich2006,duch,esva98,esva99,fevh,gaguab,HelenaSofiaRodrigues2014,BlaynehaetAl,maya,moaaca,moaaHee2012,poetal} with the focus on the construction of the basic reproductive ratio and related stability analysis of the disease free and endemic equilibria. Some of these works in the literature focus on modeling the spread of arboviral diseases and its control using some mechanism of control like imperfect vaccines \cite{gaguab,HelenaSofiaRodrigues2014} and other control tools like individual protection and vector control strategy \cite{Aldila2013,Antonio2001,duch,BlaynehaetAl,moaaca,moaaHee2012}.

In \cite{duch}, Dumont and Chiroleu proposed a compartmental model to study the impact of vector control methods used to contain or stop the epidemic of Chikungunya of 2006 in R\'eunion island.  Moulay et {\it al.} \cite{moaaca} study an optimal control based on protection and vector control strategies to fight against Chikungunya. In \cite{HelenaSofiaRodrigues2014}, Rodrigues et {\it al.} simulate an hypothetical vaccine as an extra protection to the human population against epidemics of Dengue, using the optimal control. In these models \cite{duch,HelenaSofiaRodrigues2014,moaaca},
\begin{itemize}
\item[(i)] the population is constant, 
\item[(ii)] the   disease-induced death in humans is not considered,
\item[(iii)] the complete stage progression of development of vectors is not considered,
\item[(iv)] none of the above  mentioned models takes into account the combination of the mechanisms of control already studied in the literature, such as vaccination, individual protection and vector control strategies (destruction of breeding site, eggs and larvae reduction).
\end{itemize} 

The aim of this work is to propose and study a arboviral disease control model which takes into account human immigration, disease--induced mortality in human communities, the complete stage structured model for vectors and a combination of human vaccination, individual protection and vector control strategies to fight against the spread of these kind of diseases. 

We start with the formulation of the model, which is an extension of the previous model study in \cite{AbbouEtAl2015}. We include the complete stage progression of development of vectors, the waning vaccine, and four other continuous controls (individual protection, using adulticides, the mechanical control, Eggs and larvae reduction). We compute the net reproductive number $\mathcal{N}$, as well as the basic reproduction number, $R_0$, and investigate the existence and stability of equilibria. We prove that the trivial equilibrium is globally asymptotically stable whenever $\mathcal{N}<1$. When $\mathcal{N}>1$ and $R_0<1$, we prove that the system exhibit the backward bifurcation phenomenon. The implication of this occurrence is that the classical epidemiological requirement for effective eradication of the disease, $R_0<1$, is no longer sufficient, even though necessary. However considering two situations: the model without vaccination and the model with mass incidence rates, we prove that the disease--induced death and the standard incidence functions, respectively, are the main causes of the occurrence of backward bifurcation.
We found that the disease--free equilibrium is globally asymptotically stable under certain condition. Through local and global sensitivity analysis, we determine the relative importance  parameters of the model on the disease transmission. By using the pulse control technique in numerical simulations, we evaluate the impact of different controls combinations on the decrease of the spread of these diseases.

The paper is organized as follows. In Section \ref{sec2ar3opc} we present the transmission model and in Section \ref{sec3ar3opc} we carry out some analysis by determining important thresholds such as the net reproductive number $\mathcal{N}$ and the basic reproduction number $R_0$, and different equilibria of the model. We then demonstrate the  stability of equilibria and carry out bifurcation analysis. In section \ref{SensitivityAr3}, both local and global sensitivity analysis are used to assess the important parameters in the spread of the diseases. Section \ref{secNumear3opc} is devoted to numerical simulations and discussion. A conclusion rounds up the paper.

\section{The formulation of the model}
\label{sec2ar3opc}
The model we propose here is based on the modelling approach given in \cite{AbbouEtAl2015,duch,esva98,esva99,fevh,gaguab,moaaca,moaaHee2012}. It is assumed that the human and vector populations are divided into compartments described by time--dependent state variables. The compartments in which the populations are divided are the following ones:

--For humans, we consider susceptible (denoted by $S_h$), vaccinated ($V_h$), exposed ($E_h$), infectious ($I_h$) and resistant or immune ($R_h$); So that, $N_h=S_h+V_h+E_h+I_h+R_h$. Following Garba et {\emph{al.}}~\cite{gaguab} and Rodrigues et {\emph{al.}}~\cite{HelenaSofiaRodrigues2014}, we assume that the immunity, obtained by the vaccination process, is temporary. So, the immunity has the waning rate $\omega$. The recruitment in human population is at the constant rate $\Lambda_h$, and newly recruited individuals enter the susceptible compartment $S_h$. Are concerned by recruitment people that are totally naive from the disease, and immune people whose immunity is lost. Each individual human compartment  goes out from the dynamics at natural mortality rates $\mu_h$. The human susceptible population is decreased following infection, which can be acquired via effective contact with an exposed or infectious vector at a rate $\lambda_h=\dfrac{a\beta_{hv}(\eta_vE_v+I_v)}{N_h}$ \cite{gaguab} where $a$ is the biting rate per susceptible vector, $\beta_{hv}$ is the transmission probability from an infected vector ($E_v$ or $I_v$) to a susceptible human ($S_h$). The probability that a vector chooses a particular human or other source of blood to bite can be assumed as $\dfrac{1}{N_h}$. Thus, a human receives in average $a\dfrac{N_v}{N_h}$ bites per unit of times. Then, the infection rate per susceptible human is given $a\beta_{hv}\dfrac{N_v}{N_h}\dfrac{(\eta_vE_v+I_v)}{N_v}$. In expression of $\lambda_h$, the modification parameter $0 <\eta_v< 1$ accounts for the assumed reduction in transmissibility of exposed mosquitoes relative to infectious mosquitoes \cite{gaguab} (see the references therein for the specific sources). Latent humans ($E_h$) become infectious ($I_h$) at rate $\gamma_h$. Infectious humans recover at a constant rate, $\sigma$ or dies  as consequence of infection, at a disease-induced death rate $\delta$. Immune humans retain their immunity for life. 

-- Following \cite{moaaca}, the stage structured model is used to describe the vector population dynamics, which consists of three main stages: embryonic (E), larvae (L) and pupae (P). Even if
eggs (E) and immature stages (L and P) are both aquatic, it is important to dissociate them because, for optimal control point of view, drying the breeding sites does not kill eggs, but only larvae and pupae. Moreover, chemical interventions on the breeding sites has impact on the larvae population (as such as pupae), but not on the eggs \cite{moaaca}. The number of laid eggs is assumed proportional to the number of females. The system of stage structured model of aquatic phase development of vector is given by (see \cite{moaaca} for details) 
\begin{equation}
\label{AR3}
\left\lbrace  \begin{array}{ll}
\dot{E}&=\mu_b\left(1-\dfrac{E}{\Gamma_{E}} \right)(S_v+E_v+I_v)-(s+\mu_E)E\\
\dot{L}&=sE\left(1-\dfrac{L}{\Gamma_{L}} \right)-(l+\mu_L)L\\
\dot{P}&=lL-(\theta+\mu_P)P\\
\end{array}\right.
\end{equation}
Unlike the authors of \cite{moaaca}, we take into account the pupal stage in the development of the vector. This is justified by the fact that they do not feed during this transitional stage of development, as they transform from larvae to adults. So, the control mechanisms can not be applied to them.

A rate, $\theta$, of pupae become female Adults. Each individual vector compartment  goes out from the dynamics at natural mortality rates $\mu_v$. The vector susceptible population is decreased following infection, which can be acquired via effective contact with an exposed or infectious human at a rate $\lambda_v=\dfrac{a\beta_{vh}(\eta_hE_h+I_h)}{N_h}$ \cite{gaguab} where $\beta_{hv}$ is the transmission probability from an infected human ($E_h$ or $I_h$) to a susceptible vector ($S_v$). Latent vectors ($E_v$) become infectious ($I_v$) at rate $\gamma_v$. The vector population does not have an immune class, since it is assumed that their infectious period ends with their death \cite{esva99}.
 
Then, we add new terms in the model to assess the different control tools studied: 
\begin{itemize}
\item[(i)] $\alpha_1$ represents the efforts made to protect human from mosquitoes bites. It mainly consists to the use of mosquito nets or wearing appropiate clothing \cite{moaaHee2012}. Thus we modify the infection term as follows:
\begin{equation}
\label{infetRatear3}
\lambda^{c}_h=(1-\alpha_1)\lambda_h,\quad \text{and}\quad \lambda^{c}_v=(1-\alpha_1)\lambda_v,\text{with}\quad 0\leq \alpha_1<1;
\end{equation} 
\item[(ii)] $\eta_1$ and $\eta_2$ are eggs and larvae mortality rates induced by chemical intervention respectively, 
\item[(iii)] $c_m$ is the additional mortality rate due to the adulticide,
\item[(iv)] $\alpha_2$ is the parameter associated with the efficacy of the mechanical control.
\end{itemize}
The above assumptions lead to the following non-linear system of ordinary differential equations
\begin{equation}
\label{AR3}
\left\lbrace  \begin{array}{ll}
\dot{S}_h&=\Lambda_h+\omega V_h-\left( \lambda^{c}_h+\xi+\mu_h\right)S_h\\
\dot{V}_h&=\xi S_h-\left[ (1-\epsilon)\lambda^{c}_h+\omega+\mu_h\right] V_h\\
\dot{E}_h&=\lambda^{c}_h\left[S_h+(1-\epsilon)V_h\right] -(\mu_h+\gamma_h)E_h\\
\dot{I}_h&=\gamma_hE_h-(\mu_h+\delta+\sigma)I_h\\
\dot{R}_h&=\sigma I_h-\mu_hR_h\\
\dot{S}_v&=\theta P-\lambda^{c}_vS_v-(\mu_v+c_m)S_v\\
\dot{E}_v&=\lambda^{c}_vS_v-(\mu_v+\gamma_v+c_m)E_v\\
\dot{I}_v&=\gamma_vE_v-(\mu_v+c_m)I_v\\
\dot{E}&=\mu_b\left(1-\dfrac{E}{\alpha_2\Gamma_{E}} \right)(S_v+E_v+I_v)-(s+\mu_E+\eta_1)E\\
\dot{L}&=sE\left(1-\dfrac{L}{\alpha_2\Gamma_{L}} \right)-(l+\mu_L+\eta_2)L\\
\dot{P}&=lL-(\theta+\mu_P)P\\
\end{array}\right.
\end{equation}
It is important to note that no intervention measure is performed to kill the pupae for two reasons: the first reason is the fact that at this stage, no food is absorbed by the insect, so it is impossible to make her ingest a toxic substance; the second reason is the fact that products soluble in water deposits by contact are not selective mosquito nymphs and act on all the wildlife of the cottage.

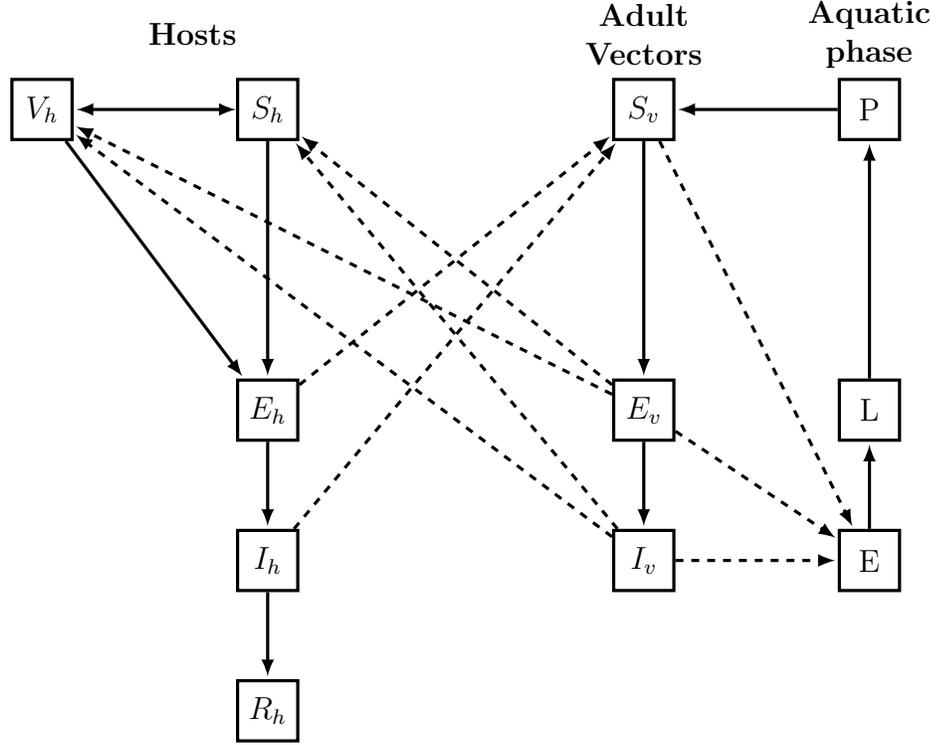
\begin{figure}[h!]
\begin{center}
\begin{tikzpicture}
[host/.style={rectangle,draw=black!0,fill=black!0,very thick,
inner sep=3pt,minimum size=8mm, font=\bf, align=center},
compartment/.style={rectangle,draw=black!100,fill=black!0,very thick,
inner sep=3pt,minimum size=8mm},
intervention/.style={align=center,red},
bend angle = 15,
post/.style={->,shorten >=1pt,>=latex, very thick},
pre/.style={<-,shorten >=1pt,>=latex, very thick},
both/.style={<->,shorten >=1pt,>=latex, very thick},
predashed/.style={dashed,<-,shorten >=1pt,>=latex, very thick},
postdashed/.style={dashed,->,shorten >=1pt,>=latex, very thick},
]

\node  [host]  (hosts)        at (-1,1)        {Hosts};
\node  [host]  (hosts)        at (5,1)        {Adult\\ Vectors};
\node  [host]  (hosts)        at (8,1)        {Aquatic\\ phase};
\node [compartment] (noeud1) at (0,0) {$S_h$};
\node [compartment] (noeud2) at (-3,0) {$V_h$};
\node [compartment] (noeud3) at (0,-4) {$E_h$};
\node [compartment] (noeud4) at (0,-6) {$I_h$};
\node [compartment] (noeud5) at (0,-8) {$R_h$};

\node [compartment] (noeud6) at (5,0) {$S_v$};
\node [compartment] (noeud7) at (5,-4) {$E_v$};
\node [compartment] (noeud8) at (5,-6) {$I_v$};

\node [compartment] (noeud9) at (8,0) {P};
\node [compartment] (noeud10) at (8,-4) {L};
\node [compartment] (noeud11) at (8,-6) {E};
\draw[both] (noeud1)-- (noeud2);
\draw[post] (noeud1)-- (noeud3);
\draw[post] (noeud2)-- (noeud3);
\draw[post] (noeud3)-- (noeud4);
\draw[post] (noeud4)-- (noeud5);
\draw[post] (noeud6)-- (noeud7);
\draw[post] (noeud7)-- (noeud8);
\draw[postdashed] (noeud3)-- (noeud6);
\draw[postdashed] (noeud4)-- (noeud6);
\draw[postdashed] (noeud7)-- (noeud1);
\draw[postdashed] (noeud8)-- (noeud1);
\draw[postdashed] (noeud7)-- (noeud2);
\draw[postdashed] (noeud8)-- (noeud2);
\draw[post] (noeud11)-- (noeud10);
\draw[post] (noeud10)-- (noeud9);
\draw[post] (noeud9)-- (noeud6);
\draw[postdashed] (noeud6)-- (noeud11);
\draw[postdashed] (noeud7)-- (noeud11);
\draw[postdashed] (noeud8)-- (noeud11);
\end{tikzpicture}
\end{center}
\caption{A compartment model for vector-borne disease with waning vaccine and mosquito aquatic development phase.} \label{ar3ModelOPC}
\end{figure}

The description of state variables and parameters of model \eqref{AR3} are given in Tables \ref{opctab1} and \ref{opctab2}--\ref{opctab21}.
\begin{table}[t]
\caption{The state variables of model (\ref{AR3}).}
\label{opctab1}
\begin{center}
\begin{tabular}{rlrlrl}
\hline
& Humans  &  &Vectors&&  \\
\hline
$S_{h}$:& Susceptible & $E$:& Eggs\\
$V_{h}$:& Vaccines      &$L$:& Larvae\\
$E_{h}$:& Infected in latent stage  &$P$:& Pupae\\
$I_{h}$:& Infectious   &$S_{v}$:& Susceptible\\
$R_{h}$:& Resistant (immune) &$E_v$& Infected in latent stage\\
&&$I_v$&Infectious\\
\hline
\end{tabular}
\end{center}
\end{table}
\begin{table}
\begin{center}
\caption{Description and baseline values/range of parameters of model (\ref{AR3}).\label{opctab2}} 
\begin{tabular}{llll}
\hline\noalign{\smallskip}
Parameters& Description&Baseline value/range &Sources\\
\noalign{\smallskip}\hline\noalign{\smallskip}
$\Lambda_h$ & Recruitment rate of humans &  2.5 day$^{-1}$&\cite{gaguab}  \\
$\mu_h$ & Natural mortality rate & $\frac{day^{-1}}{(67\times 365)}$&\cite{gaguab}\\
& in humans & &\\
$\xi$&Vaccine coverage&Variable\\
$\omega$&Vaccine waning rate &Variable\\
$\epsilon$&The vaccine efficacy&$0.61$&\cite{Figaro}\\
$a$ & Average number of bites& 1 day$^{-1}$&\cite{Aldila2013,gaguab}\\
$\beta_{hv}$ & Probability of transmission of &0.1, 0.75 day$^{-1}$&\cite{Aldila2013,gaguab}\\
&infection from an infectious human&&\\
&to a susceptible vector&&\\
$\gamma_h$ & Progression rate from $E_h$ to $I_h$ &$\left[\frac{1}{15},\frac{1}{3}\right]$ day$^{-1}$&\cite{duch,Scott2010}\\
$\delta$ & Disease--induced death rate& 10$^{-3}$ day$^{-1}$&\cite{gaguab}\\
$\sigma$ & Recovery rate for humans & 0.1428 day$^{-1}$&\cite{Aldila2013,gaguab}\\
$\eta_h$,$\eta_v$& Modifications parameter&$\left[0,1\right)$&\cite{gaguab}\\
$\mu_v$ & Natural mortality rate of vectors &$\left[\frac{1}{30},\frac{1}{14}\right]$ day$^{-1}$&\cite{Aldila2013,gaguab}\\
$\gamma_v$ & Progression rate from $E_v$ to $I_v$ &$\left[\frac{1}{21},\frac{1}{2}\right]day^{-1}$&\cite{duch,Scott2010} \\
$\beta_{vh}$ &Probability of transmission of &0.1, 0.75 day$^{-1}$&\cite{Aldila2013,gaguab}\\
&infection from an infectious vector&&\\
&to a susceptible human&&\\
$\theta$& Maturation rate from pupae & 0.08 day$^{-1}$&\cite{duch,moaaca,moaaHee2012}\\
& to adult&&\\
$\mu_b$&Number of eggs at each deposit&6 day$^{-1}$&\cite{duch,moaaca,moaaHee2012}\\
$\Gamma_{E}$&Carrying capacity for eggs&$10^3,10^6$&\cite{Aldila2013,moaaca}\\
$\Gamma_{L}$&Carrying capacity for larvae&$5\times 10^2,5\times10^5$&\cite{Aldila2013,moaaca}\\
$\mu_{E}$&Eggs death rate  &0.2 or 0.4 &\cite{moaaHee2012}\\
$\mu_{L}$&Larvae death rate&0.2 or 0.4&\cite{moaaHee2012}\\
$\mu_{P}$&Pupae death rate&$0.4$&\\
\noalign{\smallskip}\hline
\end{tabular}
\end{center}
\end{table}

\begin{table}
\begin{center}
\caption{Description and baseline values/range of parameters of model (\ref{AR3}).\label{opctab21}} 
\begin{tabular}{llll}
\hline\noalign{\smallskip}
Parameters& Description&Baseline value/range &Sources\\
\noalign{\smallskip}\hline\noalign{\smallskip}
$s$&Transfer rate from eggs to larvae&0.7 day$^{-1}$&\cite{moaaHee2012}\\
$l$&Transfer rate from larvae to pupae&0.5 day$^{-1}$&\cite{moaaca,http}\\
$\eta_{1},\eta_{2}$&Eggs and larvae mortality rates&0.001,0.3&\cite{moaaHee2012}\\
&induced by chemical intervention&&\\
$\alpha_1$&Human protection rate&$\left[0,1\right)$&\\
$\alpha_2$&Efficacy of the mechanical control&$\left(0,1\right]$&\cite{duch}\\
$c_{m}$&Adulticide killing rate&[0,0.8]&\cite{duch}\\
\noalign{\smallskip}\hline
\end{tabular}
\end{center}
\end{table}

\subsection{Well posedness of the model }
We now show that the system \eqref{AR3} is mathematically well defined and biologically feasible.
We write
\begin{equation}
\label{intervaria}
\begin{array}{l}
k_1:=\xi+\mu_h;\,\,k_2:=\omega+\mu_h;\,\,\,k_3:=\mu_h+\gamma_h;\,\,\,
k_4:=\mu_h+\delta+\sigma;\\
k_5:=s+\mu_E+\eta_1;\,\,\,k_6:=l+\mu_L+\eta_2;\,\,\,k_7:=\theta+\mu_P;\,\,\,k_8:=\mu_v+c_m;\\
k_9:=\mu_v+\gamma_v+c_m; K_{E}:=\alpha_2\Gamma_{E}; K_{L}:=\alpha_2\Gamma_{L};\pi:=1-\epsilon.
\end{array}
\end{equation}
System \eqref{AR3} can be rewritten in the following way
\begin{equation}
\label{model1Ar2} \dfrac{dX}{dt}=\mathbb{A}(X)X+F
\end{equation}
with $X=\left(S_h,V_h,E_h,I_h,R_h,S_v,E_v,I_v,E,L,P\right) ^{T}$,
$\mathbb{A}(X)=\left( \begin{array}{ccccccccccc}
A_1(X)&0\\
0&A_4(X)
\end{array}\right)$
with
$$A_1(X)=\left( \begin{array}{ccccc}
-\lambda^{c}_h-k_1&\omega&0&0&0\\
\xi&-\pi\lambda^{c}_h-k_2&0&0&0\\
\lambda_h&\pi\lambda_h&-k_3&0&0\\
0&0&\gamma_h&-k_4&0\\
0&0&0&\sigma&-\mu_h\\
\end{array}\right)$$ and
$$A_2(X)=\left( \begin{array}{cccccc}
-(\lambda^{c}_v+k_8)&0&0&0&0&\theta\\
\lambda_v&-k_9&0&0&0&0\\
0&\gamma_v&-k_8&0&0&0\\
A_{96}&A_{96}&A_{96}&-A_{97}&0&0\\
0&0&0&A_{109}&-A_{10}&0\\
0&0&0&0&l&-k_7\\
\end{array}\right)$$

where $A_{96}=\mu_b\left(1-\dfrac{E}{K_E}\right)$, $A_{97}=\left(\dfrac{\mu_bN_v}{K_E}+k_5\right)$,
$A_{109}=s\left( 1-\dfrac{L}{K_L}\right)$ and $A_{10}=\dfrac{sE}{K_L}+k_6$; and
$F=\left(\Lambda_h,0,0,0,0,0,0,0,0,0,0\right)^{T}$.

Note that $\mathbb{A}(X)$ is a Metzler matrix, i.e. a matrix such that off diagonal terms are non negative \cite{Berman,Jacquez}, for all $X\in \mathbb R^{11}_{+}$. Thus, using the fact that $F\geq 0$, system \eqref{model1Ar2} is positively invariant in $\mathbb R^{11}_{+}$, which means that any trajectory of the system starting from an initial state in the positive orthant $\mathbb R^{11}_{+}$, remains forever in $\mathbb R^{11}_{+}$. The right-hand side is Lipschitz continuous: there exists a unique maximal solution.

By adding the first four equations of model system \eqref{AR3}, it follows that
$$
\dot{N}_h(t)=\Lambda_h-\mu_hN_h-\delta I_h\leq \Lambda_h-\mu_hN_h
$$
So that
$$
0\leq N_h(t)\leq \dfrac{\Lambda_h}{\mu_h}+\left(N_h(0)-\dfrac{\Lambda_h}{\mu_h}\right)e^{-\mu_ht} \\
$$
Thus, at $t\longrightarrow \infty$,  $0\leq N_h(t)\leq\dfrac{\Lambda_h}{\mu_h}$.

By adding the equations in $S_v$, $E_v$ and $E_v$ of system \eqref{AR3}, it follows that
$$
\dot{N}_v(t)=\theta P-\mu_vN_v\\
$$
So that
$$
0\leq N_v(t)=\dfrac{\theta P}{\mu_v}+\left(N_v(0)-\dfrac{\theta P}{\mu_v}\right)e^{-\mu_vt} 
$$
Thus, at $t\longrightarrow \infty$,  $0\leq N_v(t)\leq\dfrac{\theta lK_L}{\mu_vk_7}$ since $P\leq \dfrac{lK_L}{k_7}$.

Therefore, all feasible solutions of model system \eqref{AR3} enter the region:
\[
\begin{split}
\mathcal{D}&=\left\lbrace
(S_h,V_h,E_h,I_h,R_h,S_v,E_v,I_v,E,L,P)\in\mathbb R^{11}:
N_h\leq\dfrac{\Lambda_h}{\mu_h}; E\leq K_E; L\leq K_L;\right.\\
&\left. \qquad P\leq\dfrac{lK_L}{k_7};N_v\leq \dfrac{\theta lK_L}{k_7k_8}\right\rbrace\\
\end{split}
\]
\section{Mathematical analysis}
\label{sec3ar3opc}
\subsection{The disease--free equilibria and its stability}
In the absence of disease in the both population (human and Adult vector), i.e. $\lambda^{c}_h=\lambda^{c}_v=0$ (or $E_h=I_h=E_v=I_v=0$), we obtain two equilibria without disease: the trivial equilibrium (equilibrium without vector and disease) $\mathcal{E}_{0}=\left(S^{0}_h,V^{0}_h,0,0,0,0,0,0,0,0,0\right)$ and the disease--free equilibrium (equilibrium with vector and without disease)
$\mathcal{E}_{1}=\left(S^{0}_h,V^{0}_h,0,0,0,N^{0}_v,0,0,E,L,P\right)$ with
\begin{equation}
\label{TEandDFE}
\begin{array}{l}
S^{0}_h=\dfrac{\Lambda_hk_2}{\mu_h(k_2+\xi)},\;\;\;V^{0}_h=\dfrac{\xi\Lambda_h}{\mu_h(k_2+\xi)},\,\,\,
N^{0}_v=\dfrac{K_{E}K_{L}k_{5}k_{6}\left(\mathcal{N}-1\right)} {\mu_{b}\left(K_{E}s+k_{6}K_{L}\right)},\\
P=\dfrac{K_{E}K_{L}k_{5}k_{6}k_{8}\left(\mathcal{N}-1\right)}{\mu_{b}\theta\left(K_{E}s+k_{6}K_{L}\right)}
,\,\,L=\dfrac{K_{E}K_{L}k_{5}k_{6}k_{7}k_{8}\left(\mathcal{N}-1\right)}{\mu_{b}\theta l\left(K_{E}s+k_{6}K_{L}\right)},\\
E=\dfrac{K_{E}K_{L}k_{5}k_{6}k_{7}k_{8}\left(\mathcal{N}-1\right)}{s\left(\mu_{b}lK_{L}\theta+k_{5}k_{7} k_{8}K_{E}\right)}.
\end{array} 
\end{equation}
where $\mathcal{N}$ is the net reproductive number \cite{moaaca,Cushing1998,CushingetAl1994} given by
\begin{equation}
\label{Nar3}
\mathcal{N}=\dfrac{\mu_b\theta ls}{k_5k_6k_7k_8}
\end{equation}
\subsubsection{Local stability of disease--free equilibria}
The local asymtotic stability result of equilibria $\mathcal{E}_{0}$ and $\mathcal{E}_{1}$  is given in the following.
\begin{theorem}
\label{th3Ar3}
Define the basic reproductive number \cite{dihe,vawa02}
{\footnotesize
\begin{equation}
\label{R0ar3}
\begin{split}
R_0&=\sqrt{\dfrac{a^2(1-\alpha_1)^{2}\beta_{hv}\beta_{vh}\mu_{h}k_{5}k_{6}\left(\gamma_{h}+k_{4}\eta_{h}\right) \left(\gamma_{v}+k_{8}\eta_{v}\right)
 \left(\pi\xi+k_{2}\right)\alpha_2\Gamma_{E}\Gamma_{L}(\mathcal{N}-1)}{k_{3}k_{4}k_{8}k_{9} \mu_{b}\Lambda_{h}\left(\xi+k_{2}\right)\left(k_{6}\Gamma_{L}+s\Gamma_{E}\right)}}
 \end{split}
\end{equation}
}
Then, 
\begin{itemize}
\item[(i)] if $\mathcal{N}\leq 1$, the trivial equilibrium $\mathcal{E}_{0}$ is locally asymptotically stable in $\mathcal{D}$; 
\item[(ii)] if $\mathcal{N}>1$, the trivial equilibrium is unstable and the desease--free equilibrium $\mathcal{E}_{1}$ is locally asymptotically stable in $\mathcal{D}$ whenever $R_0< 1$.
\end{itemize}
\end{theorem}
\begin{proof}
See Appendix \ref{Ar3Ap1}.\hfill
\end{proof}
The basic reproduction number of a disease is the average number of secondary cases that one infectious individual produces during his infectious period in a totally susceptible population. The epidemiological implication of Theorem \ref{th3Ar3} is that, in general, when the basic reproduction number, $R_0$ is less than unity, a small influx of infectious vectors into the community would not generate large outbreaks, and the disease dies out in time (since the DFE is LAS) \cite{gaguab,dihe,vawa02,cretal2005}. However, we show in the subsection \ref{ssEEBIF} that the disease may still persist even when $R_0 < 1$.

\subsubsection{Global stabilty of the trivial equilibrium}
The global stability of the trivial equilibrium is given by the following result:
\begin{theorem}
\label{GAsTEar3} If $\mathcal{N}\leq 1$, then $\mathcal{E}_0$ is globally asymptotically stable on $\mathcal{D}$.
\end{theorem}
\begin{proof}
To prove the global asymptotic stability of the trivial disease--free equilibrium $\mathcal{E}_0$, we use the direct Lyapunov method. To this aim, we set $Y=X-TE$ with $X=(S_h,V_h,E_h,I_h,R_h,S_v,E_v,I_v,E,L,P)^{T}$ and rewrite \eqref{AR3} in the following manner
\[
\dfrac{dY}{dt}=\mathcal{B}(Y)Y.
\]
The global asymptotic stabilty of $\mathcal{E}_0$ is achieved by considering the following Lyapunov function
$\mathcal{L}(Y)=<g,Y>$ where $g=\left(1,1,1,1,1,1,1,1,\dfrac{k_8}{\mu_b},\dfrac{k_5k_8}{\mu_bs},\dfrac{k_5k_6k_8}{\mu_bsl}\right)$. See Appendix \ref{Ar3Ap2} for the details.\hfill
\end{proof}
\subsubsection{Global stabilty of the disease--free equilibrium}
We now turn to the global stabilty of the disease--free equilibrium $\mathcal{E}_1$. we prove that the disease--free  equilibrium $\mathcal{E}_1$ is globally asymptotically stable under a certain threshold condition. To this aim, we use a result obtained by Kamgang and Sallet \cite{JCK2008}, which is an extension of some results given in \cite{vawa02}. Using the property of DFE, it is possible to rewrite \eqref{AR3} in the following manner
\begin{equation}
\label{model3Ar3}
\left\lbrace \begin{array}{ll}
\dot{X}_S=&\mathcal{A}_1(X)(X_S-X_{DFE})+\mathcal{A}_{12}(X)X_I\\
\dot{X}_I=&\mathcal{A}_2(X)X_I
\end{array}\right.
\end{equation}
where $X_S$ is the vector representing the state of different compartments of non transmitting individuals 
$(S_h, V_h, R_h, S_v, E, L, P)$ and the vector $X_I$ represents the state of compartments of different transmitting individuals ($E_h$, $I_h$, $E_v$, $I_v$). Here, we have $X_S=(S_h, V_h, R_h, S_v, E, L, P)^{T}$,
$X_I=(E_h, I_h, E_v, I_v)^{T}$, $X=(X_S, X_I)$ and\\ $X_{DFE}:=\mathcal{E}_1=\left(S^{0}_h,V^{0}_h,0,0,0,N^{0}_v,0,0,E,L,P\right)^{T}$,

$$
\mathcal{A}_1(X)=\left( \begin{array}{ccccccc}
\mathcal{A}^{(1)}_1&\mathcal{A}^{(2)}_1\\
\mathcal{A}^{(3)}_1&\mathcal{A}^{(4)}_1
\end{array}\right), 
$$
with
$
\mathcal{A}^{(1)}_1(X)=\left( \begin{array}{cccc}
-(\lambda^{c}_h+k_1)&\omega&0&0\\
\xi&-(\pi\lambda^{c}_h+k_2)&0&0\\
0&0&-\mu_h&0\\
0&0&0&-(\lambda^{c}_v+k_8)\\
\end{array}\right), 
$\\
$
\mathcal{A}^{(2)}_1(X)=\left( \begin{array}{ccc}
0&0&0\\
0&0&0\\
0&0&0\\
0&0&\theta\\
\end{array}\right), 
$
$
\mathcal{A}^{(3)}_1(X)=\left( \begin{array}{ccccccc}
0&0&0&\mu_b\left(1-\dfrac{E}{K_E}\right)\\
0&0&0&0\\
0&0&0&0\\
\end{array}\right),\\ 
$
$
\mathcal{A}^{(4)}_1(X)=\left( \begin{array}{ccccccc}
-\left(k_5+\mu_b\dfrac{S^{0}_v}{K_E}\right)&0&0\\
s\left(1-\dfrac{L}{K_L}\right) &-\left(k_6+\dfrac{sE^{*}}{K_L}\right)&0\\
0&l&-k_7\\
\end{array}\right), 
$

$$
\mathcal{A}_{12}(X)=\left( \begin{array}{ccccccc}
0&0&-\dfrac{ab_{1}\eta_vS^{0}_h}{N_h}&-\dfrac{ab_{1}S^{0}_h}{N_h}&0&0&0\\
0&0&-\dfrac{ab_{1}\eta_v\pi V^{0}_h}{N_h}&-\dfrac{ab_{1}\pi V^{0}_h}{N_h}&0&0&0\\
0&\sigma&0&0&0&0&0\\
-\dfrac{ab_{2}\eta_hS^{0}_v}{N_h}&-\dfrac{ab_{2}S^{0}_v}{N_h}&0&0&0&0&0\\
0&0&\mu_b\left(1-\dfrac{E}{K_E}\right)&\mu_b\left(1-\dfrac{E}{K_E}\right)&0&0&0\\
0&0&0&0&0&0&0\\
0&0&0&0&0&0&0\\
\end{array}\right),
$$

$$
\mathcal{A}_{2}(X)=\left( \begin{array}{cccc}
-k_3&0&\dfrac{ab_{1}\eta_v(S_h+\pi V_h)}{N_h}&\dfrac{ab_{1}(S_h+\pi V_h)}{N_h}\\
\gamma_h&-k_4&0&0\\
\dfrac{ab_{2}\eta_hS_v}{N_h}&\dfrac{ab_{2}S_v}{N_h}&-k_9&0\\
0&0&\gamma_v&-k_8
\end{array}\right).
$$
A direct computation shows that the eigenvalues of $\mathcal{A}_1(X)$ have negative real parts. Thus the system $\dot{X}_S=\mathcal{A}_1(X)(X_S-X_{DFE})$ is globally asymptotically stable at $X_{DFE}$. Note also that $\mathcal{A}_{2}(X)$ is a Metzler matrix.

We now consider the bounded set $\mathcal{G}$:
\[
\begin{split}
\mathcal{G}&=\left\lbrace (S_h,V_h,E_h,I_h,R_h,S_v,E_v,I_v,E,L,P)\in\mathbb R^{11}:S_h\leq N_h, V_h\leq N_h,E_h\leq N_h,\right.\\
&\left. I_h\leq N_h, R_h\leq N_h,\bar{N_h}=\Lambda_h/(\mu_h+\delta)\leq N_h\leq N^{0}_h=\Lambda_h/\mu_h;\right.\\
&\left. E\leq K_E; L\leq K_L;P\leq\dfrac{lK_L}{k_7};N_v\leq \dfrac{\theta lK_L}{k_7k_8}\right\rbrace\\
\end{split}
\]
Let us recall the following theorem \cite{JCK2008} (See \cite{JCK2008} for a proof in a more general setting).
\begin{theorem}
\label{JCK2008}
Let $\mathcal{G}\subset\mathcal{U}=\mathbb R^{7}\times\mathbb R^{4}$. The system \eqref{AR3} is of class $C^{1}$, defined on $\mathcal{U}$. If
\begin{itemize}
\item[(1)] $\mathcal{G}$ is positively invariant relative to \eqref{model3Ar3}.
\item[(2)] The system $\dot{X}_S=\mathcal{A}_1(X)(X_S-X_{DFE})$ is Globally asymptotically stable at $X_{DFE}$.
\item[(3)] For any $x\in \mathcal{G}$, the matrix $\mathcal{A}_2(x)$ is Metzler irreducible.
\item[(4)] There exists a matrix $\bar{\mathcal{A}_2}$ , which is an upper bound of
the set\\ $\mathcal{M}=\left\lbrace \mathcal{A}_2(x)\in\mathcal{M}_{4}(\mathbb R):x\in\mathcal{G} \right\rbrace $ with the property that if $\mathcal{A}_2\in\mathcal{M}$, for any $\bar{x}\in\mathcal{G}$, such that $\mathcal{A}_2(\bar{x})=\bar{\mathcal{A}}_2$, then $\bar{x}\in \mathbb R^{7}\times\left\lbrace 0\right\rbrace $.
\item[(5)] The stability modulus of $\bar{\mathcal{A}_2}$, $\alpha(\mathcal{A}_2)=max_{\lambda\in sp(\mathcal{A}_2)}\mathcal{R}{\bf e}(\lambda)$ satisfied $\alpha(\mathcal{A}_2)\leq 0$.
\end{itemize}
Then the DFE is GAS in $\mathcal{G}$. 
\end{theorem}

For our model system \eqref{AR3}, conditions (1--3) of the theorem \ref{JCK2008} are satisfied. An upper bound of the set of matrices $\mathcal{M}$, which is the matrix $\bar{\mathcal{A}_2}$ is given by
$$
\bar{\mathcal{A}_2}=\left( \begin{array}{cccc}
-k_3&0&\dfrac{ab_{1}\eta_v(S^{0}_h+\pi V^{0}_h)}{\bar{N}_h}&\dfrac{ab_{1}(S^{0}_h+\pi V^{0}_h)}{\bar{N}_h}\\
\gamma_h&-k_4&0&0\\
\dfrac{ab_{2}\eta_hS^{0}_v}{\bar{N_h}}&\dfrac{ab_{2}S^{0}_v}{\bar{N_h}}&-k_9&0\\
0&0&\gamma_v&-k_8
\end{array}\right),
$$ 
where $\bar{N_h}=\dfrac{\Lambda_h}{(\mu_h+\delta)}$.

To check condition (5) in theorem \ref{JCK2008}, we will use the useful lemma \cite{JCK2008} in \ref{Ar3Ap4}. To this aim, let\\
$A=\left(\begin{array}{cc}
-k_3&0\\
\gamma_h&-k_4
\end{array} \right)$,
$B=\left(\begin{array}{cc}
\dfrac{ab_{1}\eta_v(S^{0}_h+\pi V^{0}_h)}{\bar{N_h}}&\dfrac{ab_{1}(S^{0}_h+\pi V^{0}_h)}{\bar{N_h}}\\
0&0
\end{array} \right)$,\\
$C=\left(\begin{array}{cc}
\dfrac{ab_{2}\eta_hS^{0}_v}{\bar{N_h}}&\dfrac{ab_{2}S^{0}_v}{\bar{N_h}}\\
0&0
\end{array} \right)$,
$D=\left(\begin{array}{cc}
-k_9&0\\
\gamma_v&-k_8
\end{array} \right)$.

Clearly, $A$ is a stable Metzler matrix. Then, after some computations, we obtain
$D-CA^{-1}B$ is a stable Metzler matrix if and only if
\begin{equation}
\label{GasDfeCondition}
\begin{split}
R_c<1
\end{split}
\end{equation}
where
{\footnotesize
\begin{equation}
\label{R_car3}
R_c=\sqrt{\dfrac{a^2(1-\alpha_1)^{2}\beta_{hv}\beta_{vh}k_{5}k_{6}\left(\gamma_{h}+k_{4}\eta_{h}\right)
 \left(\gamma_{v}+k_{8}\eta_{v}\right)K_{E}K_{L}(k_2+\pi\xi)(\mathcal{N}-1)}{k_{3}k_{4}k_{8}k_{9}\mu_{b}(k_2+\xi)(k_{6}K_{L}+K_{E}s)\Lambda_{h}} \dfrac{(\mu_h+\delta)^{2}}{\mu_h}}.
\end{equation}
}

We claim the following result
\begin{theorem}
\label{GASDFE}
If $\mathcal{N}>1$ and $R_0< R_c<1$, then the disease--free equilibrium $\mathcal{E}_1$ is globally asymptotically stable in $\mathcal{G}$.
\end{theorem}

\begin{remark}
From \eqref{R_car3}, we have
$$
R^{2}_c=\dfrac{(\mu_h+\delta)^{2}}{\mu^{2}_h}R^{2}_0>R^{2}_0,
$$
showing that $R_c$ is not necessarily an optimal threshold parameter.
\end{remark}
\begin{remark}
Note that in the absence of disease--induced death, i.e. $\delta=0$, we have $R_c=R_0$. This suggests that the disease--induced death may be a cause of the occurence of the backward bifurcation phenomenon. 
\end{remark}
\begin{remark}
The previous results are of utmost importance, because they show that if at any time, through appropriate interventions (e.g. destruction of breeding sites, massive spraying, individual protection,...), we are able to lower $\mathcal{N}$ or $R_0$ and $R_c$ to less than 1 for a sufficiently long period, then the disease can disappear \cite{duch}. 
\end{remark}
Theorem \ref{GASDFE} means that for $R_{0}<R_c<1$, the DFE is the unique equilibrium
(no co-existence with an endemic equilibrium). If $R_{c}\leq R_0\leq 1$, then it is possible to have co-existence with endemic equilibria and thus, the occurrence of backward bifurcation phenomenon. 

The backward bifurcation phenomenon, in epidemiological systems, indicate the possibility of existence of at least one endemic equilibrium when $R_0$ is less than unity. Thus, the classical requirement of $R_0<1$ is, although necessary, no longer sufficient for disease elimination \cite{gaguab,arino2003,br,SharomietAl2007}. In some epidemiological models, it has been shown that the backward bifurcation phenomenon is caused by factors such as nonlinear incidence (the infection force), disease--induced death or imperfect vaccine \cite{gaguab,SharomietAl2007,bbnamc2011,bbnamc,duhucc,saetal}. To confirm whether or not the backward bifurcation phenomenon occurs in this case, one could use the approach developed in \cite{vawa02,duhucc,ccso}, which is based on the general centre manifold theorem \cite{guho}. We will explore this method in the next section.

\subsection{Endemic equilibria and bifurcation analysis}
\label{ssEEBIF}
\subsubsection{Existence of endemic equilibria}
We turn now to the existence of endemic equilibria. Let us introduce the following quantity $R_1=R^{2}_0\vert_{\delta=0}$. We proove the following result
\begin{theorem}\label{existenceEEAR3} 
We assume that $\mathcal{N}>1$, then

(i) In the absence of disease--induced death in human population ($\delta=0$),
model system \eqref{AR3} have
\begin{itemize}
\item[1.] an unique endemic equilibrium whenever $R_1>1$.
\item[2.] no endemic equilibrium otherwise. 
\end{itemize}
(ii) In presence of disease--induced death in human population ($\delta>0$), model system \eqref{AR3} could have
\begin{itemize}
\item[3.] at least one endemic equilibrium whenever $R_0>1$.
\item[4.] zero, one or more than one endemic equilibrium whenever $R_0<1$.
\end{itemize}
\end{theorem}
\begin{proof}
See appendix \ref{Ar3Ap5}.
\end{proof}
Note that case 4 of Theorem \ref{existenceEEAR3} indicate the possibility of existence of at least one endemic equilibrium for $R_0<1$ and hence the potential occurrence of a backward bifurcation phenomenon. 

\subsubsection{Backward bifurcation analysis}
In the following, we use the center manifold theory \cite{BlaynehaetAl,vawa02,duhucc,ccso} to explore the possibility of backward bifurcation in \eqref{AR3}. To do so, a bifurcation parameter $\beta^{*}_{hv}$ is chosen, by solving for $\beta_{hv}$ from $R_0=1$, giving
\begin{equation}
\label{bifparam}
\beta^{*}_{hv}=\dfrac{k_{3}k_{4}k_{8}k_{9}\mu_{b}\Lambda_{h}\left(\xi+k_{2}\right)\left(k_{6}K_{L}+
 sK_{E}\right)}{a^2(1-\alpha_1)^{2}\beta_{vh}\mu_{h}k_{5}k_{6}\left(\gamma_{h}+k
 _{4}\,\eta_{h}\right)\,\left(\gamma_{v}+k_{8}\,\eta_{v}\right)
 \left(\pi\xi+k_{2}\right)K_{E}K_{L}(\mathcal{N}-1)}.
\end{equation}
Let $J_{\beta^{*}_{hv}}$ denotes the Jacobian of the system \eqref{AR3} evaluated at the DFE ($\mathcal{E}_1$ ) and with $\beta_{hv}=\beta^{*}_{hv}$. Thus,
\begin{equation}
\label{Jacm1Ar3}
J_{\beta^{*}_{hv}}=\left( \begin{array}{cccc}
J_1&J_2\\
J_3&J_4
\end{array} \right),
\end{equation} 
where\\
{\footnotesize
$
J_{1}=\left( \begin{array}{ccccc}
-k_1&\omega&0&0&0\\
\xi&-k_2&0&0&0\\
0&0&-k_3&0&0\\
0&0&\gamma_h&-k_4&0\\
0&0&0&\sigma&-\mu_h\\
\end{array} \right),
$
$
J_4=\left( \begin{array}{ccccccccccc}
-k_8&0&0&0&0&\theta\\
0&-k_9&0&0&0&0\\
0&\gamma_v&-k_8&0&0&0\\
K_1&K_1&K_1&-K_2&0&0\\
0&0&0&K_3&-K_4&0\\
0&0&0&0&l&-k_7
\end{array} \right) .
$ 
 \\
$
J_{2}=\left( \begin{array}{cccccc}
0&-\dfrac{a(1-\alpha)\beta^{*}_{hv}\eta_vS^{0}_h}{N^{0}_h}&-\dfrac{a(1-\alpha)\beta^{*}_{hv}S^{0}_h}{N^{0}_h}&0&0&0\\
0&-\dfrac{a(1-\alpha)\beta^{*}_{hv}\pi\eta_vV^{0}_h}{N^{0}_h}&-\dfrac{a(1-\alpha)\beta^{*}_{hv}\pi V^{0}_h}{N^{0}_h}&0&0&0\\
0&\dfrac{a(1-\alpha)\beta^{*}_{hv}\eta_vH^{0}}{N^{0}_h}&\dfrac{a(1-\alpha)\beta^{*}_{hv}H^{0}}{N^{0}_h}&0&0&0\\
0&0&0&0&0&0\\
0&0&0&0&0&0\\
\end{array} \right),
$\\ 
}
{\footnotesize
$
J_{3}=\left( \begin{array}{ccccccccccc}
0&0&-\dfrac{a(1-\alpha)\beta_{vh}\eta_hS^{0}_v}{N^{0}_h}&-\dfrac{a(1-\alpha)\beta_{vh}S^{0}_v}{N^{0}_h}&0\\
0&0&\dfrac{a(1-\alpha)\beta_{vh}\eta_hS^{0}_v}{N^{0}_h}&\dfrac{a(1-\alpha)\beta_{vh}S^{0}_v}{N^{0}_h}&0\\
0&0&0&0&0\\
0&0&0&0&0\\
0&0&0&0&0\\
0&0&0&0&0
\end{array} \right),
$
}

with $H^{0}=S^{0}_h+\pi V^{0}_h$, $K_1=\mu_b\left(1-\dfrac{E^{*}}{K_E}\right)$, $K_2=k_5+\dfrac{\mu_b}{K_E}S^{0}_v$. $K_3=s\left(1-\dfrac{L^{*}}{K_L}\right)$, and 
$K_4=\left(k_6+\dfrac{sE^{*}}{K_L}\right)$.

Note that the system \eqref{AR3}, with $\beta_{hv}=\beta^{*}_{hv}$, has a hyperbolic equilibrium point (i.e., the linearized system \eqref{AR3} has a simple eigenvalue with zero real part and all other eigenvalues
have negative real part). Hence, the center manifold theory \cite{guho,Carr} can be used to analyze the dynamics of the model \eqref{AR3} near $\beta_{hv}=\beta^{*}_{hv}$. The technique in Castillo-Chavez
and Song (2004) \cite{ccso} entails finding the left and right eigenvectors of the linearized system
above as follows.

The left eigenvector composants of $J_{\beta^{*}_{hv}}$, which correspond to the uninfected states are zero (see Lemma 3 in \cite{vawa02}). Thus a nonzero composants correspond to the infected states. It follows that the matrix $J_{\beta^{*}_{hv}}$ has a left eigenvector given by ${\bf v} = (v_1,v_2,\hdots,v_{11} )$, where 
$$
\begin{array}{l}
v_1=v_2=v_5=v_6=v_9=v_{10}=v_{11}=0;\,\,v_3=\dfrac{k_8N^{0}_h}{a(1-\alpha_1)\beta^{*}_{hv}H^{0}}v_8;\\
v_4=\dfrac{a(1-\alpha_1)\beta_{vh}S^{0}_v(\eta_vk_8+\gamma_v)}{k_4k_9N^{0}_h}v_8,\,\,
v_7=\dfrac{(\eta_vk_8+\gamma_v)}{k_9}v_8,\;\;v_8=v_8>0.
\end{array} 
$$
The system \eqref{AR3} has a right eigenvector given by ${\bf w} = (w_1 , w_2 ,\hdots, w_{11} )^{T}$ , where
$$
\begin{array}{l}
w_{11}>0,\;\; w_8>0,\\
w_{10}=\dfrac{k_7}{l}w_{11}, w_9=\dfrac{K_1\theta}{k_5k_8}w_{11},\;\;
w_7=\dfrac{k_8}{\gamma_v}w_8,\;\;w_6=\dfrac{\theta}{k_8}w_{11}-\dfrac{k_9}{\gamma_v}w_8,\;\;\\
w_5=\dfrac{\gamma_h\sigma k_8k_9N^{0}_h}{a(1-\alpha_1)\beta_{vh}\mu_h\gamma_vS^{0}_v(\eta_hk_4+\gamma_h)}w_8,\;\;
w_4=\dfrac{\mu_h}{\sigma}w_5,\;\;w_3=\dfrac{k_4}{\gamma_h}w_4,\;\;\\
w_2=-\dfrac{a(1-\alpha_1)\beta^{*}_{hv}(\eta_vk_8+\gamma_v)}{\gamma_vN^{0}_h(k_1k_2-\xi\omega)}(\xi S^{0}_h+k_1V^{0}_h)w_8,\\
w_1=\dfrac{\omega}{k_1}w_2-\dfrac{a(1-\alpha_1)\beta^{*}_{hv}S^{0}_h}{k_1N^{0}_h}\left(\eta_vw_7+w_8\right) .
\end{array} 
$$

Theorem 4.1 in Castillo-Chavez and Song \cite{ccso} is then applied to establish the existence of backward bifurcation in \eqref{AR3}. To apply such a theorem, it is convenient to let $f_k$ represent the right-hand side of the $k^{th}$ equation of the system \eqref{AR3} and let $x_k$ be the state variables whose derivative is given by the $k^{th}$ equation for $k = 1,\hdots,11$. The local bifurcation analysis near the bifurcation point ($\beta_{hv}=\beta^{*}_{hv}$) is then determined by the signs of two associated constants, denoted by $\mathcal{A}_1$ and $\mathcal{A}_2$, defined by
\begin{equation}
\label{ccsoAr3opc}
\mathcal{A}_1=\sum\limits_{k,i,j=1}^{n}v_kw_iw_j\dfrac{\partial^{2}f_k(0,0)}{\partial x_i\partial x_j}\qquad and\qquad
\mathcal{A}_2=\sum\limits_{k,i=1}^{n}v_kw_i\dfrac{\partial^{2}f_k(0,0)}{\partial x_i\partial \phi}
\end{equation}
with $\phi=\beta_{hv}-\beta^{*}_{hv}$. It is important to note that in $f_k(0,0)$, the first zero corresponds to the disease--free equilibrium, $\mathcal{E}_{1}$, for the system \eqref{AR3}. Since $\beta_{hv}=\beta^{*}_{hv}$ is the bifurcation parameter, it follows from $\phi=\beta_{hv}-\beta^{*}_{hv}$ that $\phi=0$ when $\beta_{hv}=\beta^{*}_{hv}$ which is the second component in $f_k(0,0)$.

It follows then, after some algebraic manipulations, that
$$
\mathcal{A}_1=\Gamma_1-\Gamma_2
$$
with 
\[
\begin{split}
\Gamma_1&=\dfrac{a(1-\alpha_1)\beta^{*}_{hv}(2V^{0}_hw_1+\pi S^{0}_hw_2)}{(N^{0}_h)^{2}}(\eta_vw_7+w_8)v_3\\
&+\dfrac{a(1-\alpha_1)\beta_{vh}S^{0}_v}{N^{0}_h}\left[ (\eta_hw_3+w_4)\dfrac{1}{S^{0}_v}
+\left(\eta_hw_3+\dfrac{1}{S^{0}_v}w_4\right)\right]w_6v_7,\\
\end{split}
\]
\[
\begin{split}
\Gamma_2&=2\dfrac{a(1-\alpha_1)\beta_{vh}S^{0}_v}{(N^{0}_h)^{2}}\left( \sum\limits_{i=1}^{5}w_i\right)(\eta_hw_3+w_4)v_7\\
&+\dfrac{a(1-\alpha_1)\beta^{*}_{hv}(S^{0}_h+\pi V^{0}_h)(N^{0}_h+1)}{(N^{0}_h)^{2}}\left( \sum\limits_{i=3}^{5}w_i\right)(\eta_vw_7+w_8)v_3\\
\end{split}
\]

and
$$
\mathcal{A}_2=\dfrac{a(S^{0}_h+\pi V^{0}_h)}{N^{0}_h}\left(\eta_vw_7+w_8 \right)v_3
$$
Hence, the coefficient $\mathcal{A}_1>0$ if and only if
\begin{equation}
\label{condofbacwardAr3opc}
\Gamma_1>\Gamma_2
\end{equation}
Note that the coefficient $\mathcal{A}_2$ is automatically positive. Thus, using Theorem 4.1 in \cite{ccso}, the following result is established.
\begin{theorem}
\label{thbifARopt}
The model \eqref{AR3} exhibits a backward bifurcation at $R_0 = 1$ whenever the inequality \eqref{condofbacwardAr3opc} holds. If the reversed inequality holds, then the bifurcation at $R_0 = 1$ is
forward.
\end{theorem}
The associated bifurcation diagrams are depicted in Figures~\ref{backwardar3VacPerfect} and~ \ref{Forwardar3VacPerfect}. Parameter values used in figure \ref{backwardar3VacPerfect} correspond to those in Table \ref{vaueR0ar3}, except $\Lambda_h=10$, $\epsilon=1$, $\beta_{vh}=0.8$, $\eta_h=1$, $\eta_v=1$, $\sigma=0.01428$, $\delta=1$, $\alpha_1=0.001$, $\alpha_2=1$, $c_m=0.0001$, $\Gamma_E=10^{5}$, $\Gamma_L=50000$. In this case the conditions required by Theorem~\ref{thbifARopt}, are satisfied: $\mathcal{A}_1=0.0114>0$ and $\mathcal{A}_2=1.1393>0$. 

Parameter  values used in figure \ref{Forwardar3VacPerfect} correspond to those in Table~\ref{vaueR0ar3}, except $\Lambda_h=10$, $\beta_{vh}=0.8$, $\eta_h=\eta_v=0=\delta=c_m=\alpha_1=0$, $\alpha_2=1$, $\Gamma_E=10^{5}$, $\Gamma_L=50000$.  We also have $\mathcal{A}_1=-2.4223<0$ and $\mathcal{A}_2=0.8333>0$.
\begin{figure}[t!]
\begin{center}
\includegraphics[width=\textwidth]{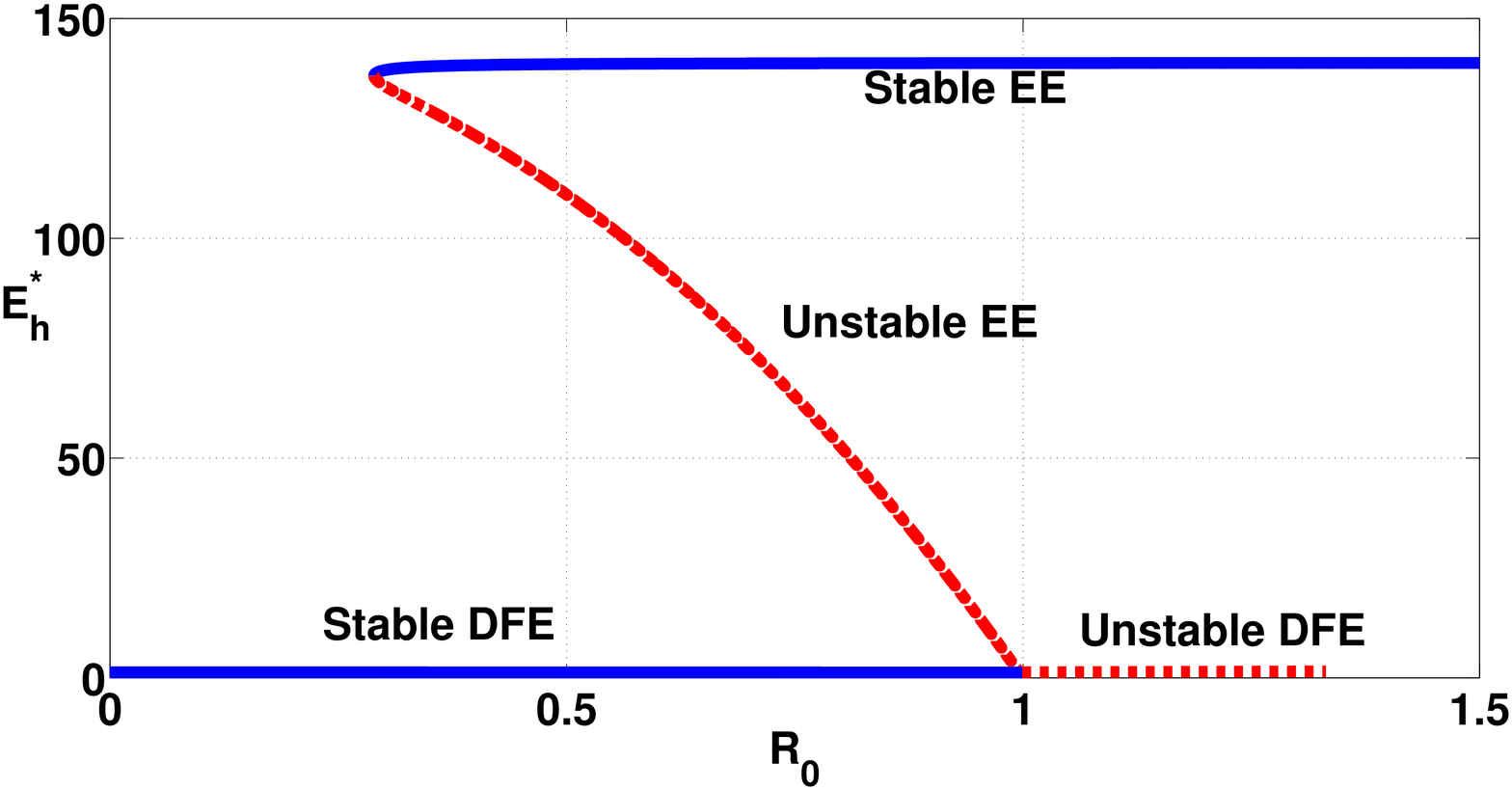}
\includegraphics[width=\textwidth]{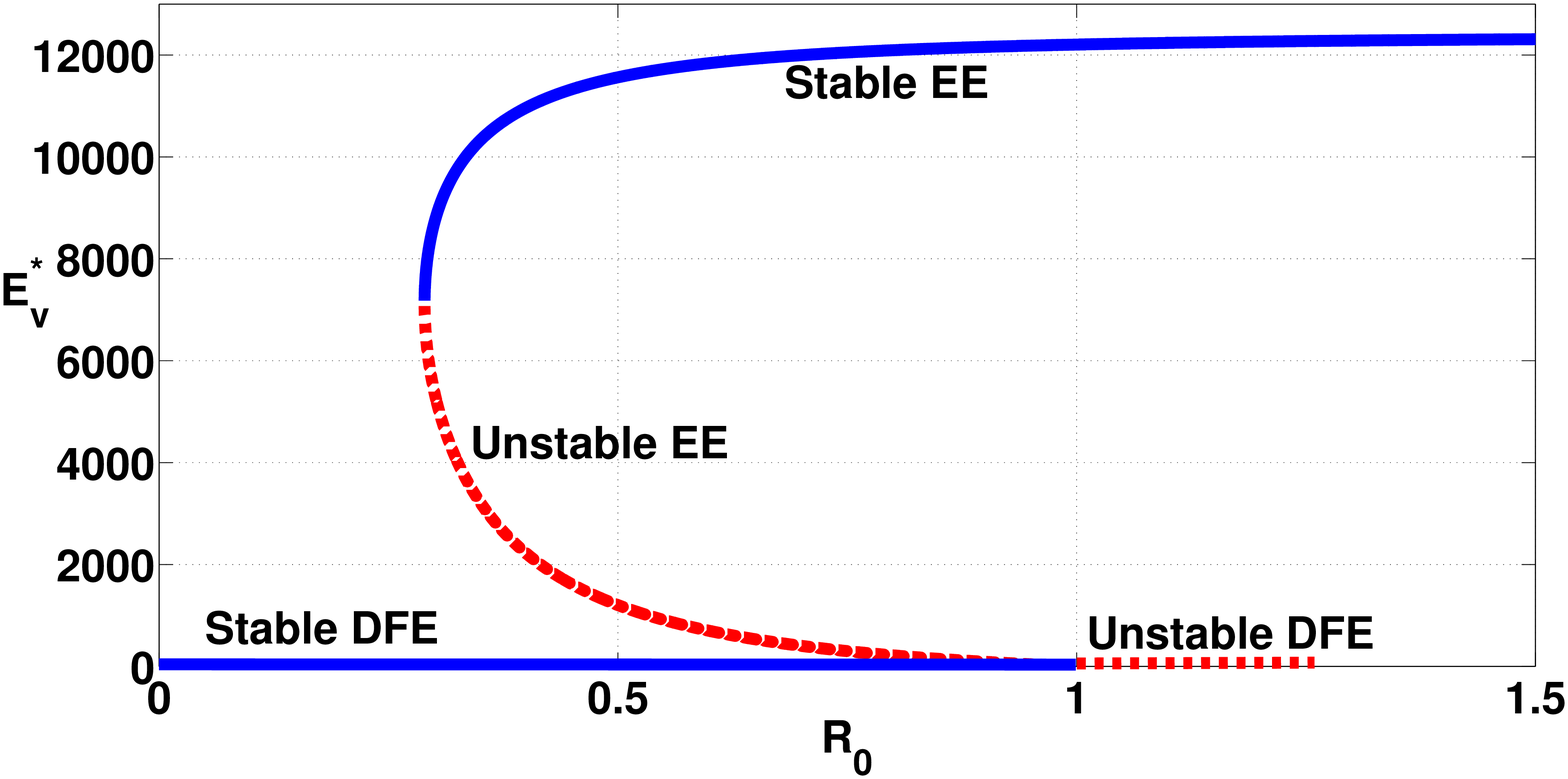}
\caption{The backward bifurcation curves for model system \eqref{AR3withoutVac} in the $(R_{0}, E^{*}_{h})$, and $(R_{0}, E^{*}_{v})$ planes. The parameter $\beta_{hv}$ is varied in the range [0, 0.2810] to allow $R_0$ to vary in the range [0, 1.5]. Two endemic equilibrium points coexist for values of $R_0$ in the range (0.2894, 1) (corresponding to the range (0.0105, 0.1249) of $\beta_{hv}$). The notation EE and DFE stand for endemic equilibrium and disease free equilibrium, respectively.  Solid line represent stable equilibria and dash line stands for  unstable equilibria. \label{backwardar3VacPerfect}}
\end{center}
\end{figure}

\begin{figure}[t!]
\begin{center}
\includegraphics[width=\textwidth]{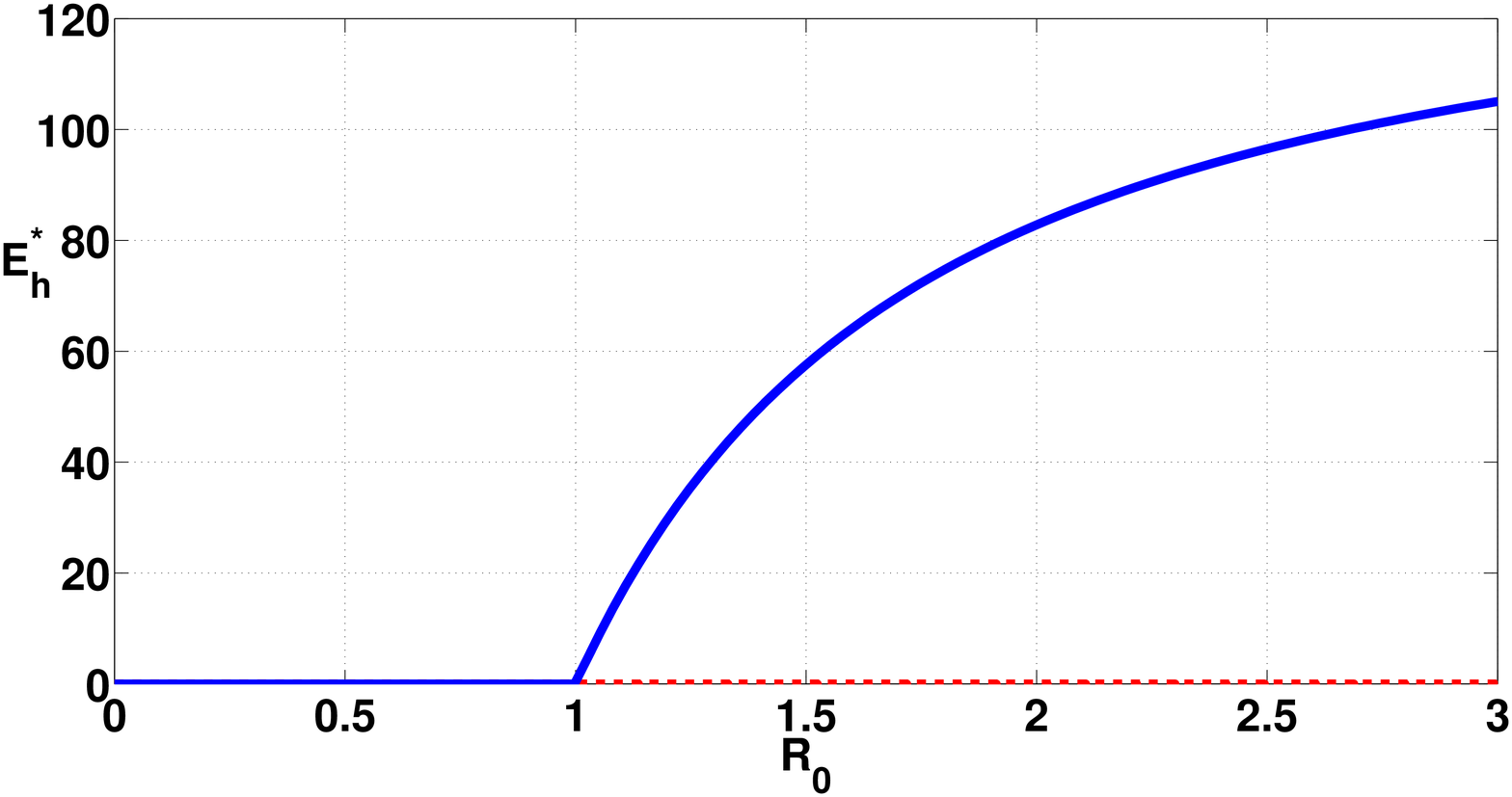}
\includegraphics[width=\textwidth]{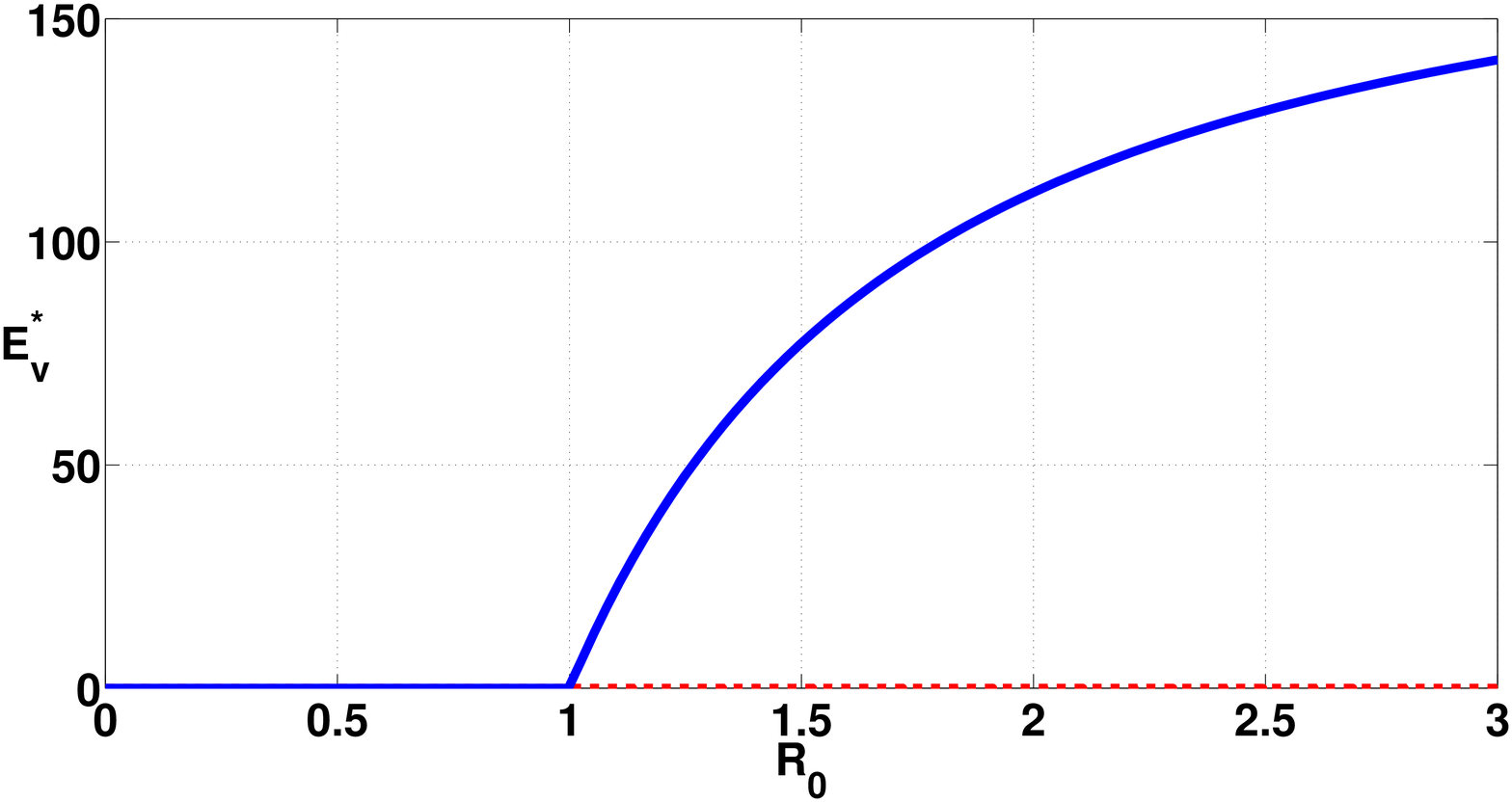}
\caption{The forward bifurcation curves for model system \eqref{AR3withoutVac} in the $(R_{0}, E^{*}_{h})$, and $(R_{0}, E^{*}_{v})$ planes. Solid line represents stable equilibria and dash line stands for  unstable equilibria. \label{Forwardar3VacPerfect}}
\end{center}
\end{figure}
The occurrence of the backward bifurcation can be also seen in Figure~\ref{EhtBackward}. Here, $R_0$ is less than the transcritical bifurcation threshold ($R_0=0.29 < 1$), but the solution of the model~\ref{AR3} can approach either the endemic equilibrium point or the disease-free equilibrium point, depending on the initial condition.
\begin{figure}[h!]
\begin{center}
\includegraphics[width=\textwidth]{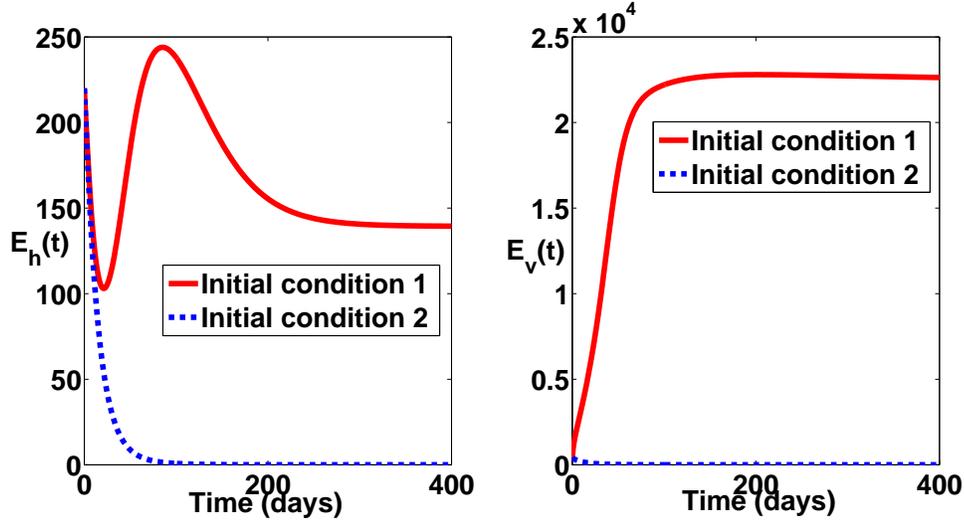}
\caption{Solutions of model \eqref{AR3} of the number of infected humans, $E_h$, and the number of infected vectors, $E_v$, for parameter values given in the bifurcation diagram in Figure~\ref{backwardar3VacPerfect} with $\beta_{hv}=0.0105$. So $R_0=0.29<1$, for two different set of initial conditions. The first set of initial conditions (corresponding to the solid trajectory) is $S_h=700$, $V_h=10$, $E_h=220$, $I_h=100$, $R_h=60$, $S_v=3000$, $E_v=400$, $I_v=120$, $E=10000$, $L=5000$ and $P=3000$. The second set of initial conditions (corresponding to the dotted trajectory) is $S_h=489100$, $V_h=10$, $E_h=220$, $I_h=100$, $R_h=60$, $S_v=3000$, $E_v=400$, $I_v=120$, $E=10000$, $L=5000$ and $P=3000$. The solution for initial condition~1 approaches the locally asymptotically stable endemic equilibrium point, while the solution for initial condition~2 approaches the locally asymptotically stable DFE.} \label{EhtBackward}
\end{center}
\end{figure}

From theorem \ref{existenceEEAR3}, item (i), it follows that the disease-induced death in human ($\delta$) may be a cause of the occurence of backward bifurcation phenomenon. In the following, we show that the backward bifurcation phenomenon is caused by the disease-induced death in human and the standard incidence functions ($\lambda^{c}_h$ and $\lambda^{c}_v$). 
\subsection{The different causes of the backward bifurcation}
\label{causeofBackward}
The occurrence of backward bifurcation phenomenon in epidemiological models, is caused by three factors: the presence of an imperfect vaccine, the presence of the death induced by the disease, and the standard incidence rates. In this section, we will consider two variants of the model \eqref{AR3} (the corresponding model without vaccination, and the corresponding model with mass action incidence), to determine the causes of this phenomenon.
\subsubsection{Analysis of the model without vaccination} 
The model without vaccination is given by
\begin{equation}
\label{AR3withoutVac}
\left\lbrace  \begin{array}{ll}
\dot{S}_h&=\Lambda_h-\left( \lambda^{c}_h+\mu_h\right)S_h\\
\dot{E}_h&=\lambda^{c}_hS_h-(\mu_h+\gamma_h)E_h\\
\dot{I}_h&=\gamma_hE_h-(\mu_h+\delta+\sigma)I_h\\
\dot{R}_h&=\sigma I_h-\mu_hR_h\\
\dot{S}_v&=\theta P-\lambda^{c}_vS_v-(\mu_v+c_m)S_v\\
\dot{E}_v&=\lambda^{c}_vS_v-(\mu_v+\gamma_v+c_m)E_v\\
\dot{I}_v&=\gamma_vE_v-(\mu_v+c_m)I_v\\
\dot{E}&=\mu_b\left(1-\dfrac{E}{\alpha_2\Gamma_{E}} \right)(S_v+E_v+I_v)-(s+\mu_E+\eta_1)E\\
\dot{L}&=sE\left(1-\dfrac{L}{\alpha_2\Gamma_{L}} \right)-(l+\mu_L+\eta_2)L\\
\dot{P}&=lL-(\theta+\mu_P)P\\
\end{array}\right.
\end{equation}
where $\lambda^{c}_h$ and $\lambda^{v}_v$ are given at section \ref{sec2ar3opc}. Model system \eqref{AR3withoutVac} is defined in the positively-invariant set 
\[
\begin{split}
\mathcal{D}_{1}&=\left\lbrace
(S_h,E_h,I_h,R_h,S_v,E_v,I_v,E,L,P)\in\mathbb R^{10}:\right.\\
&\left. N_h\leq\Lambda_h/\mu_h; E\leq K_E; L\leq K_L;P\leq\dfrac{lK_L}{k_7};N_v\leq \dfrac{\theta lK_L}{k_7k_8}\right\rbrace.\\
\end{split}
\]
Without lost of generality, we assume that $\mathcal{N}>1$. The corresponding disease--free equilibria of model \eqref{AR3withoutVac} are given by
$\mathcal{E}^{nv}_{0}=\left(N^{0}_h,0,0,0,0,0,0,0,0,0\right)$ and\\ $\mathcal{E}^{nv}_{1}=\left(N^{0}_h,0,0,0,N^{0}_v,0,0,E,L,P\right)$ with
$N^{0}_h=\frac{\Lambda_h}{\mu_h}$ and $N^{0}_v$, $E$, $L$ and $P$ are the same, given by \eqref{TEandDFE}. The associated next generation matrices, $F_1$ and $V_1 $, are, respectively, given by\\
$
F_1=\left( \begin{array}{cccc}
0&0&a(1-\alpha_1)\beta_{hv}\eta_v&a(1-\alpha_1)\beta_{hv}\\
0&0&0&0\\
\dfrac{a(1-\alpha_1)\beta_{vh}\eta_vN^{0}_v}{N^{0}_h}&\dfrac{a(1-\alpha_1)\beta_{vh}N^{0}_v}{N^{0}_h}&0&0\\
0&0&0&0
\end{array}\right)$ and\\
$
V_1=\left(\begin{array}{cccc}
k_3&0&0&0\\
-\gamma_h&k_4&0&0\\
0&0&k_9&0\\
0&0&-\gamma_v&k_8
\end{array}\right).
$

It follows that the associated reproduction number for the model without vaccination, denoted by
$R_{nv}=\rho(F_1V^{-1}_1)$, is given by
\begin{equation}
\label{RosansVacar3}
R_{nv}=\sqrt{\dfrac{a^2(1-\alpha_1)^{2}\beta_{hv}\beta_{vh}(\gamma_{h}+k_4\eta_h)(\gamma_v+k_8\eta_v)N^{0}_{v}}{k_{3}k_{4}k_{8}k_{9}N^{0}_{h}}}.
\end{equation}
Using Theorem 2 of \cite{vawa02}, we establish the following result:
\begin{theorem}
\label{lodfesansVac}
Assumed that $\mathcal{N}>1$. For basic arboviral model without vaccination, given by \eqref{AR3withoutVac}, the corresponding disease--free equilibrium is LAS if $R_{nv}<1$, and unstable if $R_{nv}>1$.
\end{theorem}
\paragraph{Existence of endemic equilibria}$\,$

Here, the existence of endemic equilibria of the model \eqref{AR3withoutVac} will be explored. Let us set the following coefficients
\begin{equation}
\label{coeffEEsnsVac}
\begin{array}{l}
R_c=\dfrac{\left\lbrace 2k_{8}(k_{3}k_{4}-\delta\gamma_{h})+(\eta_hk_4+\gamma_h)a\mu_{h}(1-\alpha_1)\beta_{vh}\right\rbrace}{k_{3}k_{4}k_{8}},\\
\footnotesize
\begin{split}
d_2&=-k_{9}\mu_{b}\Lambda_{h}(sK_E+k_6K_L)\left(k_{3}k_{4}-\delta\gamma_{h}\right)
\left((\eta_hk_4+\gamma_h)a\mu_{h}(1-\alpha_1)\beta_{vh}+(k_{3}k_{4}-\delta\gamma_{h})k_{8}\right)<0,
\end{split}\\
\normalsize
d_1=k^{2}_{3}k^{2}_{4}k_{8}k_{9}(sK_E+k_6K_L)\mu_{b}\Lambda_{h}\mu_{h}(R_{nv}^{2}-R_c),\\
d_0=k^{2}_{3}k^{2}_{4}k_{8}k_{9}(sK_E+k_6K_L)\mu_{b}\Lambda_{h}\mu_{h}^2\left(R_{nv}^{2}-1\right).
\end{array}
\end{equation}
We claim the following:
\begin{theorem}
\label{EEsansVAc}
The arboviral diseases model without vaccination \eqref{AR3withoutVac} has:
\begin{itemize}
\item[(i)] a unique endemic equilibrium if $d_0>1\Leftrightarrow R_{nv}>1$;
\item[(ii)] a unique endemic equilibrium if $d_1>0$, and $d_0=0$ or $d^{2}_1-4d_2d_0=0$;
\item[(iii)] two endemic equilibria if $d_0<0$ (i.e. $R_{nv}<1$), $d_1>0$ (i.e $R^{2}_{nv}>R_c$) and 
$d^{2}_1-4d_2d_0>0$;
\item[(iv)] no endemic equilibrium if $d_0<0$ (i.e. $R_{nv}<1$) and $\delta=0$.
\item[(v)] no endemic equilibrium otherwise.
\end{itemize}
\end{theorem}
\begin{proof}
Solving the equations in the model \eqref{AR3withoutVac} in terms of $\lambda^{c,*}_h$ and $\lambda^{c,*}_v$,
gives
\begin{equation}
\label{EEhsansvac}
\begin{array}{l}
S^{*}_h=\dfrac{\Lambda_h}{\mu_h+\lambda^{c,*}_{h}},\;\;\;
E^{*}_h=\dfrac{\lambda^{c,*}_{h}S^{*}_h}{k_3},\;\;I^{*}_h=\dfrac{\gamma_h\lambda^{c,*}_{h}S^{*}_h}{k_3k_4},\;\;R^{*}_h=\dfrac{\sigma\gamma_h\lambda^{c,*}_{h}S^{*}_h}{\mu_hk_3k_4},
\end{array}
\end{equation}
and
\begin{equation}
\label{EEvsansVac}
\begin{array}{l}
S^{*}_v=\dfrac{\theta P}{(\lambda^{c,*}_v+k_8)},\,\;\; E^{*}_v=\dfrac{\theta P\lambda^{c,*}_v}{k_9(\lambda^{c,*}_v+k_8)},\;\;I^{*}_v=\dfrac{\gamma_v\theta P\lambda^{c,*}_v}{k_8k_9(\lambda^{c,*}_v+k_8)},\\
E=\dfrac{\mu_b\theta K_EP}{(k_5k_8K_E+\mu_b\theta P)},\;\;
L=\dfrac{\mu_b\theta sK_EK_LP}{k_6K_L(k_5k_8K_E+\mu_b\theta P)+s\mu_b\theta K_EP},
\end{array}
\end{equation}
Substituting \eqref{EEhsansvac} and \eqref{EEvsansVac} into the expression of $\lambda^{*}_h$ and $\lambda^{*}_v$ and simplifying, shows that the nonzero equilibria of the model without vaccination satisfy the quadratic equation
\begin{equation}
\label{eqEEsansVac}
d_2(\lambda^{c,*}_h)^{2}+d_1\lambda^{c,*}_h+d_0=0
\end{equation}
where $d_i$, $i=0,1,2$, are given by \eqref{coeffEEsnsVac}.

Clearly, $d_2<0$ and $d_0>0$ (resp. $d_0<0$) if $R_{nv}>1$ (resp. $R_{nv}<1$). Thus Theorem \ref{EEsansVAc} is etablished.
\end{proof}
It is clear that cases (ii) and (iii) of theorem \ref{EEsansVAc} indicates the possibility of backward bifurcation (where the locally-asymptotically stable DFE co-exists with a locally-asymptotically stable endemic equilibrium when $R_{nv}<1$) in the model without vaccination \eqref{AR3withoutVac}. 

This is illustrated by simulating the model with the following set of parameter values (it should be stated that these parameters are chosen for illustrative purpose only, and may not necessarily be realistic epidemiologically): $\Lambda_h=5$, $\beta_{hv}=0.03$, $\eta_h=\eta_v=1$, $\delta=1$, $\sigma=0.01$, $c_m=0.1$, $\beta_{vh}=0.4$, $\alpha_1=0.7$ and $\alpha_2=0.5$. All other parameters are as in Table \ref{vaueR0ar3}.  With this set of parameters, $R_c=0.0216<1$, $R_{nv}=0.2725<1$ (so that $R_c<R_{nv}<1$). It follows: $d_2=-0.0263<0$, $d_1=4.8763\times 10^{-4}$ and $d_0=-3.5031\times 10^{-7}$, so that $d^{2}_1-4d_2d_0=2.0093\times 10^{-7}>0$. The resulting two endemic equilibria $\mathcal{E}^{nv}=(S^{*}_h,E^{*}_h,I^{*}_h,R^{*}_h,S^{*}_v,E^{*}_v,I^{*}_v,E,L,P)$, are:\\
$\mathcal{E}^{nv}_{1}=(281,70,5, 1207,5739,182,44,22180,10201,9977)$
 which is locally stable and\\
$\mathcal{E}^{nv}_{2}=(6333,67,4,1147,5936,37,2,22180,10201,9977)$
which is unstable.  
\begin{figure}[t!]
\begin{flushleft}
\includegraphics[width=\textwidth]{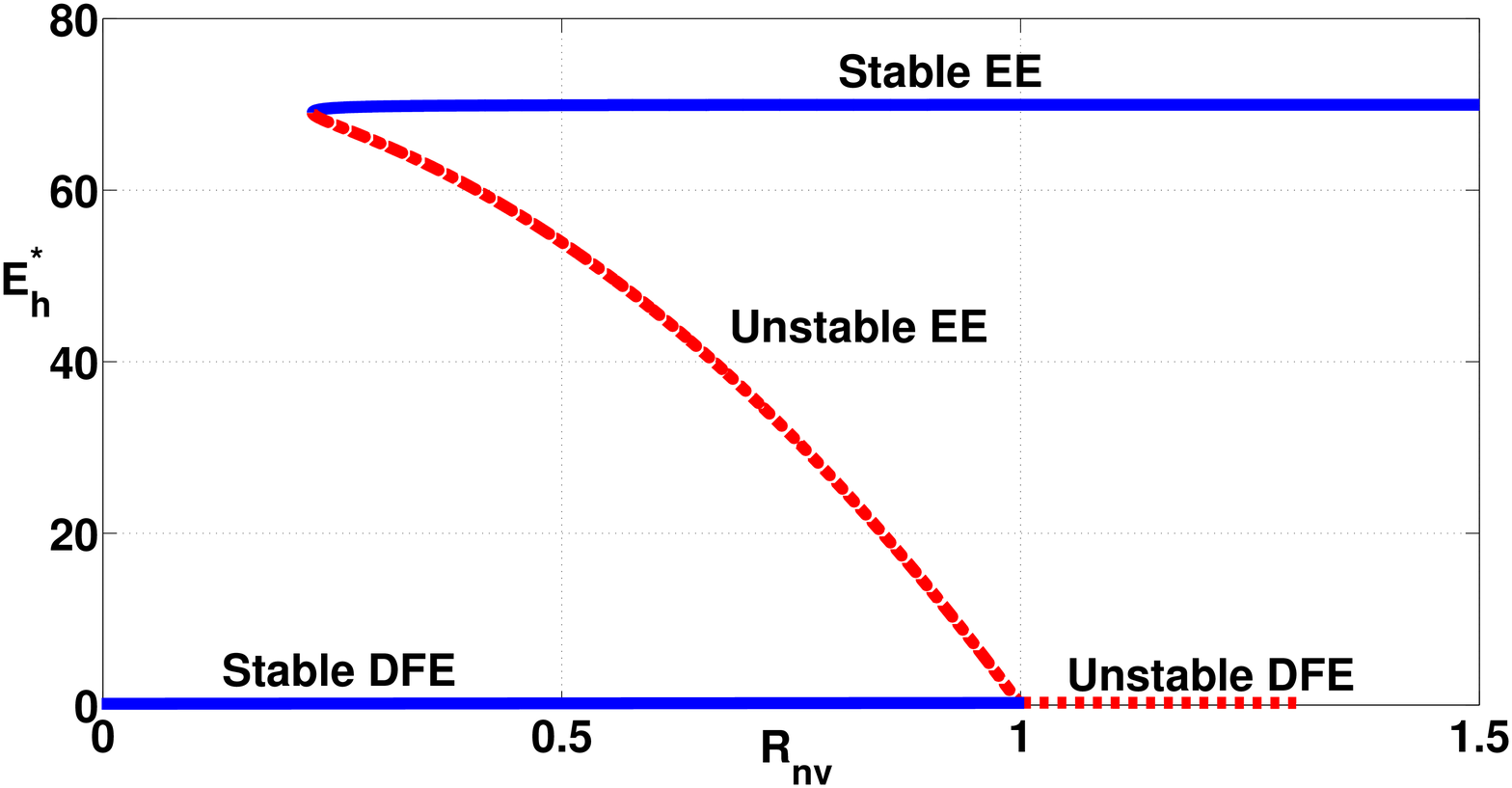}
\includegraphics[width=\textwidth]{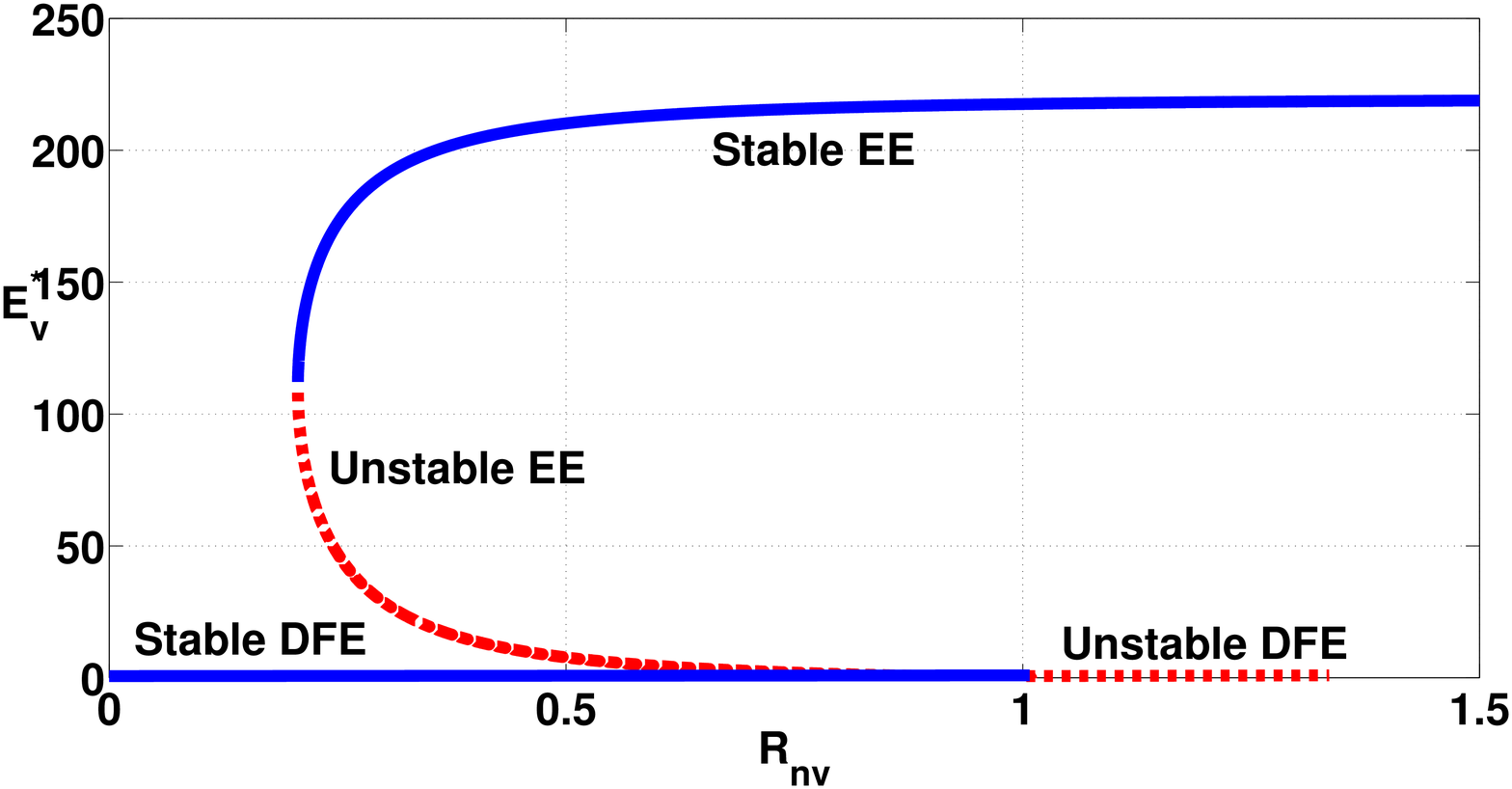}
\caption{The backward bifurcation curves for model system \eqref{AR3withoutVac} in the $(R_{nv}, E^{*}_{h})$, and $(R_{nv}, E^{*}_{v})$ planes. The parameter $\beta_{hv}$ is varied in the range [0, 0.9090] to allow $R_0$ to vary in the range [0, 1.5]. Two endemic equilibrium points coexist for values of $R_0$ in the range (0.2286, 1) (corresponding to the range (0.0211, 0.4040) of $\beta_{hv}$). The notation EE and DFE stand for endemic equilibrium and disease free equilibrium, respectively.  Solid line represent stable equilibria and dash line stands for  unstable equilibria.\label{backwardar3}}
\end{flushleft}
\end{figure}

The associated bifurcation diagram is depicted in figure \ref{backwardar3}. This clearly shows the co-existence of two locally-asymptotically stable equilibria when $R_{nv}<1$, confirming that the model without vaccination \eqref{AR3withoutVac} undergoes the phenomenon of backward bifurcation too. 

Thus, the following result is established.
\begin{lemma}
\label{nonbackward}
The model without vaccination \eqref{AR3withoutVac} undergoes backward bifurcation when Case (iii) of Theorem \ref{EEsansVAc} holds.
\end{lemma}

\paragraph{Non-existence of endemic equilibria for $R_{nv}<1$ and $\delta=0$}
\begin{lemma}
The model \eqref{AR3withoutVac} without disease--induced death ($\delta=0$) has no endemic equilibrium  when $R_{nv,\delta=0}\leq 1$, and has a unique endemic equilibrium otherwise.
\end{lemma}
\begin{proof}
Considering the model \eqref{AR3withoutVac} without disease--induced death in human, and applying the same procedure, we obtain that the nonzero equilibria of the model without vaccination satisfy the linear
equation
\[
p_1\lambda^{c,*}_h+p_0=0,
\]
where  $p_1=k_{9}k_{10}K_{12}a\mu_{b}\Lambda_{h}\mu_{h}(1-\alpha_1)\beta_{vh}+k_{3}(\mu_h+\sigma)k_{8}k_{9}K_{12}\mu_{b}\Lambda_{h}$ and \\
$p_0=-\mu_{h}k_{3}k_{4}k_{8}k_{9}K_{12}\mu_{b}\Lambda_{h}\left(R^{2}_{nv,\delta=0}-1\right)$.

Clearly, $p_1>0$ and $p_0\geq 0$ whenever $R_{nv,\delta=0}\leq 1$, so that $\lambda^{c,*}_h=-\dfrac{p_0}{p_1}\leq 0$. Therefore, the model \eqref{AR3withoutVac} without disease--induced death in human, has no endemic equilibrium whenever $R^{2}_{nv,\delta=0}\leq 1$.
\end{proof}
The above result suggests the impossibility of backward bifurcation in the model \eqref{AR3withoutVac} without disease--induced death, since no endemic equilibrium exists when $R_{nv,\delta=0}<1$ (and backward bifurcation requires the presence of at least two endemic equilibria when $R_{nv,\delta=0}<1$)
~\cite{gaguab,SharomietAl2007}. To completely rule out backward bifurcation in model \eqref{AR3withoutVac}, we use the direct Lyapunov method to proove the global stability of the DFE.
\paragraph{Global stability of the DFE of \eqref{AR3withoutVac} for $\delta=0$ }$\,$

Define the positively-invariant and attracting region
\[
\begin{split}
\mathcal{D}_{2}&=\left\lbrace
(S_h,E_h,I_h,R_h,S_v,E_v,I_v,E,L,P)\in\mathcal{D}_{1}:
S_h\leq N^{0}_h; S_v\leq N^{0}_v\right\rbrace\\
\end{split}
\]
We claim the following result.
\begin{theorem}
The DFE, $\mathcal{E}^{nv}_{1}$, of model \eqref{AR3withoutVac} without disease--induced death ($\delta=0$), is globally asymptotically stable (GAS) in $\mathcal{D}_{2}$ if $R_{nv,\delta=0}<1$.
\end{theorem}
\begin{proof}
Consider the Lyapunov function
\[
\mathcal{G}=q_1E_h+q_2I_h+q_3E_v+q_4I_v.
\]
where
\[
\begin{array}{l}
q_1=\dfrac{1}{k_3},\,\,\,
q_2=\dfrac{a^{2}(1-\alpha_1)^{2}\beta_{hv}\beta_{vh}(\eta_vk_8+\gamma_v)N^{0}_v}{k_3k_4k_8k_9N^{0}_h},\,\,
q_3=\dfrac{a(1-\alpha_1)\beta_{hv}(\eta_vk_8+\gamma_v)N^{0}_v}{k_3k_8k_9N^{0}_h},\\
q_4=\dfrac{a(1-\alpha_1)\beta_{hv}}{k_3k_8}.
\end{array}
\]
The derivative of $\mathcal{G}$ is given by
\small
\[
\begin{split}
\dot{\mathcal{G}}&=q_1\dot{E_h}+q_2\dot{I_h}+q_3\dot{E_v}+q_4\dot{I_v}\\
&=q_1(\lambda^{c}_hS_h-k_3E_h)+q_2(\gamma_hE_h-k_4I_h)+q_3(\lambda^{c}_vS_v-k_9E_v)+q_4(\gamma_vE_v-k_8I_v)\\
&=q_1((1-\alpha_1)\lambda_hS_h-k_3E_h)+q_2(\gamma_hE_h-k_4I_h)+q_3((1-\alpha_1)\lambda_vS_v-k_9E_v)+q_4(\gamma_vE_v-k_8I_v)\\
&=q_1\dfrac{\mu_h}{\Lambda_h}\left[ a(1-\alpha_1)\beta_{hv}(\eta_vE_v+I_v)S_h-\dfrac{\Lambda_h}{\mu_h}k_3E_h\right] +q_2(\gamma_hE_h-k_4I_h)\\
&+q_3\dfrac{\mu_h}{\Lambda_h}\left[a(1-\alpha_1)\beta_{vh}(\eta_hE_h+I_h)S_v-k_9\dfrac{\Lambda_h}{\mu_h}E_v\right] +q_4(\gamma_vE_v-k_8I_v)\\
&=\dfrac{1}{k_3}\dfrac{\mu_h}{\Lambda_h}a(1-\alpha_1)\beta_{hv}\eta_vS_hE_v+\dfrac{1}{k_3}\dfrac{\mu_h}{\Lambda_h}a(1-\alpha_1)\beta_{hv}S_hI_v\\
&-\dfrac{1}{k_3}k_3E_h+\dfrac{a^{2}(1-\alpha_1)^{2}\beta_{hv}\beta_{vh}(\eta_vk_8+\gamma_v)N^{0}_v}{k_3k_4k_8k_9N^{0}_h}\gamma_hE_h-\dfrac{a^{2}(1-\alpha_1)^{2}\beta_{hv}\beta_{vh}(\eta_vk_8+\gamma_v)N^{0}_v}{k_3k_4k_8k_9N^{0}_h}k_4I_h\\
&+\dfrac{a(1-\alpha_1)\beta_{hv}(\eta_vk_8+\gamma_v)N^{0}_v}{k_3k_8k_9N^{0}_h}\dfrac{\mu_h}{\Lambda_h}a(1-\alpha_1)\beta_{vh}\eta_hS_vE_h\\
&+\dfrac{a(1-\alpha_1)\beta_{hv}(\eta_vk_8+\gamma_v)N^{0}_v}{k_3k_8k_9N^{0}_h}\dfrac{\mu_h}{\Lambda_h}a(1-\alpha_1)\beta_{vh}S_vI_h-\dfrac{a(1-\alpha_1)\beta_{hv}(\eta_vk_8+\gamma_v)N^{0}_v}{k_3k_8k_9N^{0}_h}k_9E_v\\
&+\dfrac{a(1-\alpha_1)\beta_{hv}}{k_3k_8}\gamma_vE_v-\dfrac{a(1-\alpha_1)\beta_{hv}}{k_3k_8}k_8I_v\\
&\leq\left[\dfrac{1}{k_3}a(1-\alpha_1)\beta_{hv}\eta_v+q_4\gamma_v-\dfrac{a(1-\alpha_1)\beta_{hv}(\eta_vk_8+\gamma_v)N^{0}_v}{k_3k_8k_9N^{0}_h}k_9\right] E_v
+\left[\dfrac{1}{k_3}a(1-\alpha_1)\beta_{hv}-q_4k_8\right] I_v\\
&+\left[\dfrac{a(1-\alpha_1)\beta_{hv}(\eta_vk_8+\gamma_v)N^{0}_v}{k_3k_8k_9N^{0}_h}\dfrac{\mu_h}{\Lambda_h}a(1-\alpha_1)\beta_{vh}\eta_hN^{0}_v\right.\\
&\left.+\dfrac{a^{2}(1-\alpha_1)^{2}\beta_{hv}\beta_{vh}(\eta_vk_8+\gamma_v)N^{0}_v}{k_3k_4k_8k_9N^{0}_h}\gamma_h-\dfrac{1}{k_3}k_3\right] E_h\\
&+\left[\dfrac{a(1-\alpha_1)\beta_{hv}(\eta_vk_8+\gamma_v)N^{0}_v}{k_3k_8k_9N^{0}_h}\dfrac{\mu_h}{\Lambda_h}a(1-\alpha_1)\beta_{vh}N^{0}_v-\dfrac{a^{2}(1-\alpha_1)^{2}\beta_{hv}\beta_{vh}(\eta_vk_8+\gamma_v)N^{0}_v}{k_3k_4k_8k_9N^{0}_h}k_4\right] I_h\\
&=(R^{2}_{nv,\delta=0}-1) E_h\\
\end{split}
\]
\normalsize
We have $\dot{\mathcal{G}}\leq 0$ if $R_{nv,\delta=0}\leq 1$, with $\dot{\mathcal{Y}}=0$ if $\mathcal{R}_1=1$ or $E_h=0$. Whenever $E_h=0$, we also have $I_h=0$, $E_v=0$ and $I_v=0$. Substituting  $E_h=I_h=E_v=I_v=0$ in the first, fourth and fifth equation of Eq. \eqref{AR3withoutVac} with $\delta=0$ gives $S_h(t)\rightarrow S^{0}_h=N^{0}_h$, $R_h(t)\rightarrow 0$, and $S_v(t)\rightarrow S^{0}_v=N^{0}_v$ as
$t\rightarrow \infty$. Thus 
\small
\[
\begin{split}
&\left[S_h(t),E_h(t),I_h(t),R_h(t),S_v(t),E_v(t),I_v(t),E(t),L(t),P(t)\right]
\rightarrow (N^{0}_h,0,0,0,N^{0}_v,0,0,E,L,P)\\
&\text{as} \,\,\, t\rightarrow \infty.
\end{split}
\]\normalsize
It follows from the LaSalle's invariance principle \cite{invariance,14,13}
that every solution of \eqref{AR3withoutVac} (when $\mathcal{R}^{2}_{nv,\delta=0}\leq 1$), with
initial conditions in $\mathcal{D}_2$ converges to $\mathcal{E}^{nv}_{1}$, as
$t\rightarrow \infty$. Hence, the DFE, $\mathcal{E}^{nv}_{1}$, of model \eqref{AR3withoutVac} without disease--induced death, is GAS in $\mathcal{D}_2$  if $\mathcal{R}^{2}_{nv,\delta=0}\leq 1$.
\end{proof}
\subsubsection{Analysis of the model with mass action incidence} 
Consider the model \eqref{AR3} with mass action incidence. Thus, the associated forces of infection, $\lambda_h$ and $\lambda_v$, respectively, reduce to
\begin{equation}
\label{fmass}
\lambda_{mh}=C_{h}(\eta_vE_v+I_v)\quad \text{and}\quad \lambda_{mv}=C_{v}(\eta_hE_h+I_h),
\end{equation}
where, $C_{h}=a(1-\alpha_1)\beta_{hv}$ and  $C_{v}=a(1-\alpha_1)\beta_{vh}$. The resulting model (mass action model), obtained by using \eqref{fmass} in \eqref{AR3}, has the same disease--free equilibria given by \eqref{TEandDFE}. Without lost of generality, we consider that $\mathcal{N}>1$.
The associated next generation matrices, $F_m$ and $V_m$ are given by
$$
F_m=\left( \begin{array}{cccc}
0&0&C_{h}\eta_vH^{0}&C_{h}H^{0}\\
0&0&0&0\\
C_{v}\eta_vS^{0}_v&C_{v}S^{0}_v&0&0\\
0&0&0&0
\end{array}\right),\;V_m=\left(\begin{array}{cccc}
k_3&0&0&0\\
-\gamma_h&k_4&0&0\\
0&0&k_9&0\\
0&0&-\gamma_v&k_8
\end{array}\right),
$$
where $H^{0}=S^{0}_h+\pi V^{0}_h$. It follows that the associated reproduction number for the mass
action model, denoted by $R_{0,m}=\rho(F_{m}V_{m}^{-1})$, is given by
\begin{equation}
\label{Romass}
\begin{split}
R_{0,m}
&=\sqrt{\mathcal{R}^{m}_{hv}\mathcal{R}^{m}_{vh}},
\end{split}
\end{equation}
where\\
{\small $\mathcal{R}^{m}_{hv}=\left( \dfrac{a(1-\alpha_1)\beta_{hv}\Lambda_{h}\left(\gamma_{h}+k_{4}\eta_{h}\right)\left(\pi\xi+k_{2}\right)}{\mu_{h}k_{3}k_{4}\left(\xi+k_{2}\right)}\right)$ and 
$\mathcal{R}^{m}_{vh}=\left( \dfrac{a(1-\alpha_1)\beta_{vh} \left(\gamma_{v}+k_{8}\eta_{v}\right)\theta P}{k_{8}^2k_{9}}\right)$.}

Using Theorem 2 of \cite{vawa02}, the following result is established:
\begin{theorem}
\label{lmass}
Assume that $\mathcal{N}>1$. For the arboviral disease model with mass action incidence, given by \eqref{AR3} with \eqref{fmass}, the DFE, $\mathcal{E}_1$, is LAS if $R_{0,m}<1$, and unstable if $R_{0,m}>1$
\end{theorem}
\paragraph{Existence of endemic equilibria}$\,$
Solving the equations in the model \eqref{AR3withoutVac} in terms of $\lambda^{*}_{mh}$ and $\lambda^{*}_{mv}$,
gives
\begin{equation}
\label{EEhmass}
\begin{array}{l}
S^{*}_{mh}=\dfrac{\Lambda_h(\pi\lambda^{c,*}_{mh}+k_2)}
{\lambda^{c,*}_{mh}(k_2+\pi(k_1+\lambda^{c,*}_{mh}))+k_1k_2-\omega\xi},\;\;\;
V^{*}_{mh}=\dfrac{\xi S^{*}_{mh}}{k_2+\pi\lambda^{c,*}_{mh}},\;\;\;
E^{*}_{mh}=\dfrac{\lambda^{c,*}_{mh}S^{*}_{mh}}{k_3},\;\;\\
I^{*}_{mh}=\dfrac{\gamma_h\lambda^{c,*}_{h}S^{*}_{mh}}{k_3k_4},\;\;R^{*}_{mh}=\dfrac{\sigma\gamma_h\lambda^{c,*}_{mh}S^{*}_{mh}}{\mu_hk_3k_4},
\end{array}
\end{equation}
and
\begin{equation}
\label{EEvmass}
\begin{array}{l}
S^{*}_{mv}=\dfrac{\theta P}{(\lambda^{c,*}_{mv}+k_8)},\,\;\; E^{*}_{mv}=\dfrac{\theta P\lambda^{c,*}_{mv}}{k_9(\lambda^{c,*}_{mv}+k_8)},\;\;I^{*}_{mv}=\dfrac{\gamma_v\theta P\lambda^{c,*}_{mv}}{k_8k_9(\lambda^{c,*}_{mv}+k_8)}.
\end{array}
\end{equation}
Substituting \eqref{EEhmass} and \eqref{EEvmass} into the expression of $\lambda^{*}_{mh}$ and $\lambda^{*}_{mv}$ and simplifying, shows that the nonzero equilibria of the model without vaccination satisfy the quadratic equation
\begin{equation}
\label{eqEEsansVac}
e_2(\lambda^{c,*}_{mh})^{2}+e_1\lambda^{c,*}_{mh}+e_0=0,
\end{equation}
where $e_i$, $i=0,1,2$, are given by
\[
\begin{array}{l}
e_2=k_{8}k_{9}\pi
\left[\left(\gamma_{h}+k_{4}\eta_{h}\right)C_{v}\Lambda_{h}+k_{3}k_{4}k_{8}\right]\\
e_1=\dfrac{k_3k_4k^{2}_8k_{9}\kappa\pi}{(\pi\xi+k_2)}\left( R_{cm}-R^{2}_{0,m}\right),\\
e_0=k_{3}k_{4}k_{8}^2k_{9}\kappa\left(1-R^{2}_{0,m}\right),
\end{array}
\]
with $\kappa=k_1k_2-\xi\omega>0$ and 
$$
R_{cm}=\dfrac{\left[ \left(\gamma_{h}+k_{4}\eta_{h}\right)(\pi \xi+k_{2})\Lambda_{h}C_{v}
+(k_{1}\pi+k_{2})k_{3}k_{4}k_{8}\right] (\pi\xi+k_2)}{k_3k_4k_8\kappa\pi}.
$$
$e_2$ is always positive and $e_0$ is positive (resp. negative) whenever $R_{0m}$ is less (resp. greather) than unity. Thus, the mass action model admits only one endemic equilibria whenever $R_{0m}>1$. 

Now, we consider the case $R_{0m}<1$. The occurence of backward bifurcation phenomenon depend of the sign of coefficient $e_1$. The coefficent $e_1$ is always positive if and only if $R^{2}_{0,m}<R_{cm}$. It follows that the disease--free equilibrium is the unique equilibrium when $\mathcal{N}>1$ and $R_{cm}<1$. Now if $R_{cm}<R^{2}_{0,m}<1$, then in addition to the DFE  $\mathcal{E}_1$, there exists two endemic equilibria whenever $\Delta_{m}=e^{2}_1-4e_2e_0>0$. However, $R_{cm}<R^{2}_{0,m}<1\Rightarrow R_{cm}<1\Leftrightarrow 
\beta_{vh}<-\dfrac{k_3k_4k_8(\xi\omega\pi+k_{1}\pi^{2}\xi+k_{2}(\pi\xi+k_2))}{a(1-\alpha_1)\left(\gamma_{h}+k_{4}\eta_{h}\right)(\pi \xi+k_{2})(\pi\xi+k_2)\Lambda_{h}}<0$. Since all parameter of model are nonnegative, we conclude that the condition $R_{cm}<R^{2}_{0,m}<1$ does not hold.  And thus, the model with mass-action incidence does not admit endemic equilibria for $R_{cm}<1$.
\paragraph{Global stability of the DFE for the model with mass action incidence}$\,$

Since the DFE of the model with mass action incidence is the unique equilbrium whenever the corresponding basic reproduction number $R^{2}_{0,m}$ is less than unity, it remains to show that the DFE is gas. To this aim, we use the direct Lyapunov mehod. 

Let us define the following positive constants:
\[
p_1=\dfrac{1}{k_3}, p_2=\dfrac{C_hH^{0}(\eta_vk_8+\gamma_v)}{k_8k_9}\dfrac{C_vS^{0}_v}{k_3k_4},\,\,
p_3=p_1C_hH^{0}\dfrac{(\eta_vk_8+\gamma_v)}{k_8k_9},\,\,
p_4=\dfrac{C_hH^{0}}{k_3k_8}.
\]
Consider the Lyapunov function
\[
\mathcal{L}=p_1E_h+p_2I_h+p_3E_v+p_4I_v.
\]
The derivative of $\mathcal{L}$ is given by
\[
\begin{split}
\dot{\mathcal{L}}&=p_1\dot{E_h}+p_2\dot{I_h}+p_3\dot{E_v}+p_4\dot{I_v}\\
&=(p_1C_h\eta_vH+p_4\gamma_v-p_3k_9)E_v+(p_1C_hH-p_4k_8)I_v\\
&+(p_3C_v\eta_hS_v+p_2\gamma_h-p_1k_3)E_h+(p_3C_vS_v-p_2k_4)I_h\\
\end{split}
\]
Replacing $p_i,\; i=1,\hdots 4$ by their respective term, and using the fact that $H=(S_h+\pi V_h)\leq H^{0}=(S^{0}_h+\pi V^{0}_h)$ and $S_v\leq N^{0}_v$ in
\[
\begin{split}
\mathcal{D}_3&=\left\lbrace (S_h,V_h,E_h,I_h,R_h,S_v,E_v,I_v,E,L,P)\in\mathcal{D}:\right.\\
&\left. N_h\leq \frac{\Lambda_h}{\mu_h}, S_v\leq N^{0}_v=\theta P, E\leq K_E, L\leq K_L, P\leq \frac{lK_L}{k_7k_8}\right\rbrace,
\end{split}
\]
we obtain
$
\dot{\mathcal{L}}\leq \left( R^{2}_{0,m}-1\right) E_h.
$

We have $\dot{\mathcal{L}}\leq 0$ if $R_{0,m}\leq 1$, with $\dot{\mathcal{L}}=0$ if $\mathcal{R}_{0,m}=1$ or $E_h=0$. Whenever $E_h=0$, we also have $I_h=0$, $E_v=0$ and $I_v=0$. Substituting  $E_h=I_h=E_v=I_v=0$ in the first, fourth and fifth equation of Eq. \eqref{AR3withoutVac} with mass action incidence gives $S_h(t)\rightarrow S^{0}_h$, $V_h(t)\rightarrow V^{0}_h$, $R_h(t)\rightarrow 0$, and $S_v(t)\rightarrow S^{0}_v=N^{0}_v$ as
$t\rightarrow \infty$. Thus 
\small
\[
\begin{split}
&\left[S_h(t),V_h(t),E_h(t),I_h(t),R_h(t),S_v(t),E_v(t),I_v(t),E(t),L(t),P(t)\right]\\
&\rightarrow (S^{0}_h,V^{0}_h,0,0,0,N^{0}_v,0,0,E,L,P)\quad \text{as} \,\,\, t\rightarrow \infty.
\end{split}
\]\normalsize
It follows from the LaSalle's invariance principle \cite{invariance,14,13}, that every solution of \eqref{AR3} with mass action incidence, with initial conditions in $\mathcal{D}_3$ converges to the DFE, as $t\rightarrow \infty$. Hence, the DFE, $\mathcal{E}_{1}$, of the model with mass action incidence, is GAS in $\mathcal{D}_3$  if $\mathcal{R}_{0,m}\leq 1$.

Thus, we claim the following result.
\begin{theorem}
The DFE, $\mathcal{E}_{1}$, of the model \eqref{AR3} with mass action incidence, is globally asymptotically stable (GAS) in $\mathcal{D}_{3}$ if $R_{0,m}<1$.
\end{theorem}
Thus, the substitution of standard incidence with mass action incidence in the arboviral model~\eqref{AR3} removes the backward bifurcation phenomenon of the model. It should be mentioned that a similar situation was reported by Garba~\emph{et al.} in~\cite{gaguab} and by Sharomi~\emph{et al.} in~\cite{SharomietAl2007}.

We summarize the previous analysis of Subsection~\ref{causeofBackward} as follows:
\begin{lemma}
The main causes of occurence of backward bifurcation phenomenon in models \eqref{AR3} and \eqref{AR3withoutVac} are the disease--induced death and the non-linear incidence rates.
\end{lemma}
\section{Sensitivity analysis}
\label{SensitivityAr3}
As shown in the previous sections, model (\ref{AR3}) may admit single or multiple steady states according to the value of the basic reproduction  number $R_0$. In turn, $R_0$ depends on the parameters of the model. Usually there are uncertainties in data collection and estimated values, as for our model, and therefore it is important to assess the robustness of model predictions to parameter values and, in particular, to estimate the effect on $R_0$ of varying single parameters. To this aim, we use sensitivity analysis and calculate the sensitivity indices of $R_0$ to the parameters in the model using both local and global methods.
\subsection{Local sensitivity analysis}
The local sensitivity analysis, based on the {\it normalised sensitivity index} of $R_0$ (see \cite{chhycu}), is given by
\[
S_{\Psi}=\dfrac{\Psi}{R_0}\dfrac{\partial R_0}{\partial \Psi}
\]
where $\Psi$ denotes the generic parameter of \eqref{AR3}.\\
This index indicates how sensitive $R_0$ is to changes of parameter $\Psi$. Clearly, a positive (resp. negative) index indicates that an increase in the parameter value results in an increase (resp. decrease) in the $R_0$ value \cite{chhycu}.

For instance, the computation of the sensitivity index of $R_0$ with respect to $a$ is given by
\[
S_{a}=\dfrac{a}{R_0}\dfrac{\partial R_0}{\partial a}=1>0.
\]
This shows that $R_0$ is an increasing function of $a$ and the parameter $a$ has an influence on the
spread of disease. 

We tabulate the indices of the remaining parameters in Table \ref{opctab2} using parameter values on Table \ref{vaueR0ar3}. The results, displayed in Table \ref{indexR0ar3} and Figure \ref{F:SAresultsar3opc}a. The parameters are arranged from most sensitive to least.
\begin{table}[t!]
\begin{center}
\caption{Parameter values using to compute the sensitivity indices of $R_0$.\label{vaueR0ar3}}
\begin{tabular}{ccccccc}
\hline
Parameter&value&Parameter&value&Parameter&value\\
\hline
$c_m$&0.01&$s$&0.7&$\beta_{vh}$&0.75\\
$\mu_b$&6&$\eta_2$&0.3&$\Gamma_{E}$&10000\\
$\mu_P$&0.4&$\mu_E$&0.2&$\Gamma_{L}$&5000\\
$\theta$&0.08&$\epsilon$&0.61&$\alpha_2$&0.5\\
$l$&0.5&$\Lambda_h$&2.5&$\mu_h$&$\frac{1}{67*365}$\\
$a$&1&$\beta_{hv}$&0.75&$\eta_v$&0.35\\
$\mu_v$&$\frac{1}{30}$&$\mu_L$&0.4&$\sigma$&0.1428\\
$\gamma_h$&$\frac{1}{14}$&$\eta_h$&0.35&$\gamma_v$&$\frac{1}{21}$\\
$\xi$&0.5&$\omega$&0.05&$\eta_1$&0.001\\
$\delta$&0.001&$\alpha_1$&0.2\\
\hline
\end{tabular}
\end{center}
\end{table}
\begin{table}[t!]
\begin{center}
\caption{Sensitivity indices of $R_0$ to parameters of model (\ref{AR3}), evaluated at the
baseline parameter values given in Table \ref{vaueR0ar3}.\label{indexR0ar3}}
\begin{tabular}{lcccccc}
\hline
Parameter&Index&Parameter&Index&Parameter&Index\\
\hline
$a$&+1&$\sigma$&--0.2911&$\xi$&--0.0566\\
$\mu_v$&--0.9190&$c_m$&--0.2757&$\omega$&+0.0565\\
$\epsilon$&--0.6223&$\alpha_1$&--0.25&$\mu_E$&--0.0171\\
$s$&+0.5172&$\eta_h$&+0.2067&$\delta$&--0.0020\\
$\Lambda_h$&--0.5&$\gamma_h$&--0.2064&$\eta_1$&--0.0000858\\
$\beta_{hv},\beta_{vh},\Gamma_E,\Gamma_L,\alpha_2$&+0.5&$\eta_v$&+0.1207\\
$\mu_h$&+0.4996&$\gamma_v$&+0.1174\\
$\mu_P$&--0.4810&$\mu_L$&--0.1026\\
$\theta$&+0.4810&$\mu_b$&+0.0772 \\
$l$&+0.4489&$\eta_2$&--0.0770\\
\hline
\end{tabular}
\end{center}
\end{table}
The model system \eqref{AR3} is most sensitive to $a$, the average number of mosquitoes bites,
followed by $\mu_v$, $\epsilon$, $s$, $\Lambda_h$, $\beta_{hv}$, $\beta_{vh}$, $\Gamma_E$, $\Gamma_L$  and $\alpha_2$. It is important to note that increasing (decreasing) $a$ by 10\% increases (decreases) $R_0$ by 10\%. However, increasing (decreasing) the parameters $\mu_v$ by 10\% decreases (increases) $R_0$ by 9.190\%. The same reasonning can be done for other parameters.

\subsection{Uncertainty and global sensitivity analysis}
Local sensitivity analysis assesses the effects of individual parameters at particular points in parameter space without taking into account of the combined variability resulting from considering all input parameters simultaneously. Here, we perform a global sensitivity analysis to examine the model’s response to parameter variation within a wider range in the parameter space.

Following the approach by Marino \emph{et al.}\ and Wu \emph{et al.}\ \cite{Marino2008, Wu2013}, partial rank correlation coefficients (PRCC) between the basic reproduction number $R_0$ and each parameter are derived from 5,000 runs of the Latin hypercube sampling (LHS) method \cite{Stein1987}. The parameters are assumed to be random variables with uniform distributions with its mean value listed in Table~\ref{vaueR0ar3}.

With these 5,000 runs of LHS, the derived distribution of $R_0$ is given in Figure~\ref{histogramofR0opcar3}. This sampling shows that the mean of $R_0$ is 2.0642 and the standard deviation is 2.6865. The probability that $R_0>1$ is 54.86\%. This implies that for the mean of parameter values given in Table~\ref{vaueR0ar3}, we may be confident that the model predicts a endemic state.
\begin{figure}[t!]
\begin{center}
\includegraphics[width=\textwidth]{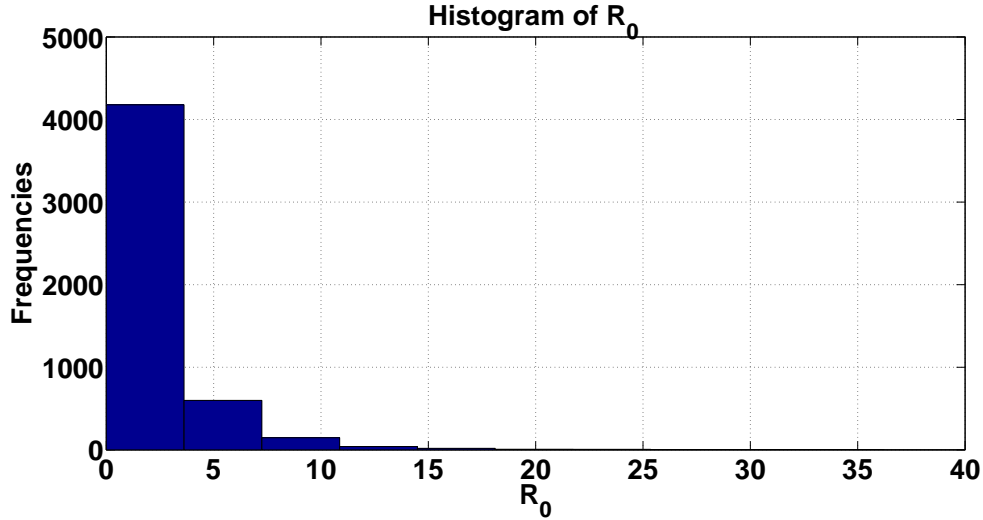}
\caption{Sampling distribution of $R_0$ from 5,000 runs of Latin hypercube sampling. The mean of $R_0$ is 2.0642  and the standard deviation is 2.6865. Furthermore, $\mathbb{P}(R_0>1)=0.5486$.  \label{histogramofR0opcar3}}
\end{center}
\end{figure}

We now use sensitivity analysis to analyze the influence of each parameter on the basic reproductive number. From the previously sampled parameter values, we compute the PRCC between $R_0$ and each parameter of model (\ref{AR3}). The parameters with large PRCC values ($>0.5$ or $<-0.5$) statistically have the most influence \cite{Wu2013}. The results, displayed in Table \ref{prccR0ar3} and Figure \ref{F:SAresultsar3opc} (b), show that the parameters $\alpha_{1}$, the human protection rate, has the highest influence on $R_0$. This suggests that individual protection may potentially be the most effective strategy to reduce $R_0$. The other parameter with an important effect are $\alpha_2$, $\beta_{hv}$, $\beta_{vh}$ and $\theta$. 

We note that the order of the most important parameters for $R_0$ from the local sensitivity analysis not match those from the global sensitivity analysis, showing that the local results are not robust.
\begin{table}[t!]
\begin{center}
\caption{Partial Rank Correlation Coefficients between $R_0$ and each parameters of model (\ref{AR3}).\label{prccR0ar3}}
\begin{tabular}{lcccccc}
\hline
Parameter&Correlation &Parameter&Correlation &Parameter&Correlation \\
&Coefficients&&Coefficients&&Coefficients\\
\hline
$\alpha_1$&--0.6125&$l$&0.3767&$\gamma_v$&0.0378\\
$\alpha_2$&0.5960&$\epsilon$&--0.3348&$\mu_L$&--0.0357\\
$\beta_{hv}$&0.5817&$s$&0.2945&$c_m$&-0.0271\\
$\beta_{vh}$&0.5815&$\sigma$&--0.1808&$\eta_h$&0.0178\\
$\theta$&0.5078&$\mu_P$&--0.1594&$\eta_1$&-0.0161\\
$a$&0.4810&$\mu_h$&0.1306&$\mu_E$&-0.0113 \\
$\mu_v$&--0.3911&$\gamma_h$&--0.0605&$\xi$&--0.0109\\
$\Gamma_L$&0.4195&$\eta_v$&0.0578&$\delta$&-0.0077\\
$\Gamma_E$&0.3888&$\mu_b$&0.0439&$\eta_2$&0.0037\\
$\Lambda_h$&--0.3876&$\omega$&0.0410\\
\hline
\end{tabular}
\end{center}
\end{table}

\begin{figure}[t!]
  \includegraphics[width=\textwidth]{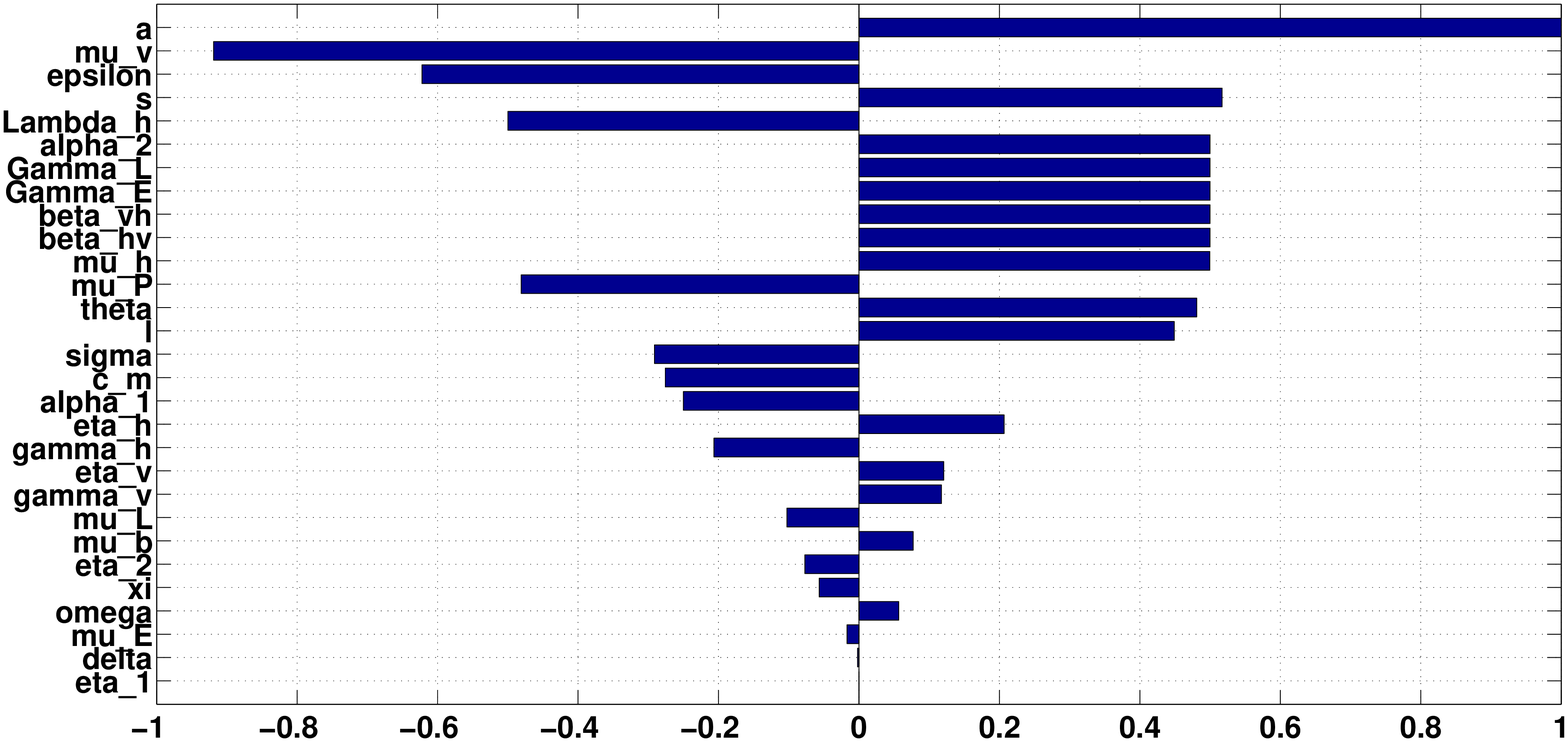}(A)
  \includegraphics[width=\textwidth]{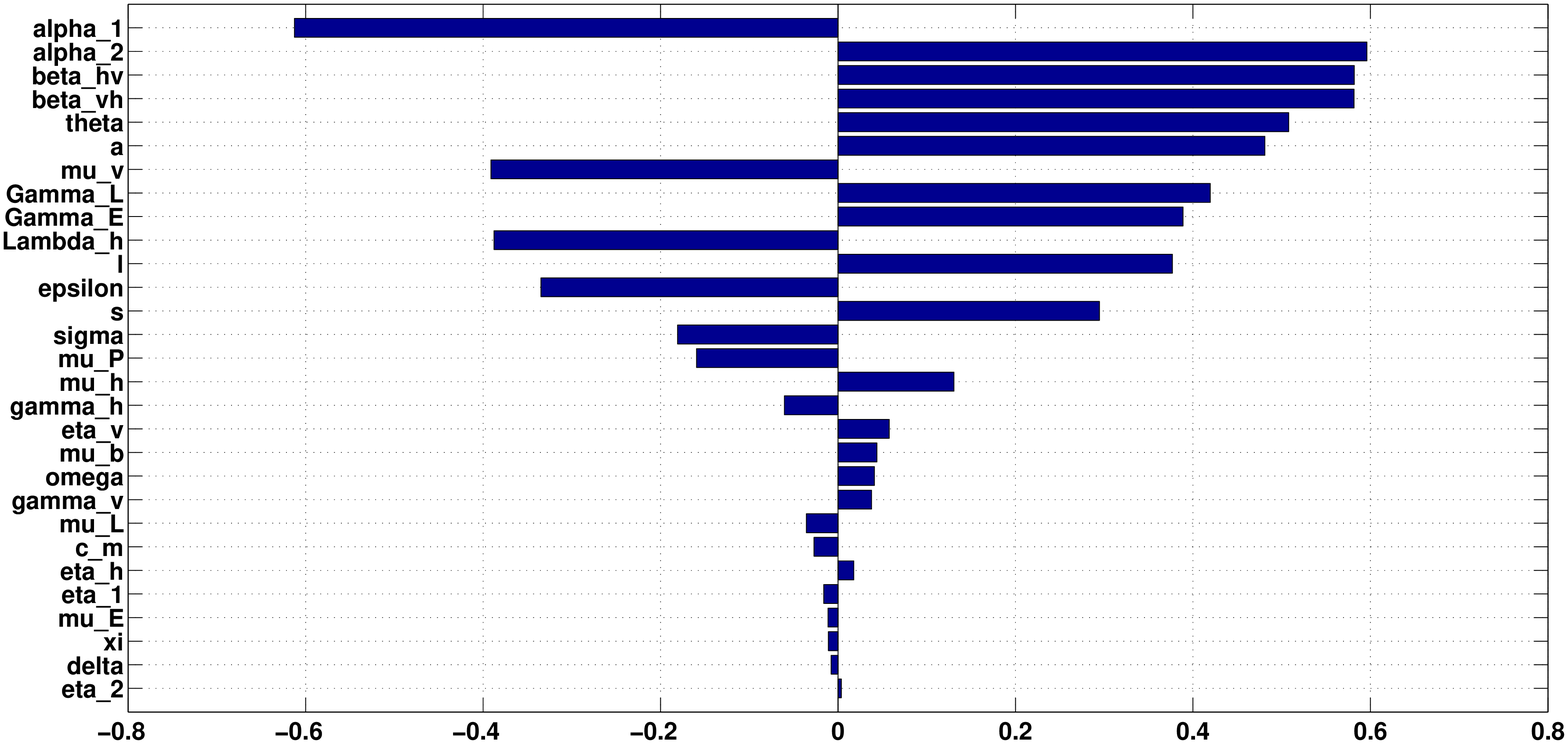}(B)
 \caption{Local (A) and global (B) sensitivity indices for $R_0$ against model parameters show that the local sensitivity results are not robust: the order of the most important parameters for $R_0$ from the local sensitivity analysis not match those from the global sensitivity analysis. }\label{F:SAresultsar3opc}
\end{figure}

\section{Numerical simulations and discussions}
\label{secNumear3opc}
In the previous model \cite{AbbouEtAl2015}, we have shown that the use of a vaccine with efficacy of about 60\%, was to be accompanied by other measurements control such as means of personal protection (Information in relation to the damage caused by these diseases, spanning wearing clothes during hours of vector activity, use of repellents), vector control (combinig the use of Adulticide to kill adult vectors, chemical control with use of Larvicide to kill the eggs and larvae, and mechanical control to reduce the number of breeding sites at least near inhabited areas)~\cite{duch}. Here, we investigate and compare numerical results, with the different scenario. We use the following initial state variables $S_h(0)=700$, $V_h(0)=10$, $E_h(0)=220$, $I_h(0)=100$, $R_h(0)=60$, $S_v(0)=3000$, $E_v(0)=400$, $I_v(0)=120$, $E(0)=10000$, $L=5000$, $P=3000$.

\subsection{Strategy A: Vaccination combined  with individual protection only}
In this strategy, we consider the model \eqref{AR3} without vector control. we set $\alpha_2=1$ and $c_m=\eta_1=\eta_2=0$ and vary the parameter related to individual potection, namely $\alpha_1$, between 0 and 0.8. The values of other parameters are given in Table \ref{vaueR0ar3}. Figure~\ref{VAC+alpha1} shows that the increase of the individual protection level, permit to reduce  the total number of infected humans, and the total number of infected vectors, but has no impact on the populations of eggs and larvae. However, from this figure, it is clear that, this reduction is significant if the level of protection must turn around 80\% at least, and this, over a long period. Thus, continuous education campaigns of people, on how to protect themselves individually, are important in the fight against the spread of arboviral diseases. 
\begin{figure}[t!]
\begin{center}
\includegraphics[width=\textwidth]{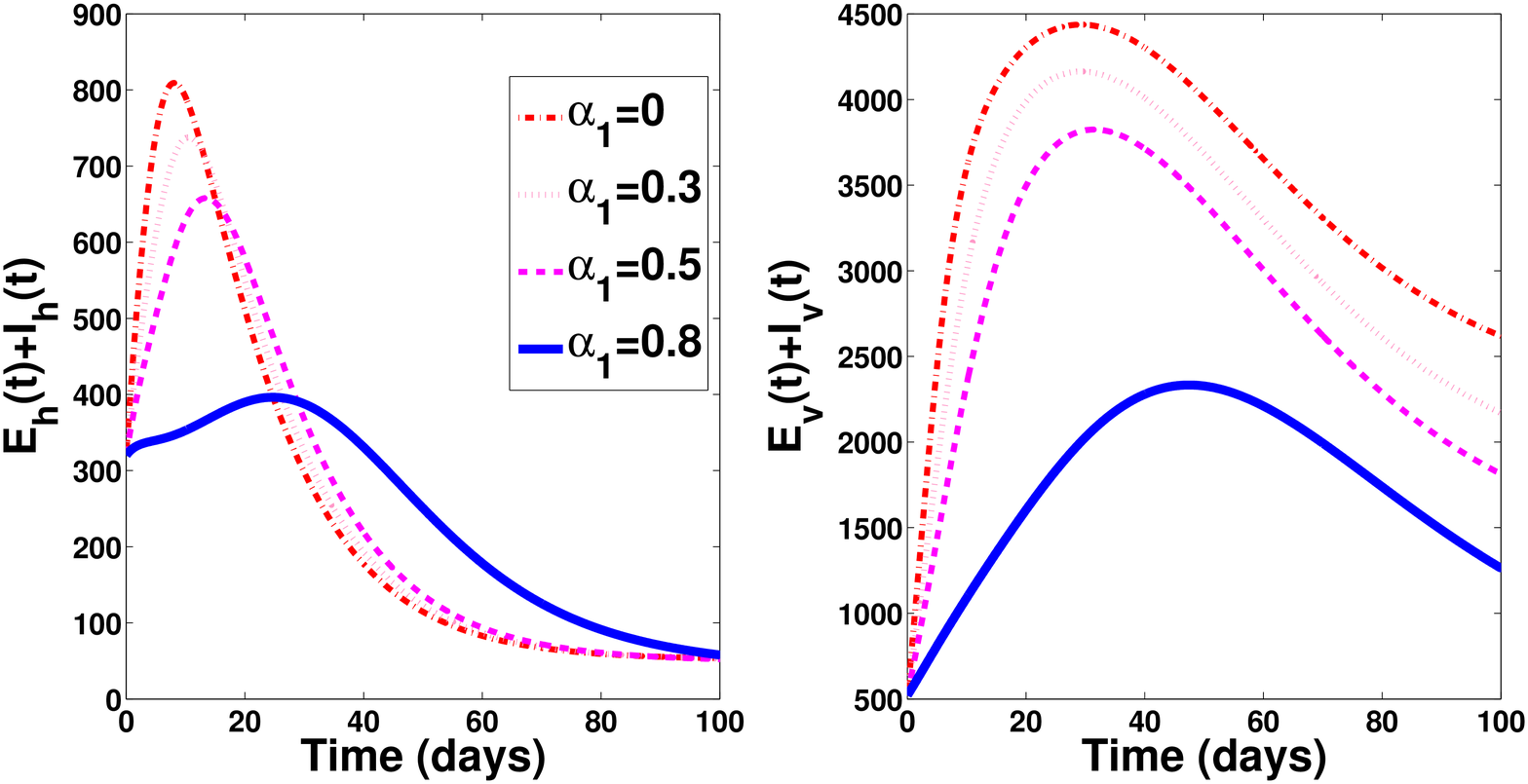}
\includegraphics[width=\textwidth]{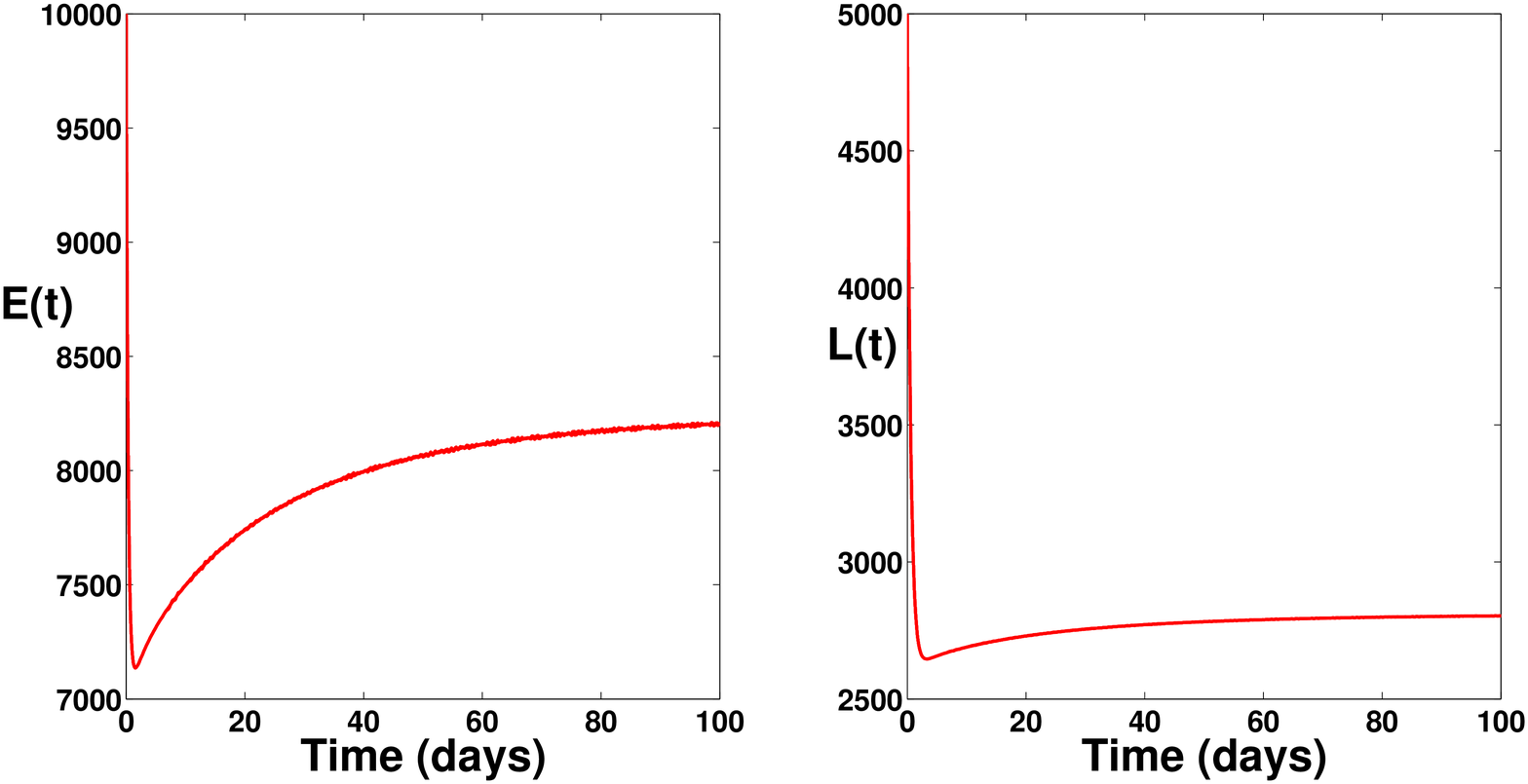}
\caption{Simulations results showing how the total number of infected humans and the total number of infected vectors decrease when the individual protection increase. All others parameters values are in Table~\ref{vaueR0ar3}.\label{VAC+alpha1}}
\end{center}
\end{figure}

\subsection{Strategy B: Vaccination combined  with adulticide}
Nowadays, Deltamethrin is the most used insecticide for impregnation of bednets, because it is a highly effective compound on mosquitoes at of very low doses~\cite{IRD}. However, when sprayed in an open environment, Deltamethrin seems to be effective only during a couple of hours~\cite{Bosc,duch}. Also, its use over a long period and continuously, leads to strong resistance of the wild populations of Aedes aegypti, for example~\cite{IRD}. The mortality of the mosquitoes after spraying varied between 20\% and 80\%. To be more realistic, we will consider the technique called "pulse control" (the control is not continuous in time order is effective only one day every $T$ days)~\cite{duch}. To this aims, we consider that spraying is carried out once a week, and this, for 100 days. We set $\alpha_1=\eta_1=\eta_2=0$ and $\alpha_2=1$.

Simulation result on figure \ref{VAC+cm} show that a mortality rate induced by the use of larvicide, cm, greater than 60\% has a significant impact on the decrease of the total number of infected humans and vectors,  and on the decrease of eggs and larvae.
\begin{figure}[t!]
\begin{center}
\includegraphics[width=\textwidth]{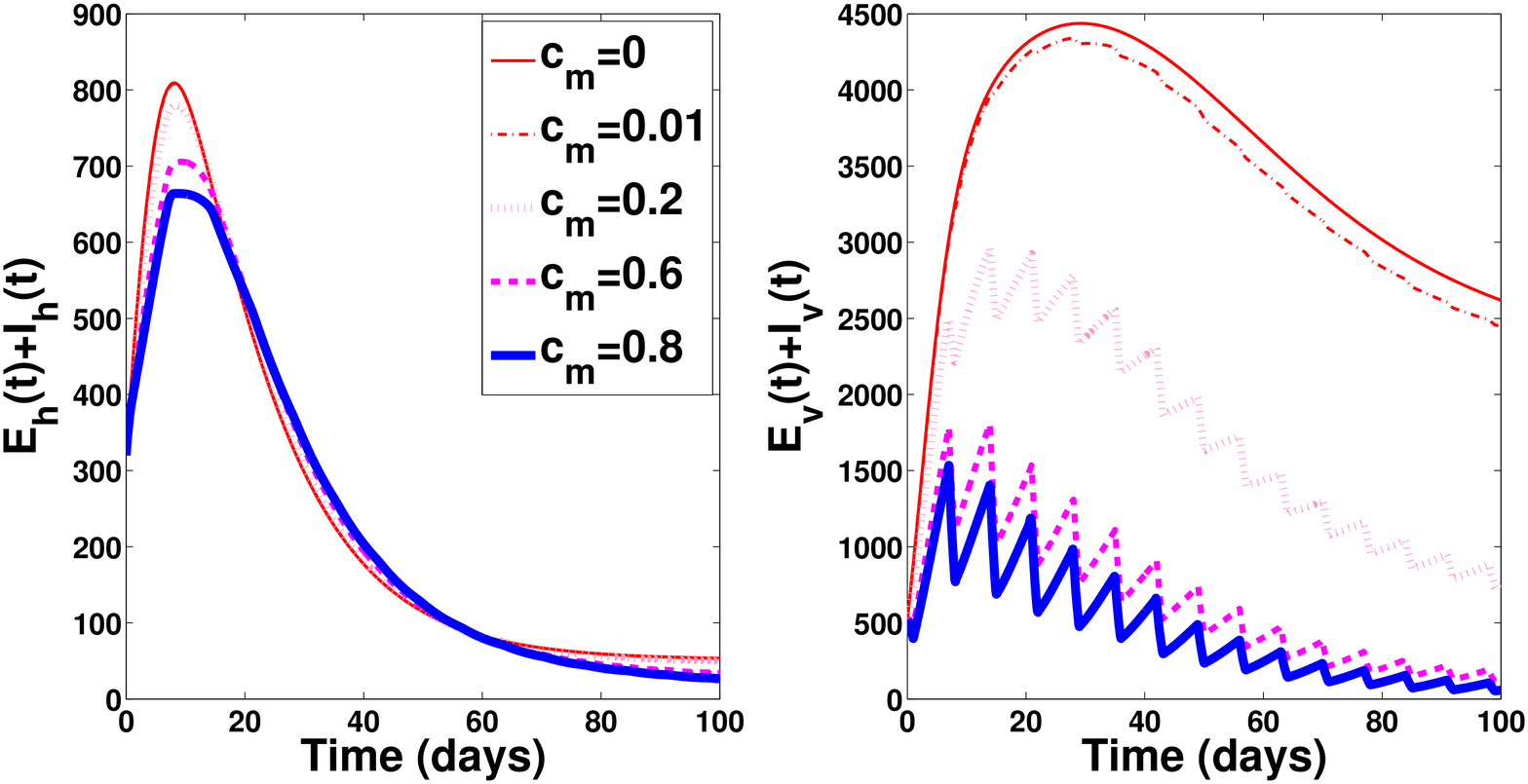}
\includegraphics[width=\textwidth]{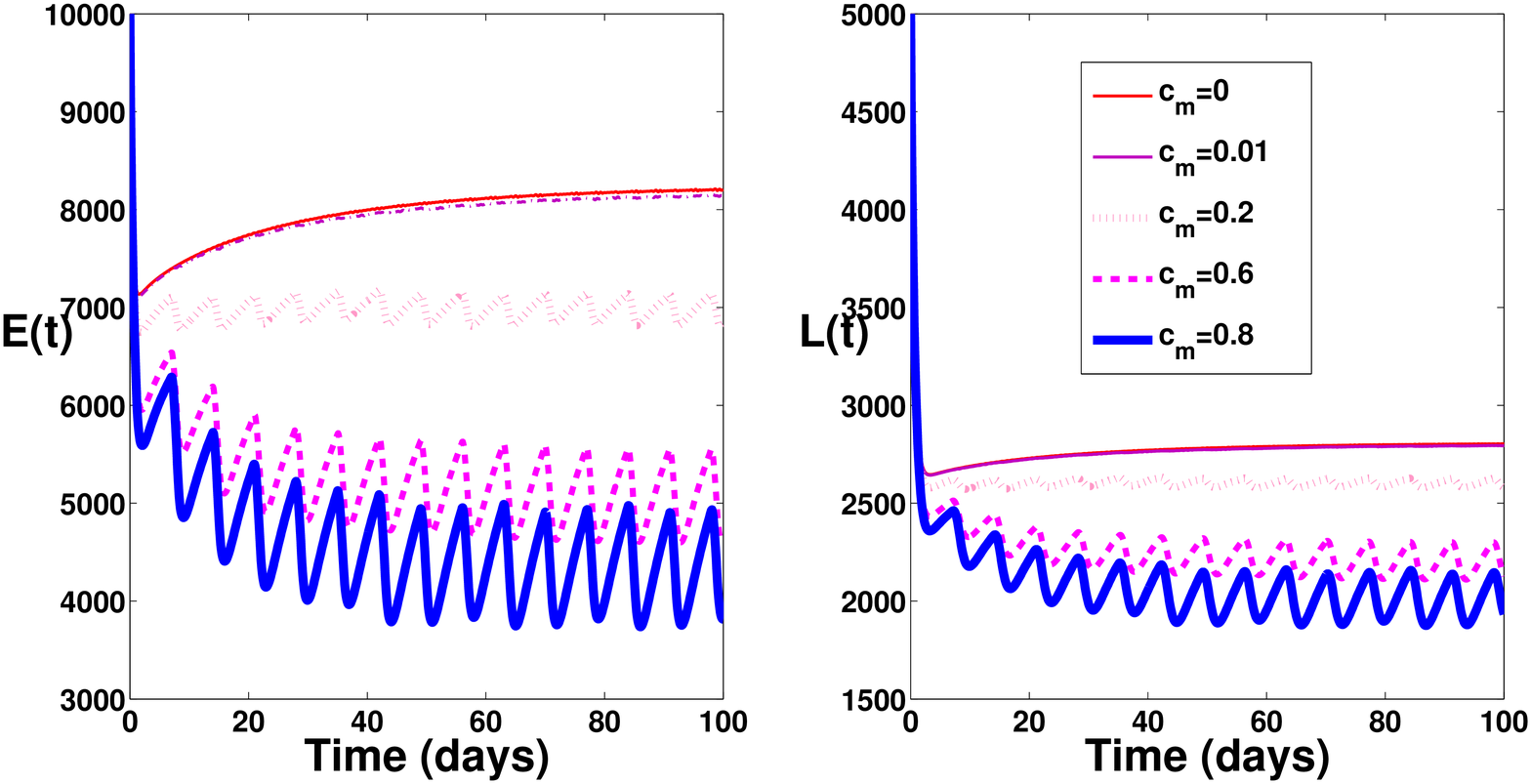}
\caption{Simulations results showing how  the total number of infected vectors, eggs and larvae populations dicrease when the aldulticide control parameter $c_m$ increase. All others parameters values are in Table~\ref{vaueR0ar3}.\label{VAC+cm}}
\end{center}
\end{figure}

\subsection{Strategy C: Vaccination combined  with larvicide}
Since the efficacy and the duration of a larvicide (\emph{Bti=Bacillus thuringiensis var. israelensis}) strongly depend on several factors like water quality, exposure, and even the type of breeding sites. To be more realistic, we thus consider that the duration can vary between a couple of days and two weeks~\cite{duch,Licciardi}. We consider that the larvicide spraying happens once every 15 days, and this, on a period of 100 days. We set $\alpha_1=c_m=0$ and $\alpha_2=1$.

The figure \ref{VAC+eta} shows that the use of larvicide has no significant impact on the decrease of total number of infected humans and vectors, as well as on the number of eggs and larvae. This can be justified by the fact that the use of conventional larvicides neccéssite certain constraints on their use: they can not be used continuously, their duration of action decreases with time. In addition, eggs of certain populations of vectors such as Aedes albopictus, come into prolonged hibernation when conditions in the breading sites are not conducive to their good growth (this is justified by the control rate value $\eta_1=0.001$). Also, the pupae do not  consume anything, until reaching the mature stage. 
\begin{figure}[t!]
\begin{center}
\includegraphics[width=\textwidth]{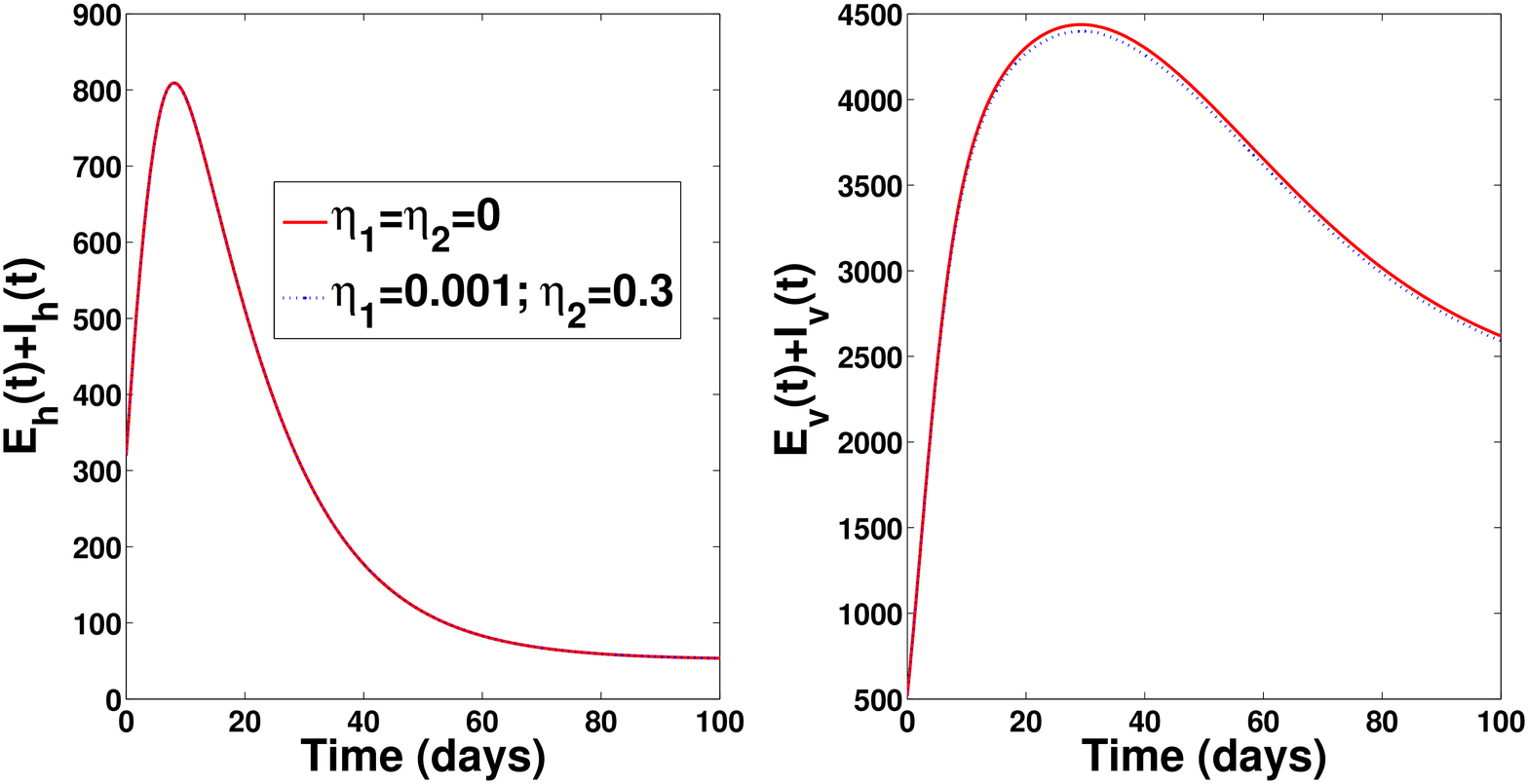}
\includegraphics[width=\textwidth]{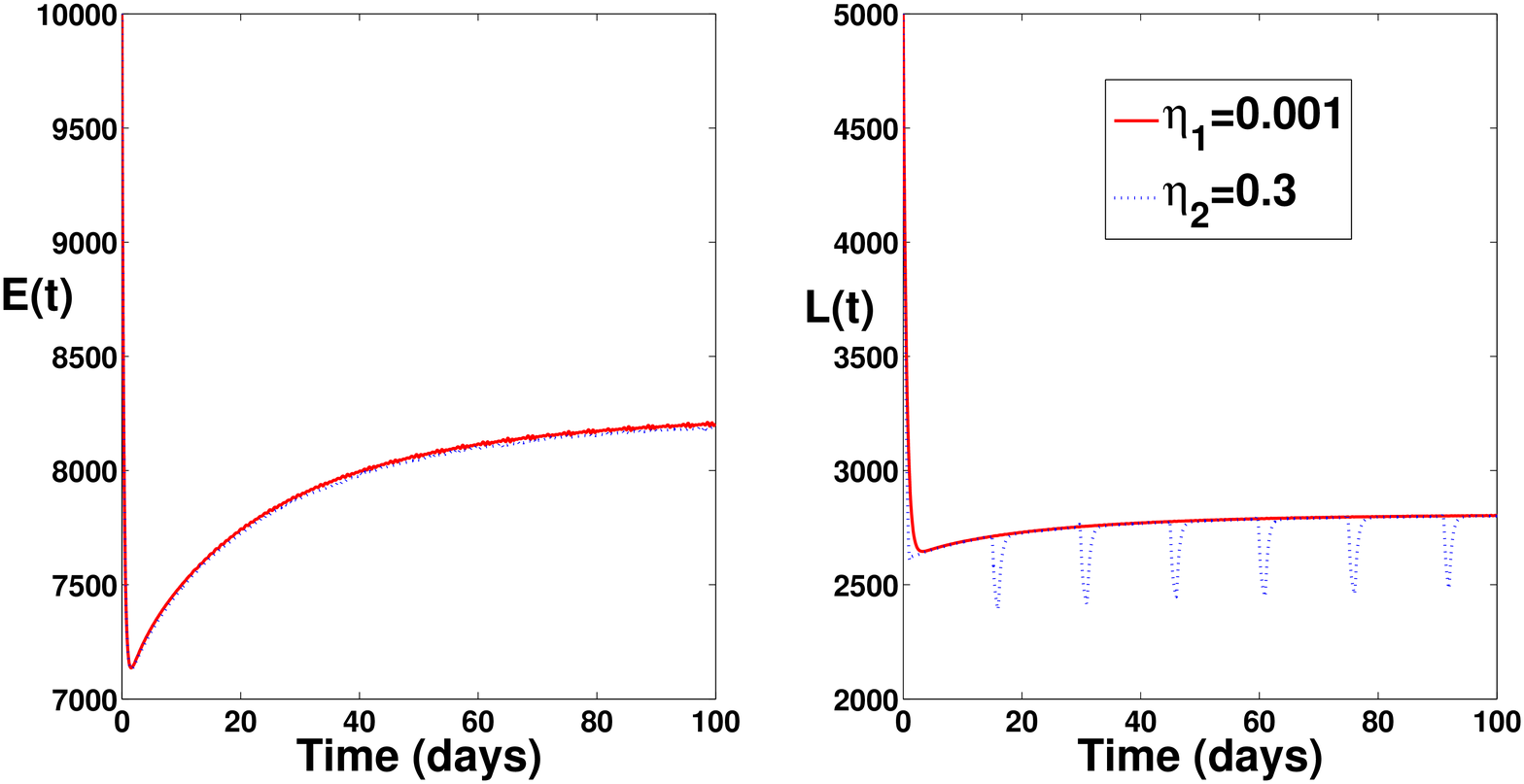}
\caption{Simulations results showing how  the total number of infected huamans, the total number of infected vectors, and the eggs and larvae populations dicrease whith the larvicide control associated parameters $\eta_1$ and $\eta_2$. All others parameters values are in Table~\ref{vaueR0ar3}.\label{VAC+eta}}
\end{center}
\end{figure}

\subsection{Strategy D: Vaccination combined  with mechanical control}
The effectiveness of this type of control depends largely of the awareness campaigns of local populations, in the sense that, to reduce the proliferation of vectors, people must always keep their environment clean by the systematic destruction of breeding sites. So, we consider that this type of control can be achieved by local populations, and this, every daily. We set $\alpha_1=c_m=0=\eta_1=\eta_2$. 

The figure \ref{VAC+alpha2} shows that this type of control is appropriate in the fight against the proliferation of vectors. This can only be possible by the multiplication of local populations awareness campaigns.
\begin{figure}[t!]
\begin{center}
\includegraphics[width=\textwidth]{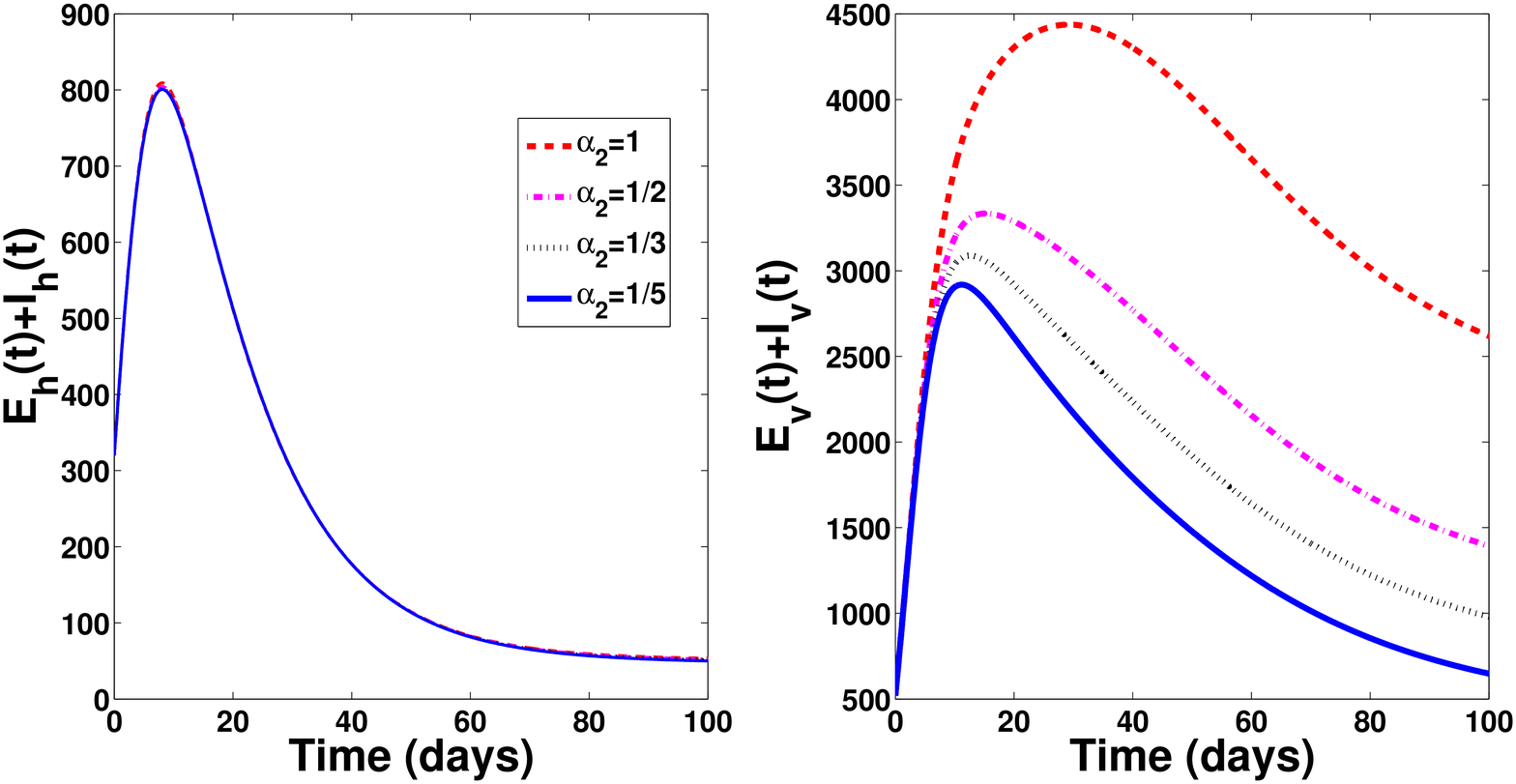}
\includegraphics[width=\textwidth]{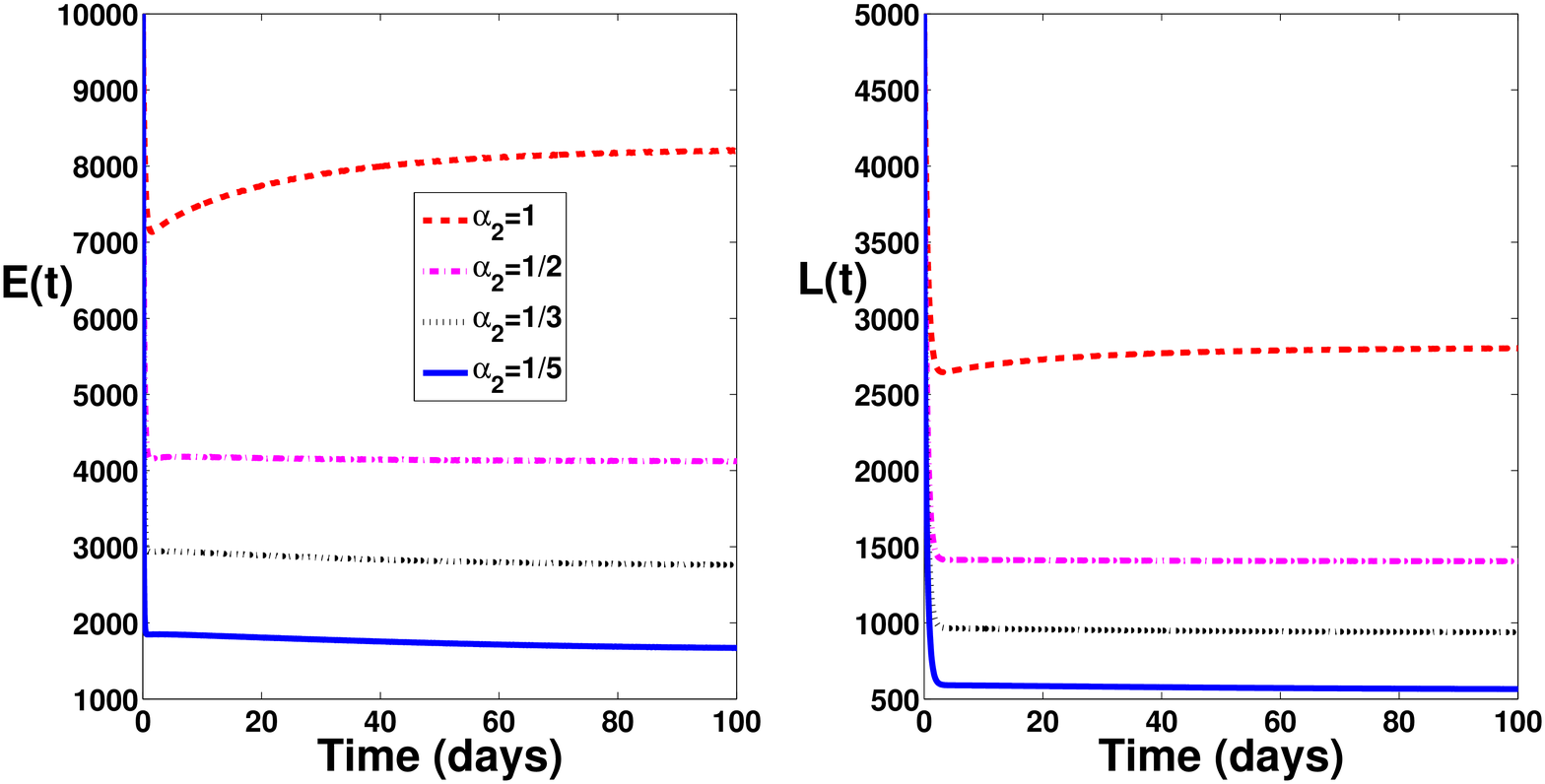}
\caption{Simulations results showing how  the total number of infected vectors, eggs and larvae populations dicrease whith the mechanical control associated parameter $\alpha_2$. All others parameters values are in Table~\ref{vaueR0ar3}.\label{VAC+alpha2}}
\end{center}
\end{figure}

\subsection{Strategy E: Combining vaccination, individual protection and adulticide}
In this strategy, we consider the model \eqref{AR3} without larvicide and mechanical control. we set $\alpha_2=1$ and $\eta_1=\eta_2=0$ and vary the parameter related to individual potection and the use of adulticide, namely $\alpha_1$ and $c_m$, respectively, between 0 and 0.8. The values of other parameters are given in Table \ref{vaueR0ar3}. Figure \ref{VAC+alpha1+cm} shows that the use of the combination of these controls decreases significantly the total number of infected humans, infected vectors as well as the number of eggs and larvae, when its associated rates, namely $\alpha_1$ and $c_m$, are greater than 0.3 and 0.2, respectively.
\begin{figure}[t!]
\begin{center}
\includegraphics[width=\textwidth]{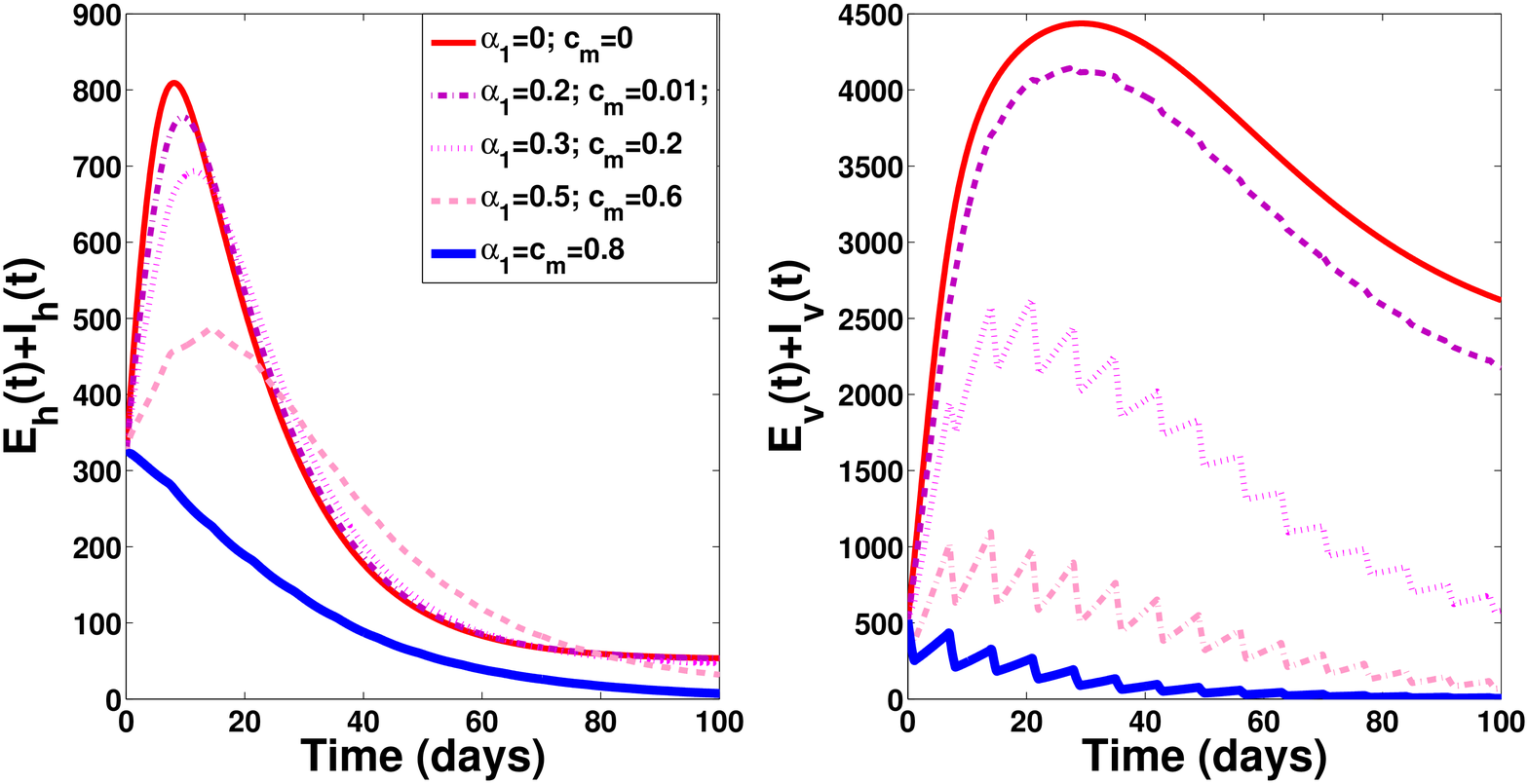}
\includegraphics[width=\textwidth]{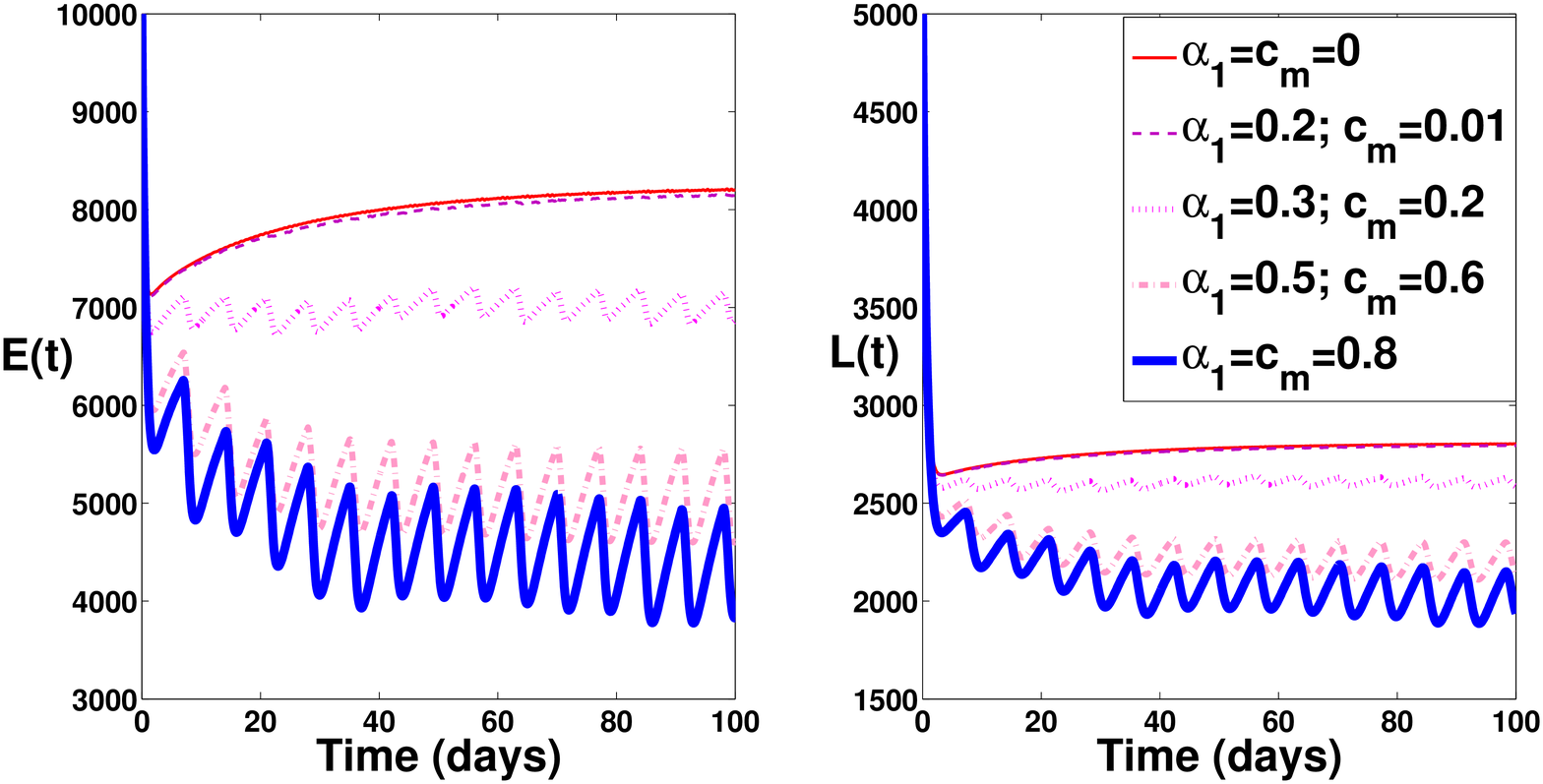}
\caption{Simulations results showing the advantage that we have to combine vaccination, individual protection and adulticide.\label{VAC+alpha1+cm}}
\end{center}
\end{figure}

\subsection{Strategy F: Combining vaccination, individual protection and mechanical control}
Like for strategy E, the combined use of these three types of controls has a positive impact in the vector control.
\begin{figure}[t!]
\begin{center}
\includegraphics[width=\textwidth]{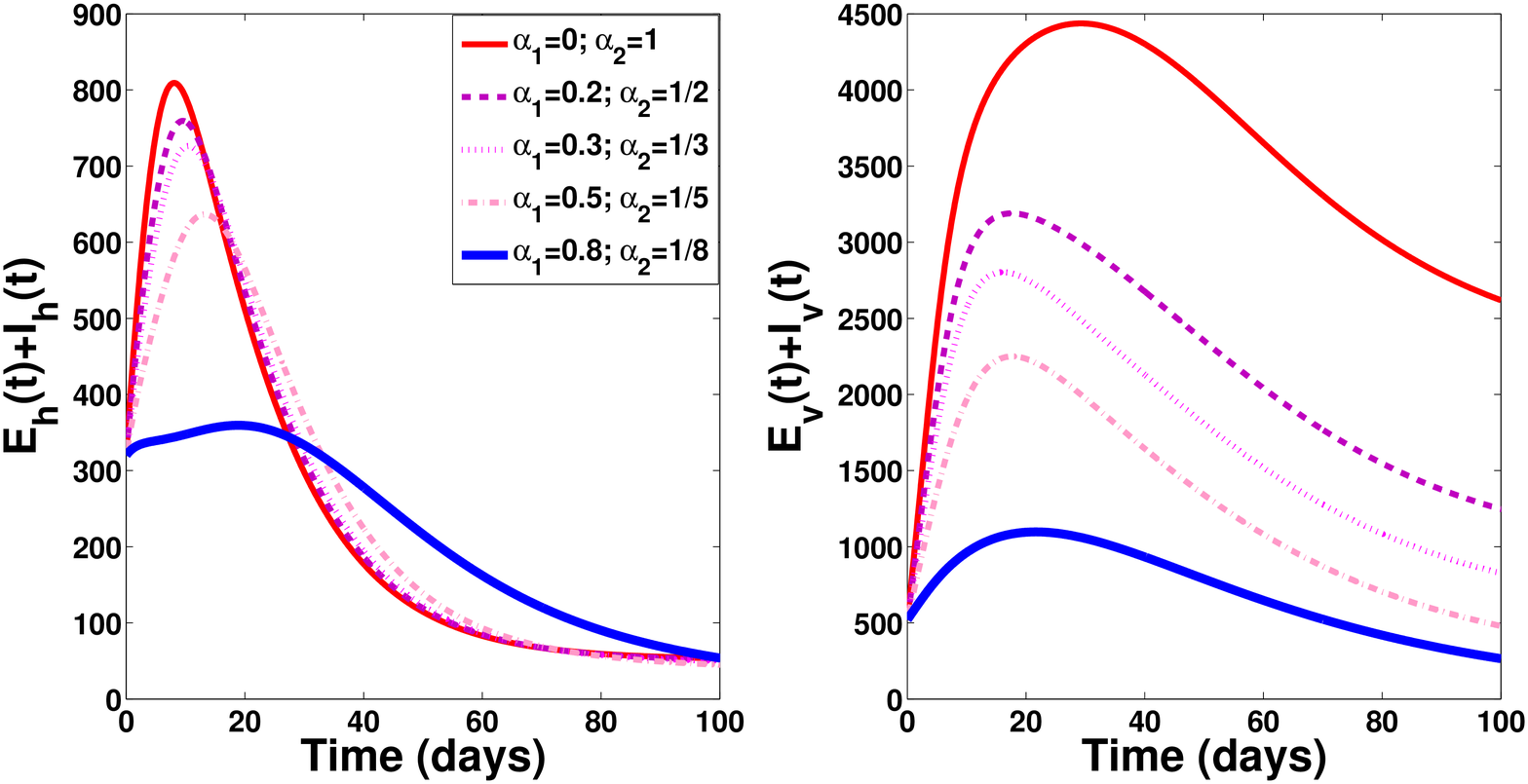}
\includegraphics[width=\textwidth]{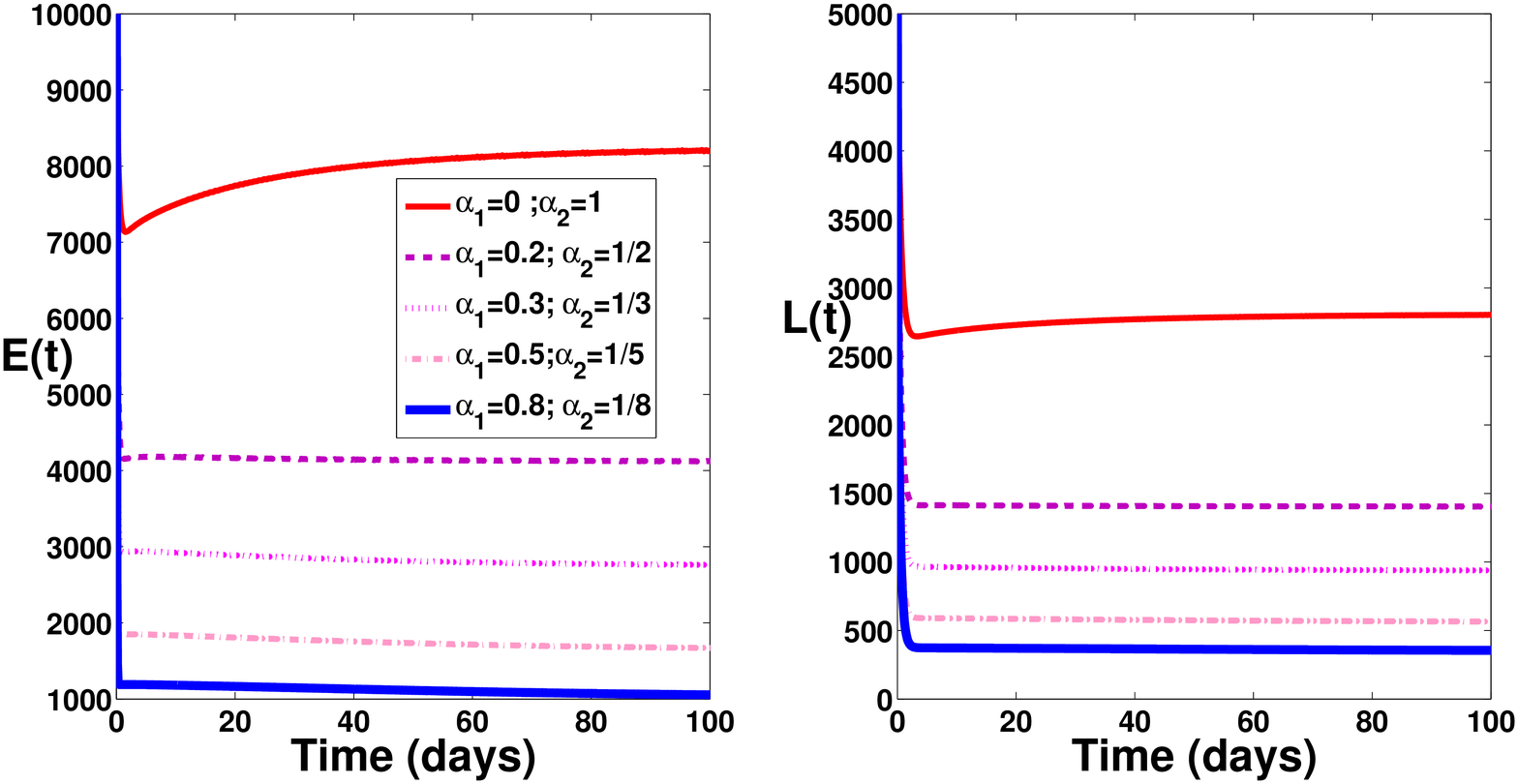}
\caption{Simulations results showing the advantage that we have to combine vaccination, individual protection and mechanical control.\label{VAC+alpha1+alpha2}}
\end{center}
\end{figure}

\newpage

\section{Conclusion}
\label{conclusionAr3opc}
In this paper, we derived and analyzed a deterministic model for the transmission of arboviral diseases with non linear form of infection and complete stage structured model for vectors, which takes into account a vaccination with waning immunity, individual protection and vector control strategies. 

We begin by calculated the net reproductive number $\mathcal{N}$ and the basic reproduction number, $R_0$, and investigated the existence and stability of equilibria. The stability analysis reveals that for $\mathcal{N}\leq 1$, the trivial equilibrium is globally asymptotically stable. When $\mathcal{N}>1$ and $R_0<1$, the disease--free equilibrium is locally asymptotically stable. Under certain condition, the disease--free equilibrium is also globally asymptotically stable.  We found that the model exhibits backward bifurcation. The epidemiological implication of this phenomenon is that for effective eradication and control of diseases, $R_0$ should be less than a critical values less than one. Thus, we proved, that the disease--induced death is the principal cause of the backward bifurcation phenomenon, in the full model and the corresponding model without vaccination.  However, the substitution of standard incidence with mass action incidence removes the backward bifurcation phenomenon. 

We proved that the model admits at least one endemic equilibrium, and only one endemic equilibrium point in the model without disease--induced death, and in the model with mass action incidences, whenever the basic reproduction number is great than unity. 

Using parameters value of Chikungunya and Dengue fever, we calculated the sensitivity indices of the basic reproduction number, $R_0$, to the parameters in the model using both local and global methods. Local sensitivity analysis showed that the model system  is most sensitive to $a$, the average number of mosquitoes bites, followed by $\mu_v$, the natural mortality rate of vectors. Considering that all input parameters vary simultaneously, we use the Latin Hypercube Sampling (LHS) to estimate statistically the mean value of the basic reproduction number. The result showed that the model is in an endemic state, since the mean of $R_0$ is 2.0642, which is greater than unity. Then, using global sensitivity analyisis, we computed the Partial Rank Correlation Coefficients between $R_0$ and each parameter of the model. Unlike the local sensitivity analysis, the global analysis showed that the parameters $\alpha_{1}$, the human protection rate, has the highest influence on $R_0$. The other parameter with an important effect are $\alpha_2$, the efficacity of the mechanical control, $\beta_{hv}$, the probability of transmission of infection from an infectious human to a susceptible vector, $\beta_{vh}$, the probability of transmission of infection from an infectious vector to a susceptible human, and $\theta$, the maturation rate from pupae to adult vectors. This showed that the order of the most important parameters for $R_0$ from the local sensitivity analysis not match those from the global sensitivity analysis. So, the local sensitivity results are not robust.

To assess the impact of combination of different controls, we conduct several simulations, using the called "pulse control" technique. According to the numerical results, we conclude that the use of an imperfect vaccine with low efficiency combined with high individual protection and good vector control statégie (reduction of breeding sites by local populations action, chemical action and use of adulticide), can effectively reduce the transmission of the pathogen and the proliferation of vector populations. However, due to lack of resources to implement these control mechanisms, developing countries should focus on the education of the local populations. Because, unlike diseases such as malaria whose breeding sites of Anopheles mosquitoes are known, those of arboviruses (old tires, flower pots, vases and other hollow...) are smaller and unknown for many local populations, which favor the development of vectors.

Thus, pending the development of a high efficacy vaccine and long-acting, individual protection and the various vector control methods are effective ways to overcome the arboviruses, for developing countries. In addition, the realization of these combination of controls may be too expensive, because it means that, for constant controls, we must keep them at levels  high, and this, for a long time. 
 
\section*{Acknowledgment}
The first author (Hamadjam ABBOUBAKAR) thanks the Direction of UIT of Ngaoundere for their financial assistance in the context of research missions in 2015.
\appendix
\section{Usefull result.\label{Ar3Ap4}}
We use the following result to compute the threshold $R_c$ at Eq. \eqref{R_car3}.
\begin{lemma}[\cite{JCK2008}]
\label{algoJCK2008}
Let M be a square Metzler matrix written in block form
$\left(\begin{array}{cc}
A&B\\
C&D
\end{array} \right) $
with A and D square matrices. M is Metzler stable if and only if matrices $A$ and $D-CA^{-1}B$
are Metzler stable.
\end{lemma}

\section{Proof of Theorem \ref{th3Ar3}.\label{Ar3Ap1}}
The Jacobian matrix of $f$ at the Trivial equilibrium is given by
\begin{equation}
\label{Jacm1Ar3}
Df(\mathcal{E}_{0})=\left( \begin{array}{ccccccccccc}
Df_1&Df_2\\
Df_3&Df_4
\end{array} \right) .
\end{equation} 
where \\
$
Df_1=\left( \begin{array}{cccccc}
-k_1&\omega&0&0&0&0\\
\xi&-k_2&0&0&0&0\\
0&0&-k_3&0&0&0\\
0&0&\gamma_h&-k_4&0&0\\
0&0&0&\sigma&-\mu_h&0\\
0&0&0&0&0&-k_8\\
\end{array} \right),
$
$
Df_3=\left( \begin{array}{ccccccccccc}
0&0&0&0&0&0\\
0&0&0&0&0&0\\
0&0&0&0&0&\mu_b\\
0&0&0&0&0&0\\
0&0&0&0&0&0
\end{array} \right),
$ \\
$
Df_2=\left( \begin{array}{ccccc}
-\dfrac{a(1-\alpha_1)\beta_{hv}\eta_vS^{0}_h}{N^{0}_h}&-\dfrac{a(1-\alpha_1)\beta_{hv}S^{0}_h}{N^{0}_h}&0&0&0\\
-\dfrac{a(1-\alpha_1)\beta_{hv}\pi\eta_vV^{0}_h}{N^{0}_h}&-\dfrac{a(1-\alpha_1)\beta_{hv}\pi V^{0}_h}{N^{0}_h}&0&0&0\\
\dfrac{a(1-\alpha_1)\beta_{hv}\eta_vH^{0}}{N^{0}_h}&\dfrac{a(1-\alpha_1)\beta_{hv}H^{0}}{N^{0}_h}&0&0&0\\
0&0&0&0&0\\
0&0&0&0&0\\
0&0&0&0&\theta\\
\end{array} \right),
$\\
$
Df_4=\left( \begin{array}{ccccccccccc}
-k_9&0&0&0&0\\
\gamma_v&-k_8&0&0&0\\
\mu_b&\mu_b&-k_5&0&0\\
0&0&s&-k_6&0\\
0&0&0&l&-k_7
\end{array} \right),
$
and $H^{0}=S^{0}_h+\pi V^{0}_h$. 

The characteristic polynomial of $Df(\mathcal{E}_{0})$ is given by:
\[
P(\lambda)=-\left(\lambda+k_{3}\right)\left(\lambda+k_{4}\right)\left(
 \lambda+k_{8}\right)\left(\lambda+k_{9}\right)\left(\lambda+\mu
 _{h}\right)\phi_1(\lambda)\phi_2(\lambda)
\]
where\\
$
\phi_1(\lambda)=\lambda^2+(k_{2}+k_{1})\lambda+\mu_h(k_{2}+\xi)
$
and 
$
\phi_2(\lambda)=\lambda^{4}+A_1\lambda^{3}+A_2\lambda^{2}+A_3\lambda+A_4.
$
we have set
\[
\begin{array}{l}
A_1=k_5+k_6+k_7+k_8,\;\;A_2=k_8(k_5+k_6+k_7)+k_7(k_5+k_6)+k_5k_6,\\
A_3=k_5k_6k_7+k_8(k_5k_6+k_7(k_5+k_6)),\;\; A_4=k_5k_6k_7k_8(1-\mathcal{N}).
\end{array}
\]
The roots of $P(\lambda)$ are $\lambda_1=-\mu_h$, $\lambda_1=-k_1$, $\lambda_2=-k_3$, $\lambda_3=-k_4$, $\lambda_4=-k_8$, $\lambda_4=-k_9$, and the others roots are the roots of $\phi_1(\lambda)$ and $\phi_2(\lambda)$. The real part of roots of $\phi_1(\lambda)$  are negative. Since $\mathcal{N}<1$, it is clear that all coefficients of $\phi_2(\lambda)$ are always positive. Now we just have to verify that the Routh--Hurwitz criterion holds for polynomial $\phi_2(\lambda)$.
To this aim, setting 
$H_1=A_1$, $H_2=\begin{vmatrix}
A_1&1\\
A_3&A_2
\end{vmatrix}$,
$H_3=\begin{vmatrix}
A_1&1&0\\
A_3&A_2&A_1\\
0&A_4&A_3
\end{vmatrix}$,
$H_4=\begin{vmatrix}
A_1&1&0&0\\
A_3&A_2&A_1&1\\
0&A_4&A_3&A_2\\
0&0&0&A_4
\end{vmatrix}=A_4H_3$.\\
The Routh-Hurwitz criterion of stability of the trivial equilibrium $\mathcal{E}^{0}$ is given by
\begin{equation}
\label{routh-Hurwith}
\left\lbrace 
\begin{array}{c}
H_1>0\\
H_2>0\\
H_3>0\\
H_4>0
\end{array}\right.\Leftrightarrow
\left\lbrace 
\begin{array}{c}
H_1>0\\
H_2>0\\
H_3>0\\
A_4>0
\end{array}\right.
\end{equation}
We have $H_1=A_1=k_5+k_6+k_7+k_8>0$,
\[
\begin{split}
H_2&=A_1A_2-A_3\\
&=\left(k_{7}+k_{6}+k_{5}\right)k^2_8+\left(k_{7}^2+\left(2k_{6}+2k_{5}\right)k_{7} +k_{6}^2+2k_{5}k_{6}+k_{5}^2\right)k_8\\
&+\left(k_{6}+k_{5}\right)k_{7}^2
+\left(k_{6}^2+2k_{5}k_{6}+k_{5}^2\right)k_{7}+k_{5}k_{6}^2+k_{5}^2k_{6}
\end{split}
\]
{\footnotesize
\[
\begin{split}
H_3&=A_1A_2A_3-A^{2}_1A_4-A^{2}_3\\
&=(k_{6}+k_{5})\left(k_{7}^2+\left(k_{6}+k_{5}\right)k_{7}+k_{5}k_{6}\right)k^3_{8}\\
&+\left(\mu_{b}ls\theta+\left(k_{6}+k_{5}\right)k_{7}^
 3+2(k_{6}+k_{5})^{2}k_{7}^2+\left(
 k_{6}^3+4k_{5}k_{6}^2+4k_{5}^2k_{6}+k_{5}^3\right)k_{7}+k
 _{5}k_{6}^3+2k_{5}^2k_{6}^2+k_{5}^3k_{6}\right)k^2_{8}\\
&+\left[ \left(2k_{7}+2k_{6}+2k_{5}\right)\mu_{b}ls
 \theta+\left(k_{6}^2+2k_{5}k_{6}+k_{5}^2\right)k_{7}^3+
 \left(k_{6}^3+4k_{5}k_{6}^2+4k_{5}^2k_{6}+k_{5}^3\right)k
 _{7}^2\right.\\
&\left.+\left(2k_{5}k_{6}^3+4k_{5}^2k_{6}^2+2k_{5}^3k_{6}
 \right)k_{7}+k_{5}^2k_{6}^3+k_{5}^3k_{6}^2\right]k_8+
 \left(k_{7}^2+\left(2k_{6}+2k_{5}\right)k_{7}+k_{6}^2+2k_{5}
 k_{6}+k_{5}^2\right)\mu_{b}ls\theta\\
 &+\left(k_{5}k_{6}^2+k_{5}^2k_{6}\right)k_{7}^3+\left(k_{5}k_{6}^3+2k_{5}^2k
 _{6}^2+k_{5}^3k_{6}\right)k_{7}^2+\left(k_{5}^2k_{6}^3+k_{5}^3
 k_{6}^2\right]k_{7}
\end{split}
\]
}
We always have $H_1>0$, $H_2>0$, $H_3>0$ and  $H_4>0$ if $\mathcal{N}<1$. Thus, the trivial equilibrium 
$\mathcal{E}_{0}$ is locally asymptotically stable whenever $\mathcal{N}<1$.

We assume the net reproductive number $\mathcal{N}>1$. Following the procedure and the
notation in \cite{vawa02}, we may obtain the basic reproduction number $R_0$ as the dominant eigenvalue of the \emph{next--generation matrix} \cite{dihe,vawa02}. Observe that model (\ref{AR3}) has four infected populations, namely $E_h$, $I_h$, $E_v$, $I_v$. It follows that the matrices $F$ and $V$ defined in \cite{vawa02}, which take into account the new infection terms and remaining transfer terms, respectively, are given by\\
$
F=\left( \begin{array}{cccc}
0&0&\dfrac{a(1-\alpha_1)\beta_{hv}\eta_vH^{0}}{N^{0}_h}&\dfrac{a(1-\alpha_1)\beta_{hv}H^{0}}{N^{0}_h}\\
0&0&0&0\\
\dfrac{a(1-\alpha_1)\beta_{vh}\eta_vS^{0}_v}{N^{0}_h}&\dfrac{a(1-\alpha_1)\beta_{vh}S^{0}_v}{N^{0}_h}&0&0\\
0&0&0&0
\end{array}\right),$\\
$V=\left(\begin{array}{cccc}
k_3&0&0&0\\
-\gamma_h&k_4&0&0\\
0&0&k_9&0\\
0&0&-\gamma_v&k_8
\end{array}\right).
$

The dominant eigenvalue of the next--generation matrix $FV^{-1}$ is given by \eqref{R0ar3}.
The local stability of the disease--free equilibrium $\mathcal{E}_{1}$ is a direct consequence of Theorem 2 of \cite{vawa02}. This ends the proof.\hfill

\section{Proof of Theorem \ref{GAsTEar3}.\label{Ar3Ap2}}
Setting $Y=X-TE$ with $X=(S_h,V_h,E_h,I_h,R_h,S_v,E_v,I_v,E,L,P)^{T}$, $H^{0}=(S^{0}_h+\pi V^{0}_h)$,
$A_{99}=\left(k_5+\mu_b\dfrac{S_v+E_v+I_v}{K_E}\right)$, and $A_{10}=\left(k_6+s\dfrac{E}{K_L}\right)$.
we can rewrite \eqref{AR3} in the following manner
\begin{equation}
\label{model2Ar2} \dfrac{dY}{dt}=\mathcal{B}(Y)Y
\end{equation}
where $ \mathcal{B}(Y)=\left( \begin{array}{cc}
A(Y)&B(Y)\\
C(Y)&D(Y)
\end{array}\right),
$
 with\\
{\small
$ A(Y)=\left( \begin{array}{cccccc}
-(\lambda^{c}_h+k_1)&\omega&0&0&0&0\\
\xi&-(\pi\lambda^{c}_h+k_2)&0&0&0&0\\
\lambda^{c}_h&\pi\lambda^{c}_h&-k_3&0&0&0\\
0&0&\gamma_h&-k_4&0&0\\
0&0&0&\sigma&-\mu_h&0\\
\end{array}\right),
$\\
$ B(Y)=\left( \begin{array}{ccccc}
-\dfrac{a(1-\alpha_1)\beta_{hv}\eta_vS^{0}_h}{N_h}&-\dfrac{a(1-\alpha_1)\beta_{hv}S^{0}_h}{N_h}
&0&0&0\\
-\dfrac{a(1-\alpha_1)\beta_{hv}\eta_v\pi V^{0}_h}{N_h}&-\dfrac{a(1-\alpha_1)\beta_{hv}\pi V^{0}_h}{N_h}&0&0&0\\
\dfrac{a(1-\alpha_1)\beta_{hv}\eta_vH^{0}}{N_h}&\dfrac{a(1-\alpha_1)\beta_{hv}H^{0}}{N_h}&0&0&0\\
0&0&0&0&0\\
0&0&0&0&0\\
\end{array}\right),
$\\
$ C(Y)=\left( \begin{array}{ccccccccccc}
0&0&0&0&0&-(\lambda^{c}_v+k_8)\\
0&0&0&0&0&\lambda^{c}_v\\
0&0&0&0&0&0\\
0&0&0&0&0&\mu_b\\
0&0&0&0&0&0\\
0&0&0&0&0&0\\
\end{array}\right),
$
$ D(Y)=\left( \begin{array}{ccccc}
0&0&0&0&\theta\\
-k_9&0&0&0&0\\
\gamma_v&-k_8&0&0&0\\
\mu_b&\mu_b&-A_{99}&0&0\\
0&0&s&-A_{10}&0\\
0&0&0&l&-k_7\\
\end{array}\right).
$
}

 It is clear that $Y=(0,0,0,0,0,0,0,0,0,0,0)$ is the only equilibrium.

Then it suffices to consider the following Lyapunov function
$\mathcal{L}(Y)=<g,Y>$ were
$g=\left(1,1,1,1,1,1,1,1,\dfrac{k_8}{\mu_b},\dfrac{k_5k_8}{\mu_bs},\dfrac{k_5k_6k_8}{\mu_bsl}\right)$.
Straightforward computations lead that
\[
\begin{split}
\dot{\mathcal{L}}(Y)&=<g,\dot{Y}>\overset{\mathrm{def}}{=}
<g,\mathcal{B}(Y)Y>\\
&=-\mu_hY_1-\mu_hY_2-\mu_hY_3-(\mu_h+\delta)Y_4-\mu_hY_5\\
&-\dfrac{k_8}{K_E}(Y_6+Y_7+Y_8)-\dfrac{k_5k_8}{\mu_bK_L}Y_9Y_{10}+\theta\left(1-\dfrac{1}{\mathcal{N}}\right)Y_{11}
\end{split}
\]
We have $\dot{\mathcal{L}}(Y)<0$ if $\mathcal{N}\leq 1$ and $\dot{\mathcal{L}}(Y)=0$ if $Y_i=0$, $i=1,2,\hdots, 11$ (i.e $S_h=S^{0}_h$, $V_h=V^{0}_h$ and $E_h=I_h=R_h=S_v=E_v=I_v=E=L=P=0$). Moreover, the maximal invariant set contained in $\left\lbrace \mathcal{L}\vert \dot{\mathcal{L}}(Y)=0\right\rbrace $ is $(0,0,0,0,0,0,0,0,0,0,0)$. Thus, from Lyapunov theory, we deduce that $(0,0,0,0,0,0,0,0,0,0,0)$ and thus, $\mathcal{E}_0$, is GAS if and only if $\mathcal{N}\leq 1$.

\section{Proof of Theorem \ref{existenceEEAR3}.\label{Ar3Ap5}}
In order to determine the existence of endemic equilibria, i. e. equilibria with all
positive components, say
\[
\mathcal{E}^{**}=\left(S^{*}_h,V^{*}_h,E^{*}_h,I^{*}_h,R^{*}_h,S^{*}_v,E^{*}_v,I^{*}_v,E,L,P\right),
\]
we have to look for the solution of the algebraic system of equations obtained by equating
the right sides of system \eqref{AR3} to zero. In this way we consider two case:
\paragraph{(i)} Special case: Absence of disease--induced death in human ($\delta=0$)\\
Note that in the absence of disease--induced death in human population, we have $N^{*}_h=N^{0}_h=\Lambda_h/\mu_h$. Let 
\begin{equation}
\label{fieear3}
\lambda^{c,*}_h=\dfrac{a(1-\alpha_1)\beta_{hv}(\eta_vE^{*}_v+I^{*}_v)}{N^{*}_{h}},\;\;\;
\lambda^{c,*}_v=\dfrac{a(1-\alpha_1)\beta_{vh}(\eta_hE^{*}_h+I^{*}_h)}{N^{*}_{h}}
\end{equation}
be the forces of infection of humans and vectors at steady state, respectively. Solving
the equations in \eqref{AR3} at steady state gives
\begin{equation}
\label{EEh}
\begin{array}{l}
S^{*}_h=\dfrac{\Lambda_h(\pi\lambda^{c,*}_h+k_2)}{\mu_h(k_2+\xi)+\lambda^{c,*}_{h}(\pi\lambda^{c,*}_{h}+\pi k_1+k_2)},\;\;\;V^{*}_h=\dfrac{\xi S^{*}_h}{(\pi\lambda^{c,*}_{h}+k_2)},\\
E^{*}_h=\dfrac{\lambda^{c,*}_{h}(S^{*}_h+\pi V^{*}_h)}{k_3},\;\;I^{*}_h=\dfrac{\gamma_h\lambda^{c,*}_{h}(S^{*}_h+\pi V^{*}_h)}{k_3k_4},\;\;R^{*}_h=\dfrac{\sigma\gamma_h\lambda^{c,*}_{h}(S^{*}_h+\pi V^{*}_h)}{\mu_hk_3k_4},
\end{array}
\end{equation}
and
\begin{equation}
\label{EEv}
\begin{array}{l}
S^{*}_v=\dfrac{\theta P}{(\lambda^{c,*}_v+k_8)},\,\;\; E^{*}_v=\dfrac{\theta P\lambda^{c,*}_v}{k_9(\lambda^{c,*}_v+k_8)},\;\;I^{*}_v=\dfrac{\gamma_v\theta P\lambda^{c,*}_v}{k_8k_9(\lambda^{c,*}_v+k_8)},\\
E=\dfrac{\mu_b\theta K_EP}{(k_5k_8K_E+\mu_b\theta P)},\;\;
L=\dfrac{\mu_b\theta sK_EK_LP}{k_6K_L(k_5k_8K_E+\mu_b\theta P)+s\mu_b\theta K_EP},
\end{array}
\end{equation}
where $P$ is solution of the following equation
\begin{equation}
\label{eggs}
f(P)=-k_7P\left[\mu_b\theta(sK_E+k_6K_L)P+k_5k_6k_8K_EK_L(\mathcal{N}-1)\right]=0
\end{equation}
A direct resolution of the above equation give $P=0$ or $P=\dfrac{k_5k_6k_8K_EK_L(\mathcal{N}-1)}{\mu_b\theta(sK_E+k_6K_L)}$.

Note that $P=0$ corresponds to the trivial equilibrium $\mathcal{E}_{0}$. Now we consider $P>0$ i.e. $\mathcal{N}>1$. Replacing \eqref{EEh} and \eqref{EEv} in \eqref{fieear3} give
\begin{equation}
\label{fiheear3}
\begin{split}
\lambda^{c,*}_h&=\dfrac{a(1-\alpha_1)\beta_{hv}\mu_h}{\Lambda_h}\left(\eta_v\dfrac{\theta P\lambda^{*}_v}{k_9(\lambda^{*}_v+k_8)}+\dfrac{\gamma_v\theta P\lambda^{*}_v}{k_8k_9(\lambda^{*}_v+k_8)}\right) \\
\end{split}
\end{equation}
\begin{equation}
\label{fiveear3}
\begin{split}
\lambda^{c,*}_v&=\dfrac{a(1-\alpha_1)\beta_{vh}\mu_h}{\Lambda_{h}}\left(\eta_h\dfrac{\lambda^{*}_{h}(S^{*}_h+\pi V^{*}_h)}{k_3}+\dfrac{\gamma_h\lambda^{*}_{h}(S^{*}_h+\pi V^{*}_h)}{k_3k_4}\right) \\
\end{split}
\end{equation}
Substuting \eqref{fiveear3} in \eqref{fiheear3} give
\begin{equation}
\label{polyeeAr3}
\left(k_{6}K_{L}+sK_{E}\right)\lambda^{*}_h\left[ a_2(\lambda^{*}_h)^2+a_1\lambda^{*}_h+a_0\right]=0
\end{equation}
where $a_2$, $a_1$ and $a_0$ are given by
\begin{equation}
\label{varintereear3}
\begin{array}{l}
R_b=\dfrac{(\pi\xi+k_2)}{\pi(\xi+k_2)}\left( \dfrac{(k_{1}\pi+k_{2})}{\mu_{h}}+\dfrac{a(1-\alpha_1)\beta_{vh}(\gamma_{h}+k_{4}\eta_{h})(\pi\xi+k_2)}{k_{3}k_{4}k_{8}}\right) ,\\
a_2=\left(a(1-\alpha_1)\beta_{vh}\mu_{h}(\gamma_{h}+k_{4}\eta_{h})+k_{3}k_{4}k_{8}\right)k_{9}\mu_{b}\Lambda_{h}\pi,\\
a_1=\dfrac{k_{3}k_{4}k_{8}k_{9}\mu_{b}\Lambda_{h}\left(\xi+k_{2}\right)\mu_{h}\pi}
{\left(\pi\xi+k_{2}\right)}(R_b-R_1),\\
a_0=\mu_hk_{3}k_{4}k_{8}k_{9}\mu_{b}\Lambda_{h}(\xi+k_{2})\left(1-R_1\right).
\end{array}
\end{equation}
The trivial solution $\lambda^{*}_h=0$ of \eqref{polyeeAr3} corresponds to the disease--free equilibrium $\mathcal{E}_{1}$. Now, we just look the equilibria when $\lambda^{*}_h>0$. Note that coefficient $a_2$ is always positive and $a_0$ is less (resp. greather) than unity if and only if $R_1>1$ (resp.  $R_1<1$). Thus model system \eqref{AR3}, in absence of disease--induced death in human population ($\delta=0$), admits only one endemic equilibrium whenever $R_0>1$. Since the sign of coefficient $a_1$ depend of the value of parameter, we investigate the possibility of occurence of backward bifurcation phenomenon when $R_0<1$. 
Furthermore, consider the inequality 
\begin{equation}
\label{condi1ar3}
R_1\leq R_{b}.
\end{equation}
Since $a_2$ is always positive and $a_0$ is always positive whenever $R_0<1$, then, the occurence of backward bifurcation phenomenon depend of the sign of coefficient $a_1$. The coefficent $a_1$ is always positive if and only if condition \eqref{condi1ar3} holds (i.e $R_1<R_b$). It follows that the disease--free equilibrium is the unique equilibrium when $\mathcal{N}>1$ and $R_0<1$. Now if $R_b<R_1<1$, then in addition to the DFE  $\mathcal{E}_1$, there exists two endemic equilibria whenever $\Delta=a^{2}_1-4a_2a_0>0$. However, $R_b<R_1<1\Rightarrow R_b<1\Leftrightarrow \beta_{vh}<- \dfrac{\left[\pi^2\xi^2+\left(\mu_{h}\pi^2+\left(2\omega+\mu_{h}\right)\pi\right)\xi
+(\omega+\mu_{h})^2\right]k_{3}k_{4}k_{8}}{a(1-\alpha_1)\mu_{h}(\pi\xi+k_2)^{2}(\gamma_{h}+k_{4}\eta_{h})} <0$. Since all parameter of model \eqref{AR3} are nonnegative, we conclude that the condition $R_b<R_1<1$ does not hold.  And thus, the backward bifurcation never occurs in the absence of disease--induced death in human.
\paragraph{(ii)} Presence of disease induced death in human ($\delta\neq 0$).
In this case, we have $N^{*}_h=\dfrac{\Lambda_h-\delta I^{*}_h}{\mu_h}$. Applying the same procedure as case  (i), we obtain that $\lambda^{*}_h$ at steady state is solution of the following equation
\begin{equation}
\label{eefullAr3}
f(\lambda^{*}_h)=\lambda^{*}_h\left[c_4(\lambda^{*}_h)^{4}+c_3(\lambda^{*}_h)^{3}+c_2(\lambda^{*}_h)^{2}
+c_1\lambda^{*}_h+c_0 \right]=0,
\end{equation} 
where
\[
\begin{split}
c_4&=-\pi^2k_{9}K_{12}\mu_{b}\Lambda_{h}\left(k_{3}k_{4}-\delta\gamma_{h}\right)
\left(k_{10}a\mu_{h}(1-\alpha_1)\beta_{vh}
+k_{8}(k_{3}k_{4}-\delta\gamma_{h})\right),
\end{split}
\]
{\footnotesize
\[
\begin{split}
c_3&=\pi(k_{3}k_{4}k_{5}k_{6}k_{10}k_{11}a^2\mu_{h}^2(1-\alpha_1)^{2}\beta_{hv}n\pi\beta_{vh}K_{E}K_{L}
+2k_{9}k_{10}K_{12}a\mu_{b}\delta\Lambda_{h}\mu_{h}\gamma_{h}\pi(1-\alpha_1)\beta_{vh}\xi\\
&-k_{3}k_{4}k_{9}k_{10}K_{12}a\mu_{b}\Lambda_{h}\mu_{h}\pi(1-\alpha_1)\beta_{vh}\xi
-2k_{8}k_{9}K_{12}\mu_{b}\delta^2\Lambda_{h}\gamma_{h}^2\pi\xi
+2k_{3}k_{4}k_{8}k_{9}K_{12}\mu_{b}\delta\Lambda_{h}\gamma_{h}\pi\xi\\
&-k_{1}k_{3}k_{4}k_{9}k_{10}K_{12}a\mu_{b}\Lambda_{h}\mu_{h}\pi(1-\alpha_1)\beta_{vh}+2k_{2}k_{9}k_{10}
K_{12}a\mu_{b}\delta\Lambda_{h}\mu_{h}\gamma_{h}(1-\alpha_1)\beta_{vh}\\
& -2k_{2}k_{3}k_{4}k_{9}k_{10}K_{12}a\mu_{b}\Lambda_{h}\mu_{h}(1-\alpha_1)\beta_{vh}
+2k_{1}k_{3}k_{4}k_{8}k_{9}K_{12}\mu_{b}\delta\Lambda_{h}\gamma_{h}\pi-2k_{1} k_{3}^2k_{4}^2k_{8}k_{9}K_{12}\mu_{b}\Lambda_{h}\pi\\
&-2k_{2}k_{8}k_{9}K_{12}\mu_{b}\delta^2\Lambda_{h}\gamma_{h}^2
+4k_{2}k_{3}k_{4}k_{8}k_{9}K_{12}\mu_{b}\delta\Lambda_{h}\gamma_{h}-2k_{2}k_{3}^2k_{4}^2k_{8}k_{9}K_{12}\mu_{b}\Lambda_{h}),
\end{split}
\]
}
{\scriptsize
\[
\begin{split}
c_2&=k_{3}k_{4}k_{5}k_{6}k_{10}k_{11}a^2\mu_{h}^2(1-\alpha_1)^{2}\beta_{hv}n\pi^2\beta_{vh}\xi K_{E}K_{L}\\ &+k_{1}k_{3}k_{4}k_{5}k_{6}k_{10}k_{11}a^2\mu_{h}^2(1-\alpha_1)^{2}\beta_{hv}n\pi^2\beta_{vh}K_{E}K_{L}\\
&+2k_{2}k_{3}k_{4}k_{5}k_{6}k_{10}k_{11}a^2\mu_{h}^2(1-\alpha_1)^{2}\beta_{hv}n\pi\beta_{vh}K_{E}K_{L}
+k_{9}k_{10}K_{12}a\mu_{b}\delta\Lambda_{h}\mu_{h}\gamma_{h}\pi^2(1-\alpha_1)\beta_{vh}\xi^2\\
&-k_{8}k_{9}K_{12}\mu_{b}\delta^2\Lambda_{h}\gamma_{h}^2\pi^2\xi^2
-k_{1}k_{3}k_{4}k_{9}k_{10}K_{12}a\mu_{b}\Lambda_{h}\mu_{h}\pi^2(1-\alpha_1)\beta_{vh}\xi\\
&+k_{3}k_{4}k_{9}
 k_{10}K_{12}a\mu_{b}\Lambda_{h}\mu_{h}\omega\pi (1-\alpha_1)\beta_{vh}\xi
+2k_{2}k_{9}k_{10}K_{12}a\mu_{b}\delta\Lambda_{h}\mu_{h}\gamma_{h}\pi(1-\alpha_1)\beta_{vh}\xi\\
&-k_{2}k_{3}k_{4}k_{9}k_{10}K_{12}a\mu_{b}\Lambda_{h}\mu_{h}\pi(1-\alpha_1)\beta_{vh}\xi
+2k_{1}k_{3}k_{4}k_{8}k_{9}K_{12}\mu_{b}\delta\Lambda_{h}\gamma_{h}\pi^2\xi\\
&-2k_{3}k_{4}k_{8}k_{9}K_{12}\mu_{b}\delta\Lambda_{h}\gamma_{h}\omega\pi\xi
+2k_{3}^2k_{4}^2k_{8}k_{9}K_{12}\mu_{b}\Lambda_{h}\omega\pi\xi
-2k_{2}k_{8}k_{9}K_{12}\mu_{b}\delta^2\Lambda_{h}\gamma_{h}^2\pi\xi\\
&+2k_{2}k_{3}k_{4}k_{8}k_{9}K_{12}\mu_{b}\delta\Lambda_{h}\gamma_{h}\pi\xi
-2k_{1}k_{2}k_{3}k_{4}k_{9}k_{10}K_{12}a\mu_{b}\Lambda_{h}\mu_{h}(1-\alpha_1)\pi\beta_{vh}\\
&+k_{2}^2k_{9}k_{10}K_{12}a\mu_{b}\delta\Lambda_{h}\mu_{h}\gamma_{h}(1-\alpha_1)\beta_{vh}
-k_{2}^2k_{3}k_{4}k_{9}k_{10}K_{12}a\mu_{b}\Lambda_{h}\mu_{h}(1-\alpha_1)\beta_{vh}\\
&-k_{1}^2k_{3}^2k_{4}^2k_{8}k_{9}K_{12}\mu_{b}\Lambda_{h}\pi^2 +4k_{1}k_{2}k_{3}k_{4}k_{8}k_{9}K_{12}\mu_{b}\delta\Lambda_{h}
 \gamma_{h}\pi-4k_{1}k_{2}k_{3}^2k_{4}^2k_{8}k_{9}K_{12}\mu_{b}\Lambda_{h}\pi\\
&-k_{2}^2k_{8}k_{9}K_{12}\mu_{b}\delta^2\Lambda_{h}\gamma_{h}^2 +2k_{2}^2k_{3}k_{4}k_{8}k_{9}K_{12}\mu_{b}\delta\Lambda_{h}\gamma_{h}
 -k_{2}^2k_{3}^2k_{4}^2k_{8}k_{9}K_{12}\mu_{b}\Lambda_{h},
\end{split}
\]
}
{\footnotesize
\[
\begin{split}
c_1&=((k_{1}k_{3}k_{4}k_{5}k_{6}k_{10}k_{11}a^2\mu_{h}^2(1-\alpha_1)\beta_{hv}n\pi^2 +k_{3}k_{4}k_{5}k_{6}k_{10}k_{11}a^2\mu_{h}^2(1-\alpha_1)^2\beta_{hv}n(k_{2}-\omega)\pi)\beta_{vh}\xi\\
&+(2k_{1}k_{2}k_{3}k_{4}k_{5}k_{6}k_{10}k_{11}a^2\mu_{h}^2(1-\alpha_1)\beta_{hv}n\pi +k_{2}^2k_{3}k_{4}k_{5}k_{6}k_{10}k_{11}a^2\mu_{h}^2(1-\alpha_1)\beta_{hv}n)(1-\alpha_1)\beta_{vh})K_{E}K_{L}\\
&+(k_{3}k_{4}k_{9}k_{10}K_{12}a\mu_{b}\Lambda_{h}\mu_{h}\omega\pi(1-\alpha_1)\beta_{vh}
 -2k_{3}k_{4}k_{8}k_{9}K_{12}\mu_{b}\delta\Lambda_{h}\gamma_{h}\omega\pi)\xi^2\\
& +((k_{2}k_{3}k_{4}k_{9}k_{10}K_{12}a\mu_{b}\Lambda_{h}\mu_{h}\omega-k_{1}k_{2}k_{3}
 k_{4}k_{9}k_{10}K_{12}a\mu_{b}\Lambda_{h}\mu_{h}\pi)(1-\alpha_1)\beta_{vh}\\
& +(2k_{1}k_{3}^2k_{4}^2k_{8}k_{9}K_{12}\mu_{b}\Lambda_{h}\omega+2k_{1}k_{2}k_{3}k
 _{4}k_{8}k_{9}K_{12}\mu_{b}\delta\Lambda_{h}\gamma_{h})\pi\\
& +(2k_{2}k_{3}^2k_{4}^2k_{8}k_{9}K_{12}\mu_{b}\Lambda_{h}-2k_{2}k_{3}k_{4}k_{8}k_{9}K_{12}
 \mu_{b}\delta\Lambda_{h}\gamma_{h})\omega)\xi\\
&-k_{1}k_{2}^2k_{3}k_{4}k_{9}k_{10}K_{12}a\mu_{b}\Lambda_{h}\mu_{h}(1-\alpha_1)\beta_{vh} -2k_{1}^2k_{2}k_{3}^2k_{4}^2 k_{8}k_{9}K_{12}\mu_{b}\Lambda_{h}\pi\\
&+2k_{1}k_{2}^2k_{3}k_{4}k_{8}k_{9}K_{12}\mu_{b}\delta\Lambda_{h}\gamma_{h} -2k_{1}k_{2}^2k_{3}^2k_{4}^2k_{8}k_{9}K_{12}\mu_{b}\Lambda_{h},
\end{split}
\]
}

\[
\begin{split}
c_0&=k^{2}_{3}k^{2}_{4}k_{8}k_{9}K_{12}\mu_{b}\Lambda_{h}\mu^{2}_h(k_{2}+\xi)^{2}
\left(R^{2}_0-1\right),
\end{split}
\]
with $k_{10}=\gamma_h+\eta_hk_4$, $k_{11}=\gamma_v+\eta_vk_8$, $K_{12}=(sK_E+k_6K_L)$ and $n=\mathcal{N}-1$.
 Notes that $c_4$ is always negative and $c_0$ is positive (resp. negative) if $R_0$ is greather (resp. less) that the unity. It follows, depending of the sign of coefficients $c_3$, $c_2$ and $c_1$, that the model system \eqref{AR3} admits at least one endemic equilibrium whenever $R_0>1$ and the phenomenon of backward (resp. forward) bifurcation can occurs when $R_0<1$ (resp. $R_0>1$).  This ends the proof.\hfill


\bibliography{mybibfile}
\end{document}